\documentclass[12pt]{amsart}
\usepackage{a4wide}
\usepackage{amssymb}
\usepackage{amsthm}
\usepackage{amsmath}
\usepackage{amscd}
\usepackage{verbatim}
\usepackage{bm}
\usepackage[all]{xy}
\usepackage[style=alphabetic, backend=bibtex, doi=false, url=false, giveninits=true, eprint=true, sorting=nyt, maxnames=6, isbn=false]{biblatex}
\usepackage{hyperref}

\numberwithin{equation}{section}

\theoremstyle{plain}
\newtheorem{theorem}{Theorem}[section]
\newtheorem{corollary}[theorem]{Corollary}
\newtheorem{lemma}[theorem]{Lemma}
\newtheorem{proposition}[theorem]{Proposition}

\newtheorem{conjecture}[theorem]{Conjecture}
\theoremstyle{definition}
\newtheorem{definition}[theorem]{Definition}
\newtheorem{remark}[theorem]{Remark}

\theoremstyle{remark}

\newcommand{\OO}{\mathcal O}
\newcommand{\A}{\mathbb{A}}
\newcommand{\R}{\mathbb{R}}
\newcommand{\G}{\mathbb{G}}
\newcommand{\Q}{\mathbb{Q}}
\newcommand{\Z}{\mathbb{Z}}
\newcommand{\N}{\mathbb{N}}
\newcommand{\C}{\mathbb{C}}
\newcommand{\h}{\mathbb{H}}
\renewcommand{\H}{\mathbb{H}}
\newcommand{\F}{\mathbb{F}}
\newcommand{\D}{\mathbb{D}}

\newcommand{\zxz}[4]{\begin{pmatrix} #1 & #2 \\ #3 & #4 \end{pmatrix}}
\newcommand{\abcd}{\zxz{a}{b}{c}{d}}

\newcommand{\kzxz}[4]{\left(\begin{smallmatrix} #1 & #2 \\ #3 & #4\end{smallmatrix}\right) }
\newcommand{\kabcd}{\kzxz{a}{b}{c}{d}}

\newcommand{\calC}{\mathcal{C}}
\newcommand{\calD}{\mathcal{D}}
\newcommand{\calE}{\mathcal{E}}
\newcommand{\calF}{\mathcal{F}}
\newcommand{\calG}{\mathcal{G}}

\newcommand{\calJ}{\mathcal{J}}

\newcommand{\calL}{\mathcal{L}}
\newcommand{\calM}{\mathcal{M}}

\newcommand{\calO}{\mathcal{O}}

\newcommand{\calQ}{\mathcal{Q}}

\newcommand{\calX}{\mathcal{X}}
\newcommand{\calZ}{\mathcal{Z}}

\newcommand{\fraka}{\mathfrak a}

\newcommand{\frakn}{\mathfrak n}
\newcommand{\frakp}{\mathfrak p}

\newcommand{\lev}{M}
\newcommand{\za}{\mathrm{Za}}
\newcommand{\shim}{\mathrm{Sh}}
\newcommand{\Res}{\operatorname{Res}}

\newcommand{\bs}{\backslash}
\newcommand{\norm}{\operatorname{N}}

\newcommand{\vol}{\operatorname{vol}}
\newcommand{\tr}{\operatorname{tr}}

\newcommand{\sgn}{\operatorname{sgn}}

\newcommand{\Cl}{\operatorname{Cl}}
\newcommand{\Sl}{\operatorname{SL}}
\newcommand{\Gl}{\operatorname{GL}}

\newcommand{\GSpin}{\operatorname{GSpin}}
\newcommand{\CT}{\operatorname{CT}}
\newcommand{\Mp}{\operatorname{Mp}}

\newcommand{\Aut}{\operatorname{Aut}}

\newcommand{\Mat}{\operatorname{Mat}}

\newcommand{\End}{\operatorname{End}}

\newcommand{\sig}{\operatorname{sig}}

\newcommand{\Pet}{\text{\rm Pet}}

\newcommand{\GL}{\operatorname{GL}}
\newcommand{\SO}{\operatorname{SO}}

\newcommand{\Gal}{\operatorname{Gal}}

\newcommand{\reg}{\text{\rm reg}}
\newcommand{\supp}{\operatorname{supp}}
\newcommand{\J}{\calJ^{\text{\rm \scriptsize cusp}}}

\newcommand{\z}{\operatorname{Z}}

\newcommand{\diag}{\operatorname{diag}}

\newcommand{\Gspin}{\operatorname{GSpin}}
\newcommand{\boldbeta}{\text{\boldmath$\beta$\unboldmath}}

\newcommand{\SL}{\operatorname{SL}}

\DeclareMathOperator{\Norm}{N}
\newcommand{\legendre}[2]{\left( \frac{\mathstrut #1}{#2} \right)} 
\renewcommand{\k}{k}
\newcommand{\Hilb}{H}
\newcommand{\Rhilb}{R}
\newcommand{\sage}{\texttt{sage}\ }
\DeclareMathOperator{\ab}{ab}

\DeclareMathOperator{\lcm}{lcm}

\begin{document}

\title[CM values of higher automorphic Green functions]{
CM values of higher automorphic Green functions for orthogonal groups
}

\author[J.~H.~Bruinier]{Jan Hendrik Bruinier}
\author[S.~Ehlen]{Stephan Ehlen}
\author[T.~Yang]{Tonghai Yang}

\address{Fachbereich Mathematik,
Technische Universit\"at Darmstadt, Schlossgartenstrasse 7, D--64289
Darmstadt, Germany}
\email{bruinier@mathematik.tu-darmstadt.de}

\address{Mathematisches Institut, University of Cologne, Weyertal 86-90, D-50931 Cologne, Germany}
\email{sehlen@math.uni-koeln.de}

\address{Department of Mathematics, University of Wisconsin Madison, Van Vleck Hall, Madison, WI 53706, USA}
\email{thyang@math.wisc.edu}

\thanks{The first author is partially supported by DFG grant BR-2163/4-2 and the LOEWE research unit USAG.
The third author is partially supported by  NSF grant DMS-1762289.}

\subjclass[2010]{11G18, 11G15, 11F37}


\begin{abstract}
Gross and Zagier conjectured that the CM values (of certain Hecke translates) of the automorphic Green function $G_s(z_1,z_2)$ for
the elliptic modular group at positive integral spectral parameter $s$ are given by logarithms of algebraic numbers in suitable class fields.
We prove a partial average version of this conjecture, where we sum in the first variable $z_1$ over all CM points of a fixed discriminant $d_1$ (twisted by a genus character), and allow in the second variable the evaluation at individual CM points of discriminant $d_2$. This result is deduced from more general statements for
automorphic Green functions on Shimura varieties associated with the group $\GSpin(n,2)$. We also use our approach to prove a Gross-Kohnen-Zagier theorem for higher Heegner divisors on Kuga-Sato varieties over modular curves.
\end{abstract}

\maketitle


\section{Introduction}


The automorphic Green function for $\Gamma=\SL_2(\Z)$, also called the resolvent kernel function for $\Gamma$, plays an important role in the theory of automorphic forms, see e.g.~\cite{Fay}, \cite{Hejhal}. It can be defined as
the infinite series
\[
G_s(z_1,z_2)= -2\sum_{\gamma\in \Gamma} Q_{s-1}\left(1+\frac{|z_1-\gamma z_2|^2}{2\Im (z_1) \Im (\gamma z_2)}\right),
\]
where $Q_{s-1}(t)= \int_0^\infty (t+\sqrt{t^2-1}\cosh(u))^{-s}\, du$ denotes the classical Legendre function of the second kind.
The sum converges absolutely for $s\in \C$ with $\Re(s)>1$, and $z_1,z_2$ in the complex upper half-plane $\H$ with $z_1\notin \Gamma z_2$.
Hence $G_s$ is invariant under the action of $\Gamma$ in both variables and descends to a function on $(X\times X)\setminus Z(1)$, where $X=\Gamma\bs \H$ and $Z(1)$ denotes the diagonal. Along $Z(1)$ it has a logarithmic singularity.
The differential equation of the Legendre function implies that $G_s$ is an eigenfunction of the hyperbolic Laplacian in both variables.
It
has a meromorphic continuation in $s$ to the whole complex plane and satisfies a functional equation relating the values at $s$ and $1-s$.



\subsection{The algebraicity conjecture}


Gross and Zagier employed the automorphic Green function in their celebrated work on canonical heights of Heegner points on modular curves to compute  archimedian height pairings of Heegner points \cite{GZ,GKZ}.
They also used it to derive explicit formulas for the norms of singular moduli, that is, for the CM values of the classical $j$-invariant. More precisely they computed the norms of the values
of $j(z_1) - j(z_2)$ at a pair of CM points $z_1$ and $z_2$, by giving a formula for the prime factorization. The main point of their analytic proof of this   result is that $\log| j(z_1)-j(z_2)|$ is essentially given by the constant term in the Laurent expansion at $s=1$ of
$G_s(z_1,z_2)$.

Gross and Zagier also studied the CM values of the automorphic Green function at positive integral spectral parameter $s=1+j$ for $j\in \Z_{>0}$ and conjectured that these quantities should have striking arithmetic properties,  which resemble those of singular moduli (see Conjecture~4.4 in \cite[Chapter 5.4]{GZ}, \cite[Chapter 5.1]{GKZ}, \cite{Vi1}). To describe their conjecture, let
\begin{align}
\label{eq:greenintro}
G_s^m(z_1,z_2) = G_s(z_1,z_2)\mid T_m =
-2\sum_{\substack{\gamma\in \Mat_2(\Z)\\ \det(\gamma)=m}} Q_{s-1}\left(1+\frac{|z_1-\gamma z_2|^2}{2\Im (z_1) \Im (\gamma z_2)}\right)
\end{align}
be the translate of $G_s$ by the $m$-th Hecke operator $T_m$, acting on any of the two variables.
Fix a weakly holomorphic modular form $f=\sum_m c_f(m) q^m\in M_{-2j}^!$ of weight $-2j$ for $\Gamma$, and put
\begin{equation}\label{eq:Gkf}
  G_{1+j,f}(z_1,z_2)= \sum_{m>0} c_f(-m) m^j G_{j+1}^m(z_1,z_2).
\end{equation}
For a discriminant $d<0$ we write $\calO_{d}$ for the order of discriminant $d$ in the imaginary quadratic field $\Q(\sqrt{d})$, and let $H_d$ be the corresponding ring class field, which we view as a subfield of the complex numbers $\C$.

\begin{conjecture}[Gross--Zagier]\label{conj:alg}
	Assume that $c_f(m) \in \Z$ for all $m<0$.
	Let $z_1$ be a CM point of discriminant $d_1$, and let $z_2$ be a
	CM point of discriminant $d_2$
	such that $(z_1, z_2)$ is not contained in
	$Z^j(f)= \sum_{m >0} c_f(-m)m^j Z(m)$,
where $Z(m)$ is the $m$-th Hecke correspondence on $X\times X$.
	Then there is an $\alpha \in H_{d_1}\cdot H_{d_2}$ such that
	\begin{align}
	\label{eq:gzprec}
	  (d_1 d_2)^{j/2} G_{j+1, f} (z_1, z_2) = \frac{w_{d_1}w_{d_2}}{4}\cdot \log|\alpha|,
	\end{align}
	where $w_{d_i}=\#\calO_{d_i}^\times$.
\end{conjecture}



Gross, Kohnen, and Zagier proved an average version of the conjecture which roughly says that the sum of $(d_1 d_2)^{j/2} G_{j+1, f} (z_1, z_2)$ over all CM points $(z_1,z_2)$ of discriminants $d_1$ and $d_2$ is equal to $\log|\beta|$ for some $\beta\in \Q$. Moreover, they provided numerical evidence in several cases \cite[Chapter V.4]{GZ}, \cite[Chapter V.1]{GKZ}.
Mellit proved the conjecture for $z_2=i$ and $j=1$ \cite{Me}. For a pair of CM points that lie in the same imaginary quadratic field, the conjecture would follow from the work of Zhang on the higher weight Gross-Zagier formula \cite{Zh1}, provided that a certain height pairing of Heegner cycles on Kuga-Sato varieties is non-degenerate.
Viazovska showed in this case that \eqref{eq:gzprec} holds for some $\alpha\in \bar\Q$
and the full conjecture assuming that $d_1 = d_2$ is prime \cite{Vi1, Vi3}.
Recently, Li proved another average version of the conjecture for odd $j$ \cite{Li}. When $d_1$ and $d_2$ are coprime fundamental discriminants, he showed that the average over the $\Gal(\bar \Q/F)$-orbit of the CM point $(z_1,z_2)$ is given by the  logarithm of an algebraic number in $F=\Q(\sqrt{d_1d_2})$.

In the present paper we prove stronger results, by only averaging over the CM points $z_1$ of {\em one} discriminant $d_1$ and allowing for $z_2$ individual CM points of discriminant $d_2$. Let $\calQ_{d_1}$ denote the set of integral binary positive definite quadratic forms of discriminant $d_1<0$.
The group $\Gamma$ acts on $\calQ_{d_1}$ with finitely many orbits.
For $Q\in \calQ_{d_1}$ we write $z_Q$ for the corresponding CM point, i.e., the unique root of $Q(z,1)$ in $\H$, and we let $w_Q$ be the order of the stabilizer $\Gamma_Q$.
The divisor
\[
C(d_1) = \sum_{Q\in \calQ_{d_1}/\Gamma} \frac{2}{w_Q} \cdot z_Q
\]
on $X$ is defined over $\Q$.
The Galois group $\Gal(H_{d_1}/\Q)$ of the ring class field $H_{d_1}$ 
acts on the points in the support of $C(d_1)$ by the theory of complex multiplication.

\begin{theorem}
\label{cor:partial-average-algebraic-general-intro}
Let $j\in \Z_{>0}$. Let $d_1<0$ be a fundamental discriminant, and let $d_2<0$ be a discriminant such that $d_1d_2$ is not the square of an integer.
If $j$ is odd, let $k = \Q(\sqrt{d_1}, \sqrt{d_2})$ and $H = H_{d_2}(\sqrt{d_1})$.
If $j$ is even, let $k = \Q(\sqrt{d_2})$ and $H = H_{d_2}$.
If $z_2$ is a CM point of discriminant $d_2$, then there exists an algebraic number $\alpha=\alpha(f,d_1,z_2)\in H$
and an $r \in \Z_{>0}$ such that
\[
     (d_1d_2)^{j/2} G_{j+1,f}( C(d_1), z_2^\sigma) =
                \frac{1}{r} \log| \alpha^\sigma |
\]
for every $\sigma \in \Gal(H/k)$.
\end{theorem}

\begin{remark}
   If $j$ is even, then $r$ depends only on $d_2$ but not on $f,d_1,z_2$.
   If $j$ is odd, then $r$ may depend on $d_1$ and $d_2$, but not on $f$ or $z_2$.
   The two cases require slightly different proofs, which explains the differences in the results.
   We refer to Section \ref{sec:partial-averages} for details.
\end{remark}


In the main text we will actually consider twists of the divisors $C(d_1)$ by genus characters, and corresponding twisted versions of the above theorem (see Corollary \ref{cor:partial-average-algebraic-general}).
As a corollary we obtain the following result.

\begin{corollary}\label{cor:exponent2-intro}
Let $d_1< 0$ be a fundamental discriminant
  and assume that the class group of $\calO_{d_1}$ is trivial or
  has exponent $2$.
  Let $z_1$ be any
  CM point of of discriminant
  $d_1$ and let $z_2$ be any CM point of
	discriminant $d_2 < 0$ (not necessarily fundamental), where $z_1 \neq z_2$ if $d_1 = d_2$.
  Then, there is an $\alpha \in H_{d_1}\cdot H_{d_2}$ and an $r\in \Z_{>0}$ such that
   \[
     (d_1 d_2)^{j/2} G_{j+1, f} (z_1, z_2) = \frac{1}{r}\log |\alpha|.
   \]
\end{corollary}

\begin{remark}
  Chowla \cite{Chowla} showed that there exist only finitely many imaginary quadratic number fields of discriminant $d_1$ such that the class group of $\Q(\sqrt{d_1})$ has exponent $2$. A quick computation using \sage \cite{sagemath} shows that out of the 305 imaginary quadratic fields of discriminant $|d_1| < 1000$, a total of $52$ have class number one or exponent $2$.
\end{remark}

We prove the above results by establishing an explicit formula for such CM values of automorphic Green functions. To simplify the exposition we assume in the rest of the introduction that $j$ is even.
Our approach is based on the realization of the modular curve $X$ as an orthogonal Shimura variety and on the regularized theta correspondence. A key observation is that $G_s(C(d_1),z_2)$ can be obtained as the regularized theta lift of a weak Maass form of weight $1/2$. The proof of this fact involves a  quadratic transformation formula for the Gauss hypergeometric function, see Proposition \ref{prop:green12}.

Let $L$ be the lattice of integral $2\times 2$ matrices of trace zero equipped with the quadratic form $Q$ given by the determinant. Let $\SO(L)^+$ be the intersection of the special orthogonal group $\SO(L)$ with the connected component of the identity of $\SO(L_\R)$. We write $\calD$ for the Grassmanian of oriented negative definite planes in $L_\R$, and fix one connected component $\calD^+$. The conjugation action of $\SL_2(\Z)$ on $L$ induces isomorphisms $ \operatorname{PSL}_2(\Z)\cong \SO(L)^+$, and $X\cong \SO(L)^+\bs \calD^+$.

Let $U\subset L_\Q$ be a rational negative definite subspace of dimension $2$. Then $U$ together with the appropriate orientation determines a CM point
\[
z_U^+=U_\R\in \calD^+.
\]
Moreover, we obtain even definite lattices
\[
N=L\cap U,\qquad P=L\cap U^\perp
\]
of signature $(0,2)$ and $(1,0)$, respectively. The binary lattice $N$ can be used to recover the corresponding CM point on $\H$ in classical notation. Both lattices determine holomorphic vector valued theta functions $\theta_{N(-1)}$ and $\theta_P$ of weight $1$ and $1/2$, where $N(-1)$ denotes the positive definite lattice given by $N$ as a $\Z$-module but equipped with the quadratic form $-Q$. According to \cite[Theorem 3.7]{BF} there exists a vector valued harmonic Maass form $\calG_N$ of weight $1$ for $\Gamma$ which maps to $\theta_{N(-1)}$ under the $\xi$-operator, see Section \ref{sect:3.2}.

Since $\theta_P$ transforms with the Weil representation $\rho_P$ of $\Mp_2(\Z)$ on $\C[P'/P]$, and $\calG_N$ transforms with the Weil representation $\rho_N$ on $\C[N'/N]$, their tensor product $\theta_P\otimes \calG_N$ can be viewed as a nonholomorphic modular form for  $\Mp_2(\Z)$ of weight $3/2$ with representation
$\rho_P\otimes \rho_N\cong \rho_{P\oplus N}$. More generally, the $l$-th Rankin-Cohen bracket $[\theta_{P}, \calG_{N}]_{l}$ defines a non-holomorphic modular form of weight $3/2+2l$ with the same representation, see Section~\ref{sect:3.1}.

Recall that for any fundamental discriminant $d<0$ the $d$-th Zagier lift \cite{duke-jenkins-integral} can be viewed as a map
\[
  \za_{d}^j: M^!_{-2j} \longrightarrow M^!_{\frac{1}{2}-j,\bar\rho_L},
\]
from weakly holomorphic modular forms of weight $-2j$ for the group $\Gamma$ to vector valued weakly holomorphic modular forms of weight $1/2-j$ transforming with the complex conjugate of the Weil representation of $\Mp_2(\Z)$ on $\C[L'/L]$, see Section \ref{sec:partial-averages}.
The following result is stated (in greater generality) as Theorem \ref{thm:partial-average-value} in the main text.

\begin{theorem}
\label{thm:partial-average-value-intro}
Let $f\in M_{-2j}^!$ be as before and assume the above notation. Then
\[
G_{j+1,f}(C(d_1), z_U^+) = -2^{j-1} \CT \left\langle \za_{d_1}^{j}(f),\, [\theta_{P}, \calG_{N}^+]_{j/2}\right\rangle,
\]
where $\calG_N^+$ denotes the holomorphic part of $\calG_N$.
Moreover, $\CT$ denotes the constant term of a $q$-series, $\langle\cdot ,\cdot \rangle$ the standard $\C$-bilinear pairing on $\C[(P\oplus N)'/(P\oplus N)]$, and $\za_{d_1}^{j}(f)$ is viewed as a modular form with representation $\bar \rho_{P\oplus N}$ via the natural intertwining operator of Lemma~\ref{sublattice}.
\end{theorem}

Note that this formula holds for {\em any} possible choice of the harmonic Maass form $\calG_N$ mapping to $\theta_{N(-1)}$ under $\xi$.
It is proved in \cite{DL, Eh} (and in greater generality in the appendix of the present paper) that there are particularly nice choices, for which the Fourier coefficients of $\calG_N^+$ are given by logarithms of algebraic numbers in the ring class field $H_{d_2}$, where $d_2$ is the discriminant of the lattice $N$.
By invoking such a nice choice of $\calG_N$, Theorem  \ref{cor:partial-average-algebraic-general-intro} can be derived.


We illustrate this result by an explicit example.
First note that for $j=2,4,6$, it is easily seen  that $G_{j+1} = G_{j+1, f}$
for $f = E_4^{3-j/2}/\Delta$,
where $E_4 \in M_4$ is the normalized Eisenstein series of weight $4$
and $\Delta \in S_{12}$ is the unique normalized cusp form of weight $12$
for $\Gamma$.

First consider the case $j=2$, $d_1 = -4$, and $d_2=-23$.
For the CM point $\frac{1 + i\sqrt{23}}{2}$ of discriminant $d_2$ the lattice
$N$ is isomorphic to the ring of integers in $\Q(\sqrt{-23})$
together with the negative of the norm.
Using the Fourier expansion of $G_s(z_1, z_2)$, we obtain numerically that
\[
  G_3\left(i, \frac{1 + i\sqrt{23}}{2}\right) \approx -1.000394556341.
\]
Let $\calG_N^+(\tau) = \sum_{m} c(m)\phi_{\bar{m}} q^\frac{m}{23}$
be the holomorphic part of a harmonic Maass form
$\calG_N$ with the property that $\xi(\calG_N) = \theta_{N(-1)}(\tau)$,
normalized such that $c(m) = 0$ for $m<-1$.
Here, $\phi_{\bar{m}} \in \C[N'/N]$ only depends on $m$ modulo $23$.
If $m \equiv 0 \bmod{23}$, then $\phi_{\bar{m}} = \phi_{0 + N}$.
If $m \not\equiv 0 \bmod{23}$, then $\phi_{\bar{m}} = \phi_{\mu + N} + \phi_{-\mu + N}$,
where $\mu \in N'/N$ satisfies $Q(\mu) \equiv \frac{m}{23} \bmod{\Z}$.
By Theorem \ref{thm:partial-average-value-intro} we obtain the formula
\begin{align*}
  G_{3}\left(i, \frac{1 + i\sqrt{23}}{2}\right) &= -\frac{25}{23} c(7) - \frac{4}{23} c(14) +  \frac{11}{23} c(19)
                  + \frac{20}{23} c(22) + \frac{1}{2} c(23)
                  + \frac{378}{23} c(-1).
\end{align*}

Now let $R_{-23} \subset H_{-23}$ be the ring of integers
in the Hilbert class field $H_{-23}$ of $\Q(\sqrt{-23})$.
Let $\alpha \approx 1.324717957244$ be the unique real root of $x^3-x-1$.
Then $\alpha$ is a generator of $H_{-23}$ over $\Q(\sqrt{-23})$
and also a generator of the unit group of the real subfield $\Q(\alpha) \subset H_{-23}$.
Using the results on harmonic Maass forms of weight one of \cite{Eh} (see Table \ref{tab:coeffs23}),
we obtain the explicit value
\[
  G_3\left(i, \frac{1 + i\sqrt{23}}{2}\right) = \frac{1}{23} \log \left| \alpha^{294} \cdot \frac{(\alpha^{2} - 2\alpha - 1)^{50}\, (3\alpha^2 - 5\alpha + 1)^{8}}{(4 \alpha^{2} - \alpha + 2)^{40}\, (\alpha^{2} - 4 \alpha + 3)^{22}\, (-3 \alpha^{2} + 2 \alpha + 1)^{23}} \right|.
\]
%
The prime factorization of the argument of the logarithm is given in Section \ref{sec:ex1}. Its norm is $7^{66}\cdot 11^{-80}\cdot 19^{-22}\cdot 23^{-23}$.
Further note that according to Theorem \ref{thm:partial-average-value-intro},
the same form $\calG_N$ appears in the formula for $G_{j+1}\left(i, \frac{1 + i\sqrt{23}}{2}\right)$ for all even $j$, see Section \ref{sect:8.2}.

\subsection{Higher automorphic Green functions on orthogonal Shimura varietes}
We shall actually consider orthogonal Shimura varieties of arbitrary dimension in greater generality as we now describe.
Let $(V,Q)$ be a quadratic space over $\Q$ of signature $(n,2)$, and let $H=\GSpin(V)$. We realize the hermitian symmetric space associated with $H$ as the Grassmannian of negative oriented planes in $V_\R$. For a compact open subgroup $K\subset H(\A_f)$, we consider the Shimura variety
\[
   X_K=H(\Q)\bs (\calD\times H(\A_f)/ K) .
\]
Let $L\subset V$ be an even lattice and assume that $K$ stabilizes $\hat L$ and acts trivially on the discriminant group $L'/L$. For $\mu\in L'/L$ and positive $m\in \Z+Q(\mu)$,  there is a special divisor $Z(m,\mu)$ on $X_K$. The automorphic Green function associated with it is defined by
\begin{align*}
\Phi_{m,\mu}(z,h,s) &=2\frac{\Gamma(s+\frac{n}{4}-\frac{1}{2})}{\Gamma(2s)}
\\
\nonumber
&\phantom{=}{}\times
\sum_{\substack{\lambda\in h(\mu+L)\\ Q(\lambda)=m}}
\left(\frac{m}{Q(\lambda_{z^\perp})}\right)^{s+\frac{n}{4}-\frac{1}{2}}
F\left(s+\frac{n}{4}-\frac{1}{2},s-\frac{n}{4}+\frac{1}{2},2s;\frac{m}{Q(\lambda_{z^\perp})}\right)
\end{align*}
for $(z,h)\in X_K\setminus Z(m,\mu)$ and $s\in \C$ with $\Re(s)\geq s_0:=\frac{n}{4}+\frac{1}{2}$, see Section \ref{sect:4} and \cite{Br}, \cite{OT}.
The sum converges normally and defines a smooth function
in this region with a logarithmic singularity along $Z(m,\mu)$.
It has a meromorphic continuation in $s$ to the whole complex plane and is an eigenfunction of the invariant Laplacian on $X_K$.

In the special case when $L$ is the even unimodular lattice of signature $(2,2)$, and $K\subset H(\A_f)$ is the stabilizer of $\hat L$, the Shimura variety $X_K$ is isomorphic to $X\times X$ and
$\Phi_{m,0}(z,1,s)$ is equal to the Hecke translate
$ - \frac{2}{\Gamma(s)}G_s^m(z_1,z_2)$ of the automorphic Green function for $\Sl_2(\Z)$ above, see Section \ref{sect:resolvent} and Theorem \ref{thm:via}.

The special values of automorphic Green functions at the harmonic point $s=s_0$ are closely related to logarithms of Borcherds products. The logarithm of the Petersson metric of any Borcherds product is a linear combination of the functions $\Phi_{m,\mu}(z,h,s_0)$.
This implies in particular that the CM values of such a linear combination
of Green functions are given by logarithms of algebraic numbers. In view of Conjecture \ref{conj:alg} it is natural to ask whether the values of (suitable linear combinations) of automorphic Green functions at higher spectral parameter $s_0+j$ with $j\in \Z_{>0}$ are also given by logarithms of algebraic numbers. We shall prove this in the present paper for small CM points.

Let $k=1-n/2$ and let $f\in H_{k-2j,\bar\rho_L}$ be a harmonic Maass form of weight $k-2j$ for the conjugate Weil representation $\bar\rho_L$. Applying the $j$-fold iterate raising operator to $f$ we obtain a weak Maass form
 $R_{k-2j}^j f$ of weight $k$. Recall that the Siegel theta function $\theta_L(\tau,z,h)$ associated with $L$ has weight $-k$.
We consider the regularized theta lift
\[
\Phi^j(z,h,f) = \frac{1}{(4\pi)^j}\int_{\calF}^\reg \langle R_{k-2j}^j f(\tau), \theta_L(\tau,z,h)\rangle\, d\mu(\tau),
\]
were $\calF$ denotes the standard fundamental domain for the action of $\SL_2(\Z)$ on $\H$, and  the regularization is done as in \cite{Bo1}.
It turns out that $\Phi^j(z,h,f)$ has a finite value at every point $(z,h)$.
It defines a smooth function on the complement of a certain linear combination of special divisors, which is equal to an explicit linear combination of the `higher' Green functions $\Phi_{m,\mu}(z,h,s_0+j)$, see Proposition \ref{prop:higherphi}.

Let $U\subset V$ be a negative definite $2$-dimensional subspace. Then $T=\GSpin(U)$ determines a torus in $H$, which is isomorphic to the multiplicative group of an imaginary quadratic field, and $U_\R$ together with the choice of an orientation determines two points $z_U^\pm$ in $\calD$. For every $h\in T(\A_f)$ we obtain a (small) CM point $(z_U^+,h)$ in $X_K$.
Moreover, there is a CM cycle
\begin{align*}
Z(U)= T(\Q)\bs (\{ z_U^\pm\} \times T(\A_f)/ K_T)\longrightarrow X_K,
\end{align*}
which is defined over $\Q$.
As in the signature $(1,2)$ case above, the subspace $U$ determines definite lattices $N=L\cap U$, $P=L\cap U^\perp$ and their associated theta series.

\begin{theorem}
\label{thm:fund2-intro}
Let $f\in M^!_{k-2j, \bar\rho_L}$ and $h\in T(\A_f)$.
Then
\begin{align*}
 \Phi^j(z_U^\pm,h,f)&=
\CT\left\langle
f,\, [\theta_P,\calG_{N}^+(\tau,h)]_j\right\rangle,
\end{align*}
where $\calG_{N}^+(\tau,h)$ denotes the holomorphic part of any harmonic Maass form $\calG_{N}(\tau,h)$
with $L_1 \calG_{N}(\tau,h) = \theta_N(\tau,h)$.
\end{theorem}

As before,  when the coefficients of $f$ with negative index are integral,  we may conclude that $\Phi^j(z_U^\pm,h,f)=\frac{1}{r}\log |\alpha|$
for some $\alpha\in H_d$ and $r\in \Z_{>0}$, where $d=-|N'/N|$. Moreover, the Galois action on $\alpha$ is compatible with the action on $(z_U^\pm,h)$ by Shimura reciprocity.
Theorem~\ref{thm:fund2-intro} (and certain variants involving other regularized theta liftings) represents one of the main ingredients of the proof of Theorem \ref{thm:partial-average-value-intro}.
For the average values of higher Green functions at small CM cycles we obtain the following result (Theorem \ref{thm:fund}).


\begin{theorem}
\label{thm:fund-intro}
Let $f\in H_{k-2j, \bar\rho_L}$.
The value of the higher Green function
$\Phi^j(z,h,f)$ at the CM cycle $Z(U)$ is given by
\begin{align}
\label{eq:fund-intro}
\frac{1}{\deg(Z(U))}\Phi^j(Z(U),f)&=
\CT\left(\langle
f^+,\, [\theta_P,\calE^+_N]_j\rangle\right)
-L'\left(\xi_{k-2j}(f),U,0\right).
\end{align}
Here $\calE_N^+$ denotes the holomorphic part of the harmonic Maass form $E_N'(\tau,0;1)$, see \eqref{eq:calE}, and $L(g,U,s)$ is a certain convolution $L$-function of a cusp form $g\in S_{2-k+2j,\rho_L}$ and the theta series $\theta_P$, see Lemma \ref{lem:dirser}.
\end{theorem}

Theorem \ref{thm:fund-intro} is very similar to one of the main results, Theorem 1.2, of \cite{BY}. In loc.~cit.\ it was conjectured that this quantity is the archimedian contribution of an arithmetic intersection pairing of a linear combination of arithmetic special divisors determined by the principal part of $f$ and the CM cycle $Z(U)$ on an integral model of $X_K$.
Here the first quantity on the right hand side is the negative of the non-archimedean intersection pairing. This conjecture was proved in \cite{AGHM} for maximal even lattices. It would be very interesting to establish a similar interpretation of Theorem \ref{thm:fund-intro}. Is it possible to define suitable cycles on fiber product powers of the Kuga-Satake abelian scheme over the canonical integral model of $X_K$ whose non-archimedian intersection pairing is given by the first quantity on the right hand side of \eqref{eq:fund-intro}? We show that this question has an affirmative answer answer when $X_K$ is a modular curve in Section \ref{sect:gkz}, see in particular Theorem \ref{thm:higherheight}.

\subsection{A higher weight Gross-Kohnen-Zagier theorem}

In the special case when $V$ has signature $(1,2)$ and $X_K$ is a modular curve  such `higher' Heegner cycles are defined in \cite{Zh1}, \cite{Xue}. We will use Theorem \ref{thm:fund-intro} to prove a Gross-Kohnen-Zagier theorem in this setting.
Let $\lev$ be a positive integer and let $L$ be the lattice of signature $(1,2)$ and level $4\lev$ defined in \eqref{latticeN}.
Taking $K=\GSpin(\hat L)\subset H(\A_f)$, the Shimura variety $X_K$ is isomorphic to the modular curve $\Gamma_0(\lev)\bs \H$. The special divisor $Z(m,\mu)$ agrees with the Heegner divisor of discriminant $D=-4\lev m$ of \cite{GKZ}. Moreover, the small CM cycle $Z(U)$ agrees with a primitive Heegner divisor. In particular, when the lattice $N$ has fundamental discriminant $D_0$, then $Z(U)$ is equal to a Heegner divisor $Z(m_0,\mu_0)$, where $D_0=-4\lev m_0$.

Let $\kappa$ be an odd positive integer.
For an elliptic curve $E$ with complex multiplication by  $\sqrt{D}$, let $Z(E)$ denote the divisor $\Gamma - (E \times \{0\}) + D(\{0\} \times  E)$ on $ E \times E$, where $\Gamma$  is the graph of multiplication by $\sqrt D$. Then
$Z(E)^{\kappa-1}$ defines a cycle of codimension $\kappa -1$ in $E^{2\kappa-2}$. By means of  this construction Zhang and Xue defined higher Heegner cycles
$Z_\kappa(m,\mu)$
on the $(2\kappa-2)$-tuple fiber product of the universal
degree $\lev$ cyclic isogeny of elliptic curves over the modular curve $X_0(\lev)$, see Section \ref{sect:gkz} for details. Zhang used the Arakelov intersection theory of Gillet and Soul\'e to define a height pairing
of such higher Heegner cycles, which is a sum of local height pairings for each prime $p\leq \infty$. The archimedian contribution to the global height pairing
\[
\langle Z_\kappa(m,\mu), Z_\kappa(m_0,\mu_0)\rangle
\]
is given by the evaluation of a higher Green function at $Z(m_0,\mu_0)$, which can be computed by means of Theorem \ref{thm:fund-intro}. The non-archimedian contribution can be calculated using results of \cite{Xue} and \cite{BY}. It turns out to agree with the negative of the first quantity on the right hand side of \eqref{eq:fund-intro}, yielding a formula for the global height pairing (Theorem~\ref{thm:higherheight}). By invoking a refinement of Borcherds' modularity criterion, we obtain the following higher weight Gross-Kohnen-Zagier theorem (Theorem \ref{thm:modat} and Corollary \ref{cor:modat}).

\begin{theorem}
\label{cor:modat-intro}
Assume the above notation and that $D_0$ is a fixed fundamental discriminant which is coprime to $2\lev$. The generating series
\begin{align*}
\sum_{\substack{m,\mu\\ (4\lev m,D_0)=1}}
\langle Z_\kappa(m,\mu), Z_\kappa(m_0,\mu_0)\rangle \cdot q^m\phi_\mu
\end{align*}
is the Fourier expansion of a cusp form in
$S_{\kappa+1/2,\rho_L}(\Gamma_0(D_0^2))$.
\end{theorem}

Note that this result does {\em not} depend on any assumption regarding positive definiteness of the height pairing.
We consider the generating series which only involves the $Z_\kappa(m,\mu)$ with $4\lev m$ coprime to $D_0$ in order to avoid improper intersections of Heegner cycles. It would be interesting to drop this restriction,
which causes the additional level of the generating series.

The structure of this paper is as follows. In Section~2 we recall some background on orthogonal Shimura vareties and Siegel theta functions, and in Section~3 we collect some facts on weak Maass forms, differential operators, and Rankin-Cohen brackets. Section~4 deals with higher automorphic Green functions on orthogonal Shimura varieties, and in Section~5 the main formulas for their values at small CM cycles are derived.
We also comment on potential analogues for  big CM cycles, see Theorem \ref{thm:fundbig}.
In Section~6 we specialize to signature $(2,2)$ and prove a refinement of the main result of \cite{Vi1}. We also specialize to signature $(1,2)$
and obtain some preliminary results towards Theorem \ref{thm:partial-average-value-intro}. We use this to prove the higher weight Gross-Kohnen-Zagier theorem. In Section~7 we extend the results in the case of signature $(1,2)$ by looking at more general theta kernels which involve different Schwartz functions at the archimedian and non-archimedian places.
In that way we prove (a generalization of) Theorem \ref{thm:partial-average-value-intro}, from which we deduce
Theorem \ref{cor:partial-average-algebraic-general-intro} and Corollary \ref{cor:exponent2-intro}.
Section~8 deals with some numerical examples illustrating our main results. Finally, in the appendix we explain how the main results of \cite{DL, Eh} on harmonic Maass forms of weight 1 can be extended to more general binary lattices.

We thank Ben Howard, Claudia Alfes-Neumann, Yingkun Li and Masao Tsuzuki for helpful conversations and comments related to this paper. We also thank the referee for his/her careful reading of our manuscript
and for the insightful comments.

\section{Orthogonal Shimura varieties and theta functions}
\label{sec:2}
Throughout, we write $\A$
for the ring of adéles over $\Q$ and $\A_f$
for the finite adéles. Moreover, we let $\hat\Z = \prod_{p<\infty} \Z_p$ be the closure of $\Z$ in $\A_f$.

Let $(V,Q)$ be a quadratic space over $\Q$ of signature $(n,2)$. We denote the symmetric bilinear form associated to $Q$ by
$(x,y)=Q(x+y)-Q(x)-Q(y)$. Let $H=\GSpin(V)$, and realize the corresponding hermitean symmetric space as the Grassmannian $\calD$ of two-dimensional negative oriented subspaces of $V_\R$. This space has two connected components, $\calD=\calD^+\sqcup \calD^-$, given by the two possible choices of an orientation. It is isomorphic to the complex manifold
\begin{align}
\label{eq:projmod}
\{z\in V_\C:\; (z,z)=0\; (z,\bar z)<0\}/\C^\times.
\end{align}
For a compact open subgroup $K\subset H(\A_f)$ we consider the quotient
\[
   X_K=H(\Q)\bs (\calD\times H(\A_f)/ K) .
\]
It is the complex analytic space of a Shimura variety
of dimension $n$, which has a canonical model over $\Q$.

There are natural families of special cycles which are given by embeddings of rational quadratic subspaces $V'\subset V$ of signature $(n',2)$ for $0\leq  n'\leq n$. As in \cite{BY} we consider these cycles for $n'=0$ and $n'=n-1$.
Let $U\subset V$ be a negative definite $2$-dimensional subspace. It defines a two point subset $\{ z_U^\pm\}\subset \calD$ given by $U_\R$ with the two possible choices of the orientation. The group $T=\GSpin(U)$ is isomorphic to the multiplicative group of an imaginary quadratic field. It embeds into $H$ acting trivially on $U^\perp$. If we put $K_T=K\cap T(\A_f)$ we obtain the CM cycle
\begin{align}
\label{eq:defzu}
Z(U)= T(\Q)\bs (\{ z_U^\pm\} \times T(\A_f)/ K_T)\longrightarrow X_K.
\end{align}
Here each point in the cycle is counted with multiplicity $\frac{2}{w_{K,T}}$, where $w_{K,T}= \# (T(\Q)\cap K_T)$. The cycle $Z(U)$ has
dimension $0$ and is defined over $\Q$.

To define special divisors, we consider a vector $x\in V$ with $Q(x)> 0$, and let $H_x\subset H$ be its stabilizer. The hermitean symmetric space of $H_x$ can be realized as the analytic divisor
\[
\calD_x = \{ z\in \calD:\; z\perp x\}
\]
in $\calD$. For $h\in H(\A_f)$ we
 let $K_{h, x} = H_x(\A_f) \cap h K h^{-1}$ be the corresponding compact open subgroup of $H_x(\A_f)$.  Then
$$
H_x(\Q) \backslash ( \calD_x \times H_x(\A_f)/K_{h, x}) \rightarrow X_K,  \quad [z, h_1]\mapsto [z, h_1 h]
$$
gives rise to a divisor $Z(h, x)$ in $X_K$.
Given
$m\in \Q_{>0}$ and a $K$-invariant Schwartz function $\varphi\in S(V(\A_f))$, we  define a special divisor $Z(m,\varphi)$ following \cite{Ku:Duke}:
If there exists an $x \in V(\Q)$ with $Q(x) =m$, put
\begin{align}
\label{eq:defzm}
Z(m, \varphi)  =\sum_{h\in H_x(\A_f)\backslash H(\A_f)/K} \varphi(h^{-1}x) Z(h, x) .
\end{align}
If there is no such  $x$, set $Z(m, \varphi)=0$.

\subsection{Siegel theta functions}

We briefly recall the properties of some Siegel theta functions. For more details we refer to \cite{Bo1}, \cite{Ku:Integrals}, \cite{BY}.

We write $\Gamma'=\Mp_2(\Z)$ for the metaplectic extension of $\SL_2(\Z)$ given by the two possible choices of a holomorphic square root on $\H$ of the automorphy factor $j(\gamma, \tau) =c\tau +d $ for $\gamma=\kabcd\in \SL_2(\Z)$ and $\tau\in \H$.

Let $L\subset V$ be an even lattice and write $L'$ for its dual. The discriminant group $L'/L$ is a finite abelian group, equipped with a $\Q/\Z$-valued quadratic form. We write $S_L=\C[L'/L]$ for the space of complex valued functions on $L'/L$. For $\mu \in L'/L$ we denote the characteristic function of $\mu$ by $\phi_\mu$, so that $(\phi_\mu)_\mu$ forms the standard basis of $S_L$. This basis determines a $\C$-bilinear pairing
\[
  \langle x,y\rangle = \sum_{\mu\in L'/L} x_\mu y_\mu
\]
for $x=\sum_\mu x_\mu \phi_\mu$ and $y=\sum_\mu y_\mu \phi_\mu$ in $S_L$.
Recall that there is a Weil representation $\rho_L$ of $\Gamma'$ on $S_L$. In terms of the generators $S=(\kzxz{0}{-1}{1}{0},\sqrt{\tau})$ and $T=(\kzxz{1}{1}{0}{1},1)$ it is given by
\begin{align}
\label{eq:weilt}
\rho_L(T)(\phi_\mu)&=e(Q(\mu))\phi_\mu,\\
\label{eq:weils}
\rho_L(S)(\phi_\mu)&=\frac{e((2-n)/8)}{\sqrt{|L'/L|}} \sum_{\nu\in
L'/L} e(-(\mu,\nu)) \phi_\nu,
\end{align}
see e.g. \cite{Bo1}, \cite{BY}, \cite{Br}. We frequently identify $S_L$ with the subspace of Schwartz-Bruhat functions $S(V(\A_f))$ which are translation invariant under $\hat L = L \otimes_\Z \hat\Z$ and supported on $\hat L'$.
Then the representation $\rho_L$ can be identified with the restriction to $\Mp_2(\Z)$ of the complex conjugate of the usual Weil representation $\omega_f$ on $S(V(\A_f))$ with respect to the standard additive character of $\A/\Q$.

If $z\in \calD$ and $x\in V(\R)$ we write $x_z$ and $x_{z^\perp}$ for the orthogonal projections of $x$ to the subspaces $z$ and $z^\perp$ of $V(\R)$, respectively. The positive definite quadratic form $x\mapsto Q(x_{z^\perp})-Q(x_z)$ is called the majorant associated with $z$.
For $\tau=u+iv\in \H$, $(z,h)\in \calD\times H(\A_f)$, and $\varphi\in S(V(\A_f))$ we define a Siegel theta function by
\begin{align}
\label{theta2} \theta(\tau,z,h,\varphi)  = v  \sum_{x\in V(\Q)}
e\big( Q(x_{z^\perp})\tau+Q(x_{z})\bar\tau\big)\cdot
\varphi(h^{-1} x).
\end{align}
Moreover, we define a $S_L$-valued theta function by
\begin{align}
\label{theta2.5}
\theta_L(\tau,z,h)=\sum_{\mu\in L'/L}
\theta(\tau,z,h,\phi_\mu) \phi_\mu.
\end{align}
As a function of $\tau$ it transforms as a (non-holomorphic) vector-valued modular form of weight $\frac{n}{2}-1$ with representation $\rho_L$ for $\Gamma'$. As a function of $(z,h)$ it descends to $X_K$ if $K$
stabilizes $\hat L$ and acts trivially on $\hat L'/\hat L \cong L'/L$.


\section{Differential operators and weak Maass forms}

Here we recall some differential operators acting on automorphic forms for  $\Gamma'$ and some facts about weak Maass forms.

\subsection{Differential operators}
\label{sect:3.1}

Throughout we use $\tau$ as a standard variable for functions on the upper complex half plane $\H$. We write $\tau =u+iv$ with $u\in \R$ and $v\in \R_{>0}$ for the decomposition into real and imaginary part.
Recall that the Maass raising and lowering operators on smooth functions on $\H$ are defined as the differential operators
\begin{align*}
R_k  &=2i\frac{\partial}{\partial\tau} + k v^{-1},\\
L_k  &= -2i v^2 \frac{\partial}{\partial\bar{\tau}}.
\end{align*}
Occasionally, to lighten the notation, we drop the weight $k$ and simply write $L$ for the lowering operator,
since its definition in fact does not depend on $k$.
The lowering operator annihilates holomorphic functions. Moreover, if $g$ is a holomorphic function on $\H$, then
\begin{align}
\label{eq:rn}
R_{-k}(v^k \bar g)=0.
\end{align}
For any smooth function $f:\H\to \C$, $k\in \frac{1}{2}\Z$, and $\gamma\in \Gamma'$ we have
\begin{align*}
R_k(f\mid_k \gamma) &= (R_k f)\mid_{k+2} \gamma,\\
L_k(f\mid_k \gamma) &= (L_k f)\mid_{k-2} \gamma.
\end{align*}
The hyperbolic Laplacian in weight $k$ is defined by
\begin{equation}\label{defdelta}
\Delta_k = -v^2\left( \frac{\partial^2}{\partial u^2}+ \frac{\partial^2}{\partial v^2}\right) + ikv\left( \frac{\partial}{\partial u}+i \frac{\partial}{\partial v}\right).
\end{equation}
It commutes with the weight $k$ action of $\Gamma'$ on functions on $\H$. It can be expressed in terms of $R_k$ and $L_k$ by
\begin{equation}\label{deltalr}
-\Delta_k = L_{k+2} R_k +k = R_{k-2} L_{k}.
\end{equation}
For $j\in \Z_{\geq 0}$ we abbreviate
\begin{align*}
R_k^j  &= R_{k+2(j-1)}\circ\dots\circ R_k,\\
L_k^j  &= L_{k-2(j-1)}\circ\dots\circ L_k.
\end{align*}
The following lemma is an easy consequence of \eqref{deltalr}.

\begin{lemma}
\label{lem:comrel}
We have
\[
R_{k-2}^j L_k = L_{k+2j} R_k^j + j(k+j-1) R_k^{j-1}.
\]
\end{lemma}

If $\Gamma''\subset \Gamma'$ is a congruence subgroup we write $A_{k}(\Gamma'')$ for the complex vector space of smooth functions $f:\H\to \C$ satisfying the transformation law $f\mid_k \gamma = f$ for all $\gamma\in \Gamma''$.
If $f,g\in A_k(\Gamma'')$ we define their Petersson inner product by
\[
(f,g)_\Pet = \int_{\Gamma''\bs \H} f(\tau)\overline{g(\tau)}v^k\,d\mu,
\]
provided the integral converges. Here $d\mu(\tau) =\frac{du\,dv}{v^2}$ is the usual invariant volume form.
If $f\in A_{k}(\Gamma'')$ and $g\in A_l(\Gamma'')$, we have
\begin{align}
R_{k+l}(fg)= (R_k f) g + f (R_l g).
\end{align}
Moreover, if $g\in A_{-k-2}(\Gamma'')$ we have the identity of invariant differential forms
\begin{align}
\label{eq:rd}
\big((R_k f) g + f (R_{-k-2} g)\big)\,d\mu(\tau)= R_{-2}(fg) \,d\mu(\tau)
=  -d(v^{-2}fg\,d\bar \tau).
\end{align}
In combination with \eqref{eq:rn} this identity implies the following lemma
\begin{lemma}
\label{lem:rpet}
Let $h\in A_{k-2}(\Gamma'')$ and assume that $h$ has moderate growth at the cusps. Then for any holomorphic cusp form $g\in S_k(\Gamma'')$ we have
\[
(R_{k-2}(h),g)_\Pet =0.
\]
\end{lemma}

We will also need Rankin-Cohen brackets on modular forms.  Let $j\in \Z_{\geq 0}$.
If $f\in A_{k}(\Gamma'')$ and $g\in A_l(\Gamma'')$ we define the $j$-th Rankin-Cohen bracket by
\begin{align}
\label{eq:defRC}
[f,g]_j = \sum_{s=0}^j (-1)^s \binom{k+j-1}{s} \binom{l+j-1}{j-s} f^{(j-s)} g^{(s)},
\end{align}
where $f^{(s)}:= \frac{1}{(2\pi i)^s}\frac{\partial^s}{\partial \tau^s} f$. It is well known that the Rankin-Cohen bracket can also be expressed in terms of iterated raising operators as
\begin{align}
\label{eq:defRC2}
[f,g]_j = \frac{1}{(-4\pi)^j}\sum_{s=0}^j (-1)^s \binom{k+j-1}{s} \binom{l+j-1}{j-s} (R_k^{j-s} f) (R_l^{s}g).
\end{align}
The latter identity implies that $[f,g]_j$ belongs to $A_{k+l+2j}(\Gamma'')$. Moreover, \eqref{eq:defRC} implies that the Rankin-Cohen bracket takes (weakly) holomorphic modular forms to (weakly) holomorphic ones.

\begin{lemma}
\label{lem:rcr}
Let $f_1\in A_k(\Gamma'')$ and $f_2\in A_l(\Gamma'')$. There exists a function $h\in A_{k+l+2j-2}(\Gamma'')$ such that
\[
[f_1,f_2]_j=\frac{1}{(4\pi)^j}\binom{k+l+2j-2}{j} \cdot f_1 \cdot R_l^j( f_2)+ R_{k+l+2j-2}(h).
\]
If $f_1$ and $f_2$ have moderate growth, then $h$ can also be chosen to have moderate growth.
\end{lemma}

\begin{proof}
See \cite[Proposition 3]{Vi1}.
\end{proof}


\subsection{Weak Maass forms}
\label{sect:3.2}

As before let $L\subset V $ be an even lattice.
Recall the Weil representation $\rho_L$ of $\Gamma'=\Mp_2(\Z)$ on $S_L$.
Let $k\in \frac{1}{2}\Z$.
For $\gamma=(g,\sigma) \in \Gamma'$ and a function $f: \H\to S_L$ we define the Petersson slash operator in weight $k$ by
\[
(f\mid_{k,\rho_L} \gamma)(\tau) = \sigma(\tau)^{-2k} \rho_L(\gamma)^{-1} f(g\tau).
\]
A smooth function $f:\H\to S_L$ is called a weak Maass form of weight $k$ with representation $\rho_L$ for the group $\Gamma'$ (c.f.~\cite[Section 3]{BF})
if
\begin{enumerate}
\item
$f\mid_{k,\rho_L} \gamma = f$
for all $\gamma\in \Gamma'$;
\item there exists a $\lambda\in \C$ such that
$\Delta_k f=\lambda f$;
\item
there is a $C>0$ such that
$f(\tau)=O(e^{C v})$ as $v\to \infty$ (uniformly in $u$).
\end{enumerate}

In the special case when $\lambda=0$, the function $f$ is called a {\em harmonic} weak Maass form\footnote{To lighten the terminology we will frequently drop the attribute `weak' and simply speak of {\em harmonic Maass forms}.}.
The differential operator
\[
f(\tau)\mapsto \xi_k(f)(\tau):=v^{k-2} \overline{L_k f(\tau)}
\]
defines an antilinear map from harmonic Maass forms of weight $k$ to weakly
holomorphic modular forms of dual weight $2-k$ for the dual representation.
Following \cite{BF}, we denote the holomorphic part of any harmonic Maass form $f$
by $f^+$ and the non-holomorphic part by $f^-$.
As in \cite{BY} we write $H_{k,\rho_L}$ for the vector space of harmonic
Maass forms of weight $k$ (with representation $\rho_L$ for $\Gamma'$) whose
image under $\xi_k$ is a cusp form.
The larger space of all harmonic Maass forms of weight $k$ is denoted by $H_{k,\rho_L}^!$.
We write $M^!_{k,\rho_L}$,  $M_{k,\rho_L}$, $S_{k,\rho_L}$ for the subspaces
of weakly holomorphic modular forms, holomorphic modular forms, and cusp
forms, respectively. Then we have the chain of inclusions
\[
  S_{k,\rho_L}\subset M_{k,\rho_L}\subset M^!_{k,\rho_L}\subset H_{k,\rho_L} \subset H_{k,\rho_L}^!,
\]
and the exact sequence
\begin{align*}
\xymatrix{
0\ar[r]& M_{k,\rho_L}^! \ar[r]& H_{k,\rho_L} \ar[r]^{\xi_k}&  S_{2-k,\bar\rho_L} \ar[r] & 0.}
\end{align*}

Important examples of weak Maass forms are given by Poincar\'e series, see \cite[Chapter~1.3]{Br}.
Let $M_{\nu,\mu}(z)$ be the usual $M$-Whittaker function as defined in
\cite{AS}, Chapter 13, p.~190. For convenience, for $s\in \C$ and $v\in \R_{>0}$ we put
\[
\calM_{s,k}(v)= v^{-k/2}M_{-k/2,s-1/2}(v).
\]
The special value at $s_0=1-k/2$ is given by
\[
\calM_{s_0,k}(v) = e^{v/2}\left( \Gamma(2-k)-(1-k)\Gamma(1-k,v)\right).
\]
For any $m>0$ the function
$\calM_{s,k}(4\pi m v) e(-mu)$
is an eigenfunction of $\Delta_k$ with eigenvalue $(s-k/2)(1-k/2-s)$.

For simplicity we assume here that $2k\equiv -\sig(L)=2-n \pmod{4}$.
Then, for $\mu\in L'/L$ and $m\in \Z+Q(\mu)$ with $m>0$, the $S_L$-valued function
\[
\calM_{s,k}(4\pi m v) e(-mu)(\phi_\mu+\phi_{-\mu})
\]
is invariant under the $\mid_{k,\bar\rho_L}$-action of the stabilizer $\Gamma_\infty'\subset\Gamma'$ of the cusp $\infty$.
%
We define the Hejhal-Poincar\'e series of index $(m,\mu)$ and weight $k$ by
\begin{align}
\label{eq:hps}
F_{m,\mu}(\tau,s,k) = \frac{1}{\Gamma(2s)} \sum_{\substack{\gamma\in \Gamma'_\infty\bs \Gamma'}}\left[ \calM_{s,k}(4\pi m v) e(-mu) (\phi_\mu+\phi_{-\mu})\right] \mid_{k,\bar\rho_L} \gamma.
\end{align}
The series converges normally for $\Re(s)>1$ and defines a weak Maass form of weight $k$ with representation $\bar\rho_L$ and eigenvalue $(s-k/2)(1-k/2-s)$ for $\Gamma'$, see e.g.~\cite[Theorem~1.9]{Br} and note that we work here with signature $(n,2)$ instead of signature $(2,n)$. If $k\leq 1/2$, then the special value
$
F_{m,\mu}(\tau,s_0,k)
$
at $s_0=1-k/2$ defines an element of $H_{k,\bar\rho_L}$ with Fourier expansion
\[
F_{m,\mu}(\tau,s_0,k)= e(-m\tau)\phi_\mu + e(-m\tau)\phi_{-\mu} +O(1),
\]
as $v\to \infty$, see \cite[Proposition 1.10]{Br}.
The next proposition describes the images of the Hejhal-Poincar\'e series under the Maass raising operator.

\begin{proposition}
\label{prop:raisef}
We have that
\[
\frac{1}{4\pi m} R_k F_{m,\mu}(\tau,s,k) = (s+k/2) F_{m,\mu}(\tau,s,k+2).
\]
\end{proposition}

\begin{proof}
Since $R_k$ commutes with the slash operator, it suffices to show that
\[
\frac{1}{4\pi m} R_k\calM_{s,k}(4\pi m v) e(-m u) = (s+k/2) \calM_{s,k+2}(4\pi m v) e(-m u).
\]
This identity follows from (13.4.10) and (13.1.32) in \cite{AS}.
\end{proof}

\begin{corollary}
\label{cor:raisef}
For $s=s_0+j =1-k/2 +j$ we have that
\[
\frac{1}{(4\pi m)^j} R_{k-2j}^j F_{m,\mu}(\tau,s_0+j,k-2j) = j! \cdot F_{m,\mu}(\tau,s_0+j,k).
\]
\end{corollary}

The following lemma on Rankin-Cohen brackets of harmonic Maass forms will be crucial for us. To lighten the notation we formulate it here for scalar valued forms with respect to a congruence subgroup $\Gamma''\subset \Gamma'$. An analogue also holds for vector valued forms.

\begin{proposition}
\label{prop:RC}
Let $f$ be a harmonic Maass form of weight $k$ and let $g$ be a harmonic Maass form of weight $l$ for $\Gamma''$. For any non-negative integer $j$ we have
\[
(-4\pi)^j L [f,g]_j = \binom{l+j-1}{j} (R_k^j f)(Lg) + (-1)^j \binom{k+j-1}{j} (Lf)(R_l^j g).
\]
\end{proposition}

\begin{proof}
According to \eqref{eq:defRC2} we have
\begin{align*}
(-4\pi)^j L[f,g]_j &= \sum_{s=0}^j (-1)^s \binom{k+j-1}{s} \binom{l+j-1}{j-s} (LR_k^{j-s} f) (R_l^{s}g)\\
&\phantom{=}{} + \sum_{s=0}^j (-1)^s \binom{k+j-1}{s} \binom{l+j-1}{j-s} (R_k^{j-s} f) (LR_l^{s}g).
\end{align*}
In the first sum on the right hand side we consider the $s=j$ term separately and in the second sum the $s=0$ term. We obtain
\begin{align}
(-4\pi)^j L[f,g]_j &=\binom{l+j-1}{j} (R_k^j f)(Lg) + (-1)^j \binom{k+j-1}{j} (Lf)(R_l^j g)
\label{eq:RC1}\\
\nonumber
&\phantom{=}{} +\sum_{s=0}^{j-1} (-1)^s \binom{k+j-1}{s} \binom{l+j-1}{j-s} (LR_k^{j-s} f) (R_l^{s}g)\\
\nonumber
&\phantom{=}{} + \sum_{s=1}^j (-1)^s \binom{k+j-1}{s} \binom{l+j-1}{j-s} (R_k^{j-s} f) (LR_l^{s}g).
\end{align}
Since $f$ is a harmonic Maass form, Lemma \ref{lem:comrel} and \eqref{deltalr} imply for $s\leq j-1$ that
\begin{align*}
LR_k^{j-s} f= -(j-s)(k+j-s-1) R_k^{j-s-1} f.
\end{align*}
Similarly, we have for $s\geq 1$ that
\begin{align*}
LR_l^{s} g= -s(l+s-1) R_l^{s-1} g.
\end{align*}
Now a straightforward computation shows
that the two sums over $s$ on the right hand side of \eqref{eq:RC1} cancel. This implies the assertion.
\end{proof}

Let $M \subset L$ be a sublattice of finite index. A vector
valued harmonic Maass form $f \in H_{k,\rho_L}$ can be naturally viewed as
a harmonic form in $H_{k, \rho_M}$. Indeed, we have the
inclusions $M\subset L\subset L'\subset M'$, and therefore an inclusion
\[
 L'/M\subset M'/M
\]
and a natural map $L'/M\to L'/L$, $\mu\mapsto \bar \mu$.

\begin{lemma}
\label{sublattice}
There is a natural map
$$
\operatorname{res}_{L/M}:H_{k,\rho_L} \rightarrow  H_{k,\rho_M},\quad
f\mapsto f_M,
$$
given by
\[
(f_M)_\mu=\begin{cases} f_{\bar\mu},&\text{if $\mu\in L'/M$,}\\
0,&\text{if $\mu\notin L'/M$.}
\end{cases}
\]
\end{lemma}
We refer to \cite[Lemma 3.1]{BY}
for a proof of the lemma and for more details about this construction.


\subsection{Binary theta functions}
\label{sect:bintheta}

In this subsection we assume that $V$ is a definite quadratic space of signature $(0,2)$ and let $L\subset V$ be an even lattice. Then the corresponding  Grassmannian consists of the two points $z_V^\pm$ given by $V_\R$ with the two possible choices of an orientation. The Siegel theta function $\theta_L(\tau,z,h)$ defined in  \eqref{theta2.5} does not depend on $z$, and therefore we often drop this variable from the notation.
This theta function is a non-holomorphic $S_L$-valued modular form of weight $-1$ with representation $\rho_L$. Because of \eqref{eq:rn}, we have
\[
R_{-1}\theta_L(\tau,h) =0.
\]

Following \cite[Section~2]{BY}
we define an $S_L$-valued Eisenstein series of weight $\ell$ by
\begin{align*}
E_L(\tau,s;\ell)=\sum_{\gamma \in \Gamma_\infty'\bs \Gamma'} [v^{(s+1-\ell)/2}\phi_0]\mid_{\ell,\rho_L}\gamma.
\end{align*}
For the Maass operators we have the well known identities
\begin{align}
\label{eq:er}
R_\ell E_L(\tau,s;\ell)&= \frac{1}{2}(s+1+\ell) E_L(\tau,s;\ell+2),\\
\label{eq:el}
L_\ell E_L(\tau,s;\ell)&= \frac{1}{2}(s+1-\ell) E_L(\tau,s;\ell-2).
\end{align}

The following special case of the Siegel-Weil formula relates the average of Siegel theta functions over the genus of $L$  and such an  Eisenstein series, see e.g.~\cite[Theorem 4.1]{Ku:Integrals}. We use the Tamagawa measure on $\SO(V)(\A)$ (such that $\vol(\SO(V)(\Q)\bs \SO(V)(\A))=2$) and normalize the Haar measure on the compact group $\SO(V)(\R)$ to have total mass $1$.
This determines a Haar measure $dh$ on $\SO(V)(\A_f)$ with
$\vol(\SO(V)(\Q)\bs \SO(V)(\A_f))=2$.

\begin{proposition}
\label{prop:sw}
With the above normalization of the Haar measure, we have
\[
\int_{\SO(V)(\Q)\bs \SO(V)(\A_f)} \theta_L(\tau,h) \,dh = E_L(\tau,0;-1).
\]
\end{proposition}

Using the identities
\begin{align}
\label{eq:l1}
L_1 E_L(\tau,s;1)&= \frac{s}{2} E_L(\tau,s;-1),\\
\label{eq:l2}
L_1 E_L'(\tau,0;1)&= \frac{1}{2} E_L(\tau,0;-1),
\end{align}
which follow  from \eqref{eq:el},
we see that the derivative $E_L'(\tau,0;1)$ of the incoherent Eisenstein series of weight $1$ is a preimage under the lowering operator of the average of binary theta functions on the left hand side of Proposition~\ref{prop:sw}.
Moreover, $\calE_L(\tau)=E_L'(\tau,0;1)$ is a harmonic Maass form of weight $1$ with representation $\rho_L$ for $\Gamma'$. We write
\begin{align}
\label{eq:calE}
\calE_L^+(\tau)=\sum_{m,\mu}\kappa(m,\mu)q^m\phi_\mu
\end{align}
for its holomorphic part.
The Fourier coefficients of $E_L'(\tau,0;1)$ are computed in \cite{KY}, see also \cite[Theorem 2.6]{BY}.
In particular, up to a common rational scaling factor, the $\kappa(m, \mu)$ with $(m, \mu) \ne (0, 0)$  are given by logarithms of positive rational numbers.

Later we will also need harmonic Maass forms of weight $1$ that are preimages under the lowering operator of the individual  binary theta functions $\theta_L(\tau,h)$, without taking an average over $h$. While the existence of preimages follows from the surjectivity of the $\xi$-operator \cite[Theorem 3.7]{BF}, it was proved in \cite{Eh} and \cite{DL} that there are actually preimages with nice arithmetic properties.

We let $D$ be the discriminant of $L$ and
$k_{D} = \Q(\sqrt{D}) \cong V$ be the corresponding imaginary quadratic field.
We write $\calO_D \subset \k_{D}$ for the order of discriminant $D$ in $\k_{D}$.
Note that $H(\Q) \cong \k_{D}^\times$ and $H(\A_f) \cong \A_{k_D,f}^\times$.
There is a canonical map from class field theory
\[
	[\, \cdot,\, \k_D ]: \A_{k_D,f}^\times \to \Gal(\k_D^{\ab}/\k_D),
\]
where $\k_D^{\ab}$ denotes the maximal abelian extension of $\k_D$.
The ring class field $H_D$ of $\calO_D$ is fixed by
$[\hat\calO_{D}^\times, \k_D]$ and if $\Cl(\calO_D)$
denotes the class group of $\calO_D$, we have the canonical isomorphisms
(see e.g. \cite[\S 15.E]{cox})
\begin{align}\label{eq:clOD}
  \Cl(\calO_D) \cong k_D^\times \bs \A_{k_D,f}^\times / \hat\calO_{D}^\times \cong \Gal(H_D/\k_D).
\end{align}
Moreover, $K = \hat\calO_D^\times$ stabilizes $L$ as above and acts trivially on $L'/L$.
Hence, $\theta_L(\tau, h)$ defines a function on
\begin{align}\label{eq:XK02}
  X_K \cong \{z_V^\pm\} \times \Cl(\calO_D).
\end{align}
By abusing notation, we will simply write the same symbol $h$
for the element of $H(\A_f)$ and for its class in $H(\Q)\bs H(\A_f) / K \cong \Cl(\calO_D)$.

In \cite{DL} lattices of prime discriminant
and in \cite{Eh} lattices of fundamental discriminant were considered.
The following result generalizes Theorem 4.21 of \cite{Eh}
to arbitrary binary lattices and will be proved in the appendix
(see Theorem \ref{thm:preimages-appendix}).

\begin{theorem} \label{thm:preimages}
  For every $h \in \Cl(\calO_D)$, there is a harmonic Maa\ss{} form
  $\calG_L(\tau,h) \in H_{1,\rho_L}^!$
  with holomorphic part
  \[
    \calG_L^+(\tau,h) = \sum_{\mu \in L'/L} \sum_{m \gg -\infty} c_{L}^{+}(h,m,\mu) e(m \tau)  \phi_{\mu}
  \]
  with the following properties:
  \begin{enumerate}
  \item We have $L_1(\calG_L(\tau,h)) = \theta_L(\tau,h)$.
  \item For all $\mu \in L'/L$ and all $m \in \Q$ with $m \equiv Q(\mu) \bmod{\Z}$ and $(m, \mu) \neq (0,0)$,
    there is an $\alpha_L(h,m,\mu) \in \Hilb_D^\times$ such that
    \begin{equation}
      \label{eq:cprma}
      c_{L}^{+}(h,m,\mu) = - \frac{1}{r} \log|\alpha_L(h,m,\mu)|,
    \end{equation}
    where $r \in \Z_{>0}$ only depends on $L$.
  \item The algebraic numbers $\alpha_L(h, m, \mu)$ satisfy the Shimura reciprocity law
    \[
      \alpha_L(h,m,\mu) = \alpha_L(1,m,\mu)^{[h,\, \k_D]}.
    \]
  \item Additionally, there is an $\alpha_L(h, 0, 0) \in H_D^\times$, such that
  \[
      c_{L}^{+}(h, 0, 0) = \frac{2}{r} \log|\alpha_L(h, 0, 0)| + \kappa(0,0).
  \]
  \end{enumerate}
\end{theorem}

\section{Automorphic Green functions}
\label{sect:4}
In this section we return to the general case that $L$ is an even lattice of signature $(n, 2)$.
Throughout, we assume that the compact open subgroup $K\subset H(\A_f)$ stabilizes $\hat L$ and acts trivially on $L'/ L$.
Let  $f:\H\to S_L$ be a weak Maass form of weight $k=1-n/2$ with representation $\bar\rho_L$ for $\Gamma'$. For $z\in \calD$ and $h\in H(\A_f)$ we consider the regularized theta integral
\begin{align*}
\Phi(z,h,f) = \int_{\calF}^\reg \langle f(\tau), \theta_L(\tau,z,h)\rangle\, d\mu(\tau).
\end{align*}
Here $\calF$ denotes the standard fundamental domain for the action of $\SL_2(\Z)$ on $\H$ and $d\mu(\tau)= \frac{du\,dv}{v^2}$. The regularization is done as in \cite{Bo1}. It turns out that $\Phi(z,h,f)$ has a finite value at every point $(z,h)$. It defines a smooth function on the complement of a certain cycle in $X_K$ (see Proposition \ref{prop:sing} below), and it is locally integrable on $X_K$.

Let $\mu\in L'/L$ and $m\in \Z+Q(\mu)$ with $m>0$.
For the Hejhal-Poincare series of index $(m,\mu)$ defined in \eqref{eq:hps}, the lift
\begin{align}
\label{eq:agf}
\Phi_{m,\mu}(z,h,s) = \Phi(z,h,F_{m,\mu}(\tau,s,k))
\end{align}
is studied in detail in \cite{Br}. It turns out to be an automorphic Green function in the sense of \cite{OT} and is therefore of particular significance.

\subsection{Properties of automorphic Green functions}

Let $F(a,b,c;z)$ denote the Gauss hypergeometric function, see e.g.~\cite[Chapter 15]{AS}. Then, according to \cite[Theorem 2.14]{Br}, we have for $\Re(s)>s_0:=\frac{n}{4}+\frac{1}{2}$ and $(z,h)\in X_K\setminus Z(m,\mu)$ that
\begin{align}
\label{eq:phim}
\Phi_{m,\mu}(z,h,s) &=2\frac{\Gamma(s+\frac{n}{4}-\frac{1}{2})}{\Gamma(2s)}
\\
\nonumber
&\phantom{=}{}\times
\sum_{\substack{\lambda\in h(\mu+L)\\ Q(\lambda)=m}}
\left(\frac{m}{Q(\lambda_{z^\perp})}\right)^{s+\frac{n}{4}-\frac{1}{2}}
F\left(s+\frac{n}{4}-\frac{1}{2},s-\frac{n}{4}+\frac{1}{2},2s;\frac{m}{Q(\lambda_{z^\perp})}\right).
\end{align}
The sum converges normally on the above region and defines a smooth function there. In $s$ it has a meromorphic continuation to the whole complex plane. At $s=s_0$ it has a simple pole with residue proportional to the  degree of $Z(m,\mu)$.

Let $\Delta$ be the $\SO(V)(\R)$-invariant Laplace operator on $\calD$ induced by the Casimir element of the Lie algebra of $\SO(V)(\R)$, normalized as in \cite{Br}.
Note that $\Delta$ is a negative operator in this normalization, and that it is equal to $-8$ times the Laplacian in \cite{OT}.
According to \cite[Theorem~4.6]{Br}, for $(z,h)\in X_K\setminus Z(m,\mu)$ and $\Re(s)>s_0$, the Green function is an eigenfunction of $\Delta$, more precisely
\[
\Delta \Phi_{m,\mu}(z,h,s)= \frac{1}{2}(s-s_0)(s+s_0-1)\Phi_{m,\mu}(z,h,s)
.
\]
The behavior of $ \Phi_{m,\mu}(z,h,s)$ near the divisor $Z(m,\mu)$ is described by the following proposition.

\begin{proposition}
\label{prop:sing}
For any $z_0\in \calD$ there exists a neighborhood $U\subset\calD$ such that the function
\[
\Phi_{m,\mu}(z,h,s)-\frac{2}{\Gamma(2s)}\sum_{\substack{\lambda\in h(\mu+L)\cap z_0^\perp\\ Q(\lambda)=m}}
\int_{v>1}^\reg \calM_{s,k}(4\pi m v) e^{-2\pi m v+4\pi Q(\lambda_z) v}\,\frac{dv}{v}
\]
is smooth on $U$.
Here the regularized integral is defined as the constant term at $s'=0$ in the Laurent expansion of the meromorphic continuation in $s'$ of
\[
\int_{v>1}\calM_{s,k}(4\pi m v) e^{-2\pi m v+4\pi Q(\lambda_z) v} v^{-s'-1}\,dv.
\]
\end{proposition}

We note that the latter integral exists if $\Re(s')$ is large and has a meromorphic continuation in $s'$ because of the asymptotic expansion of the $M$-Whittaker function, see (13.5.1) in \cite{AS}. Hence the regularized integral exists and has a finite value for all $s$. Moreover, we note that the sum over $\lambda$ in Proposition~\ref{prop:sing} is finite, since $z_0^\perp$ is a positive definite subspace of $V(\R)$. In particular we see that $\Phi_{m,\mu}(z,h,s)$ has a well defined finite value even on the divisor $Z(m,\mu)$.
 However, it is not continuous along $Z(m,\mu)$. The proof of the proposition can be given as in \cite[Theorem 6.2]{Bo1}.

We may use the asymptotic behavior of the $M$-Whittaker function to analyze the regularized integral further. We have for $v>0$ that
\begin{align}
\label{eq:masy}
\calM_{s,k}(v)&=e^{-v/2}v^{s-k/2}M(s+k/2,2s,v)\\
\nonumber
&=\frac{\Gamma(2s)}{\Gamma(s+k/2)}e^{v/2}(1+O(v^{-1}))
\end{align}
as $v\to \infty$. We extend the incomplete $\Gamma$-function $\Gamma(0,t)= \int_1^\infty e^{-tv}\,\frac{dv}{v}$ to a function on $\R_{\geq 0}$ by defining it as the regularized integral, that is, as the constant term of the Laurent expansion at $s'=0$ of the meromorphic continuation of $\int_1^\infty e^{-tv}v^{-s'}\,\frac{dv}{v}$. Hence we have
\[
\tilde\Gamma(0,t)= \begin{cases}
\Gamma(0,t),& \text{if $t>0$,}\\
0,&\text{if $t=0$.}
\end{cases}
\]
Using \eqref{eq:masy}, it is easily seen that the function
\[
\frac{2}{\Gamma(2s)}
\int_{v>1}^\reg \calM_{s,k}(4\pi m v) e^{-2\pi m v+4\pi Q(\lambda_z) v}\,\frac{dv}{v}-\frac{2}{\Gamma(s+k/2)}\tilde\Gamma(0,-4\pi Q(\lambda_z))
\]
is continuous in $z$. Taking into account higher terms of the asymptotic expansion, one can also describe the behavior of the derivatives in $z$.

\begin{corollary}
\label{cor:sing}
For any $z_0\in \calD$ there exists a neighborhood $U\subset\calD$ such that the function
\[
\Phi_{m,\mu}(z,h,s)-\frac{2}{\Gamma(s+\frac{1}{2}-\frac{n}{4})}
\sum_{\substack{\lambda\in h(\mu+L)\cap z_0^\perp\\ Q(\lambda)=m}}
\tilde\Gamma(0,-4\pi Q(\lambda_z))
\]
is continuous on $U$.
\end{corollary}

Since $\Gamma(0,t)=-\log(t)+\Gamma'(1)+o(t)$ as $t\to 0$, we see that
$\Phi_{m,\mu}(z,h,s)$ is locally integrable on $X_K$ and defines a current on compactly supported smooth functions on $X_K$.
To describe it, we first fix the normalizations of invariant measures.
Using the projective model \eqref{eq:projmod}, we see that there is a tautological hermitean line bundle $\calL$ over $\calD$. Orthogonal modular forms of weight $w$ can be identified with sections of $\calL^w$.
The first Chern form
\[
\Omega = -dd^c \log|(z,\bar z)|
\]
of $\calL$ is $H(\R)$-invariant and positive. We fix the corresponding invariant volume form $\Omega^n$ on $X_K$ and the volume form $\Omega^{n-1}$ on any divisor on $X_K$.

\begin{theorem}
\label{thm:greencurrent}
As a current on compactly supported smooth functions we have
\[
\left(\Delta -\frac{1}{2}(s-s_0)(s+s_0-1)\right)\left[ \Phi_{m,\mu}(z,h,s)\right]
= -\frac{n}{2\Gamma(s+\frac{1}{2}-\frac{n}{4})} \delta_{Z(m,\mu)}.
\]
Here $\delta_{Z(m,\mu)}$ denotes the Dirac current given by integration over the divisor $Z(m,\mu)$ with respect to the measure $\Omega^{n-1}$.
\end{theorem}

This is essentially \cite[Corollary 3.2.1]{OT}, for the comparison of the normalizations of measures see also \cite[Theorem 4.7]{BK}.

\begin{theorem}
\label{thm:greenint}
If $V$ is anisotropic over $\Q$ or $\Re(s)>\frac{3}{4}n>\frac{3}{2}$, then
$\Phi_{m,\mu}(z,h,s)$ belongs to $L^{2}(X_K)$.
\end{theorem}

\begin{proof}
This follows from \cite[Theorem 5.1.1]{OT} by noticing that $\rho_0=\frac{n}{2}$ and $\tau=2-\frac{2}{n}$ in the notation of loc.~cit. Moreover, by equality (4.19) of  \cite{BK}, the function $\Phi_{m,\mu}(z,h,s)$ is a finite linear combination of the Green functions $G_{2s-1}(x)$ of \cite{OT}.
\end{proof}

The Green function can be characterized as follows.

\begin{proposition}
\label{prop:greenchar}
Assume that  $\Re(s)>\frac{3}{4}n>\frac{3}{2}$.
Let $G(z,h,s)$ be a smooth function on $X_K\setminus Z(m,\mu)$ which is square integrable  
 and  satisfies the current equation of Theorem~\ref{thm:greencurrent}. Then
$G(z,h,s)=\Phi_{m,\mu}(z,h,s)$ on $X_K\setminus Z(m,\mu)$.
\end{proposition}

\begin{proof}
We consider the function
\[
g(z,h,s)= G(z,h,s)-\Phi_{m,\mu}(z,h,s).
\]
It satisfies the current equation
\begin{align}
\label{eq:diff}
\left(\Delta -\frac{1}{2}(s-s_0)(s+s_0-1)\right)\left[ g(z,h,s)\right]
= 0.
\end{align}
Hence, by a standard regularity argument, $g$ extends to a smooth function on all of $X_K$. By the hypothesis $g$ belongs to $L^{2}(X_K)$.  Since $\Delta$ is a negative operator on $L^{2}(X_K)$ and $\Re(s)>s_0=\frac{n}{4}+\frac{1}{2}$, the eigenvalue equation \eqref{eq:diff} implies that $g=0$.
\end{proof}

\begin{remark}
Similar characterizations of $\Phi_{m,\mu}(z,h,s)$ can also be obtained for $n\leq 2$ and $\Re(s)>s_0$ by imposing additional conditions on the growth at the boundary of $X_K$, see e.g. Section~\ref{sect:resolvent} and \cite[Section 3.4]{Zh1}.
\end{remark}


\subsection{Positive integral values of the spectral parameter}
Recall that $k=1-n/2$ and $s_0=\frac{n}{4}+\frac{1}{2}$.
In the present paper we are mainly interested in
the Green function $\Phi_{m,\mu}(z,h,s)$ when the spectral parameter $s$ is specialized to $s_0+j$ for $j\in \Z_{\geq 0}$. These special values can be described using lifts of {\em harmonic} Maass forms.

\begin{proposition}
\label{prop:higherphi}
Let $\mu\in L'/L$ and $m\in \Z+Q(\mu)$ with $m>0$. For $j\in \Z_{> 0}$ let $F_{m,\mu}(\tau,k-2j)\in H_{k-2j, \bar\rho_L}$ be the unique harmonic Maass form whose principal part is given by
\[
F_{m,\mu}(\tau,k-2j)
= e(-m\tau)\phi_\mu + e(-m\tau)\phi_{-\mu} +O(1),
\]
as $v\to \infty$. Then
\[
\Phi_{m,\mu}(z,h,s_0+j)= \frac{1}{(4\pi m)^j j!}\Phi\left(z,h,R^j_{k-2j} F_{m,\mu}(\tau,k-2j)\right).
\]
\end{proposition}

\begin{proof}
The assertion is an immediate consequence of \eqref{eq:agf} and Corollary~\ref{cor:raisef}.
\end{proof}

\section{CM values of higher Green functions}\label{sect:cm}
Recall that $k=1-n/2$.
Proposition \ref{prop:higherphi} suggests to study for any harmonic Maass form $f\in H_{k-2j, \bar\rho_L}$ its `higher' regularized theta lift
\begin{align}
\label{eq:higherphi}
\Phi^j(z,h,f) := \frac{1}{(4\pi)^j }\Phi(z,h, R^j_{k-2j} f).
\end{align}
Denote the Fourier coefficients of $f$ by $c^\pm(m,\mu)$ for $\mu\in L'/L$ and $m\in \Z-Q(\mu)$. Then $\Phi^j(z,h,f)$ is a higher Green function for the divisor
\begin{align}
\label{eq:defzf}
Z^j(f)=\sum_{\substack{m>0\\ \mu\in L'/L}} c^+(-m,\mu) m^j Z(m,\mu)
\end{align}
on $X_K$. Here we compute the values of such regularized theta lifts at small CM cycles.

Let $U\subset V$ be a subspace of signature $(0,2)$ which is defined over $\Q$ and consider the corresponding CM cycle defined in \eqref{eq:defzu}. For
$(z,h)\in Z(U)$ we want to compute the CM value $\Phi^j(z,h,f)$.
Moreover, we are interested in the average over the cycle $Z(U)$,
\begin{align}
\Phi^j(Z(U),f):= \frac{2}{w_{K,T}} \sum_{(z,h)\in \supp(Z(U))} \Phi^j(z,h,f).
\end{align}
According to \cite[Lemma 2.13]{Scho}, we have
\begin{align}
\label{eq:greencm}
\Phi^j(Z(U),f):= \frac{\deg(Z(U))}{2}
\int_{h\in \SO(U)(\Q)\bs \SO(U)(\A_f)}
\Phi^j(z_U^+,h,f)\,dh
\end{align}
and $\deg(Z(U))= \frac{4}{\vol(K_T)}$.

The space $U$ determines two
definite lattices
\begin{align*}
N=L\cap U,\quad  P=L\cap U^\perp.
\end{align*}
The direct sum $N\oplus P\subset L$ is a sublattice of finite index.
For $z=z_U^\pm$ and $h\in T(\A_f)$, the Siegel theta function $\theta_L(\tau,z,h)$ splits as a product
\begin{align*}
\theta_L(\tau,z_U^\pm,h)= \theta_P(\tau)\otimes
\theta_N(\tau,z_U^\pm,h).
\end{align*}
In this case
Lemma~3.1 of \cite{BY} implies for the $\C$-bilinear pairings on $S_L$ and $S_{N\oplus P}$, that
$$
\langle f, \theta_L \rangle =\langle f_{P\oplus N}, \theta_P \otimes
\theta_N \rangle,
$$
where $f_{P\oplus N}$ is defined by Lemma \ref{sublattice}.
Hence we may assume in the following calculation that $L=P\oplus N$ if we
replace $f$ by $f_{P\oplus N}$. For $h\in T(\A_f)$ the CM value we are interested in is given by the regularized integral
\[
\Phi^j(z_U^\pm,h,f)
=\frac{1}{(4\pi)^j }\int_{\calF}^\reg \langle R_{k-2j}^j f,\,
\theta_P(\tau)\otimes \theta_N(\tau,z_U^\pm,h) \rangle\,d\mu(\tau) .
\]
To compute it, we first replace the regularized integral by a limit of truncated integrals.
If $S(q)=\sum_{n\in \Z} a_n q^n$ is a Laurent series in $q$ (or a
holomorphic Fourier series in $\tau$), we write
\begin{align}
\CT(S)=a_0
\end{align}
for the constant term in the $q$-expansion.

\begin{lemma}
\label{limit}
If we define
\begin{align*}
A_0 =(-1)^j\CT\big(\langle (f^{+})^{(j)}(\tau),\, \theta_P(\tau)\otimes \phi_{0+N}
\rangle\big),
\end{align*}
where $f^+$ is the holomorphic part of $f$ and  $g^{(j)}:= \frac{1}{(2\pi i)^j}\frac{\partial^j}{\partial \tau^j} g$, we have
\[
\Phi^j(z_U^\pm,h,f)
=\lim_{T\to \infty}\left[ \frac{1}{(4\pi)^j }\int_{\calF_T} \langle R_{k-2j}^j f(\tau),\,
\theta_P(\tau)\otimes \theta_N(\tau,z_U^\pm,h) \rangle\,d\mu(\tau) -A_0
\log(T)\right].
\]
\end{lemma}

\begin{proof}
This result is proved in the same way as \cite[Lemma 4.5]{BY}.
In addition we use the fact that for a polynomial $P(v^{-1})=\sum_{\ell=0}^j c_\ell v^{-\ell}$ in $v^{-1}$ we have
\[
\int_{v>1}^\reg P(v)\,\frac{dv}{v} =
\lim_{T\to\infty}\left(\int_{v=1}^T P(v)\,\frac{dv}{v}-c_0\log(T)\right).
\]
\end{proof}

By means of \eqref{eq:greencm} and the Siegel-Weil formula in Proposition \ref{prop:sw}, we obtain the following corollary.

\begin{corollary}
\label{cor:limit}
We have
\begin{align*}
\Phi^j(Z(U),f)
&=\frac{\deg Z(U)}{2}\\
&\phantom{=}{}\times \lim_{T\to \infty}\left[\frac{1}{(4\pi)^j } \int_{\calF_T} \langle R_{k-2j}^j f(\tau),
\theta_P(\tau)\otimes E_N(\tau,0;-1) \rangle d\mu(\tau) -2A_0
\log(T)\right].
\end{align*}
\end{corollary}

For any $g\in S_{1+n/2+2j,\rho_L}$ we define an $L$-function by means
of the convolution integral
\begin{align}
\label{eq:L} L(g, U,s)=\big( [\theta_P(\tau),E_N(\tau,s;1)]_j,\,
g(\tau)\big)_{\Pet}.
\end{align}
Here the Petersson scalar product is antilinear in the second argument.
The meromorphic continuation of the Eisenstein series
$E_N(\tau,s;1)$ can be used to obtain a
meromorphic continuation of $L(g, U,
s)$ to the whole complex plane. At $s=0$, the center of symmetry,
$L(g, U, s)$ vanishes because the Eisenstein series $E_N(\tau,s;1)$
is incoherent.

\begin{lemma}
\label{lem:dirser}
Let $g\in S_{1+n/2+2j,\rho_L}$ and denote  by
$g=\sum_{m,\mu}b(m,\mu)q^m\phi_\mu$ the Fourier expansion. Write
$\theta_P=\sum_{m,\mu}r(m,\mu)q^m\phi_\mu$.
Then $L(g, U,s)$
can also be expressed as
\begin{align*}
L(g, U,s)&=\frac{1}{(4\pi)^j}
\binom{2j-k}{j}\frac{\Gamma(\frac{s}{2}+1+j)}{\Gamma(\frac{s}{2}+1)}
\cdot \big( \theta_P \otimes E_N(\tau,s;1+2j),\,
g\big)_{\Pet}.
\end{align*}
Morerover, it has the Dirichlet series representation
\begin{align*}
L(g,U,s)= \frac{1}{(4\pi)^j}
\binom{2j-k}{j}\frac{\Gamma(\frac{s}{2}+1+j)}{\Gamma(\frac{s}{2}+1)}\frac{\Gamma(\frac{s}{2}+\frac{n}{2}+j)}{(4\pi )^{\frac{s}{2}+\frac{n}{2}+j}}
\sum_{\substack{\mu\in P'/P\\m>0}} r(m,\mu)\overline{b(m,\mu)}
 m^{-\frac{s}{2}-\frac{n}{2}-j} .
\end{align*}
\end{lemma}

\begin{proof}
We rewrite the Petersson scalar product \eqref{eq:L} using
Lemma~\ref{lem:rcr} and Lemma~\ref{lem:rpet}. We obtain
\begin{align*}
L(g, U,s)&=\frac{1}{(4\pi)^j}
\binom{2j-k}{j}
\big( \theta_P \otimes (R_1^j E_N(\tau,s;1)),\,
g \big)_{\Pet}.
\end{align*}
According to
\eqref{eq:er} we have
\[
R_\ell^j E_N(\tau,s;1)=\frac{\Gamma(\frac{s}{2}+1+j)}{\Gamma(\frac{s}{2}+1)}  E_N(\tau,s;1+2j),
\]
and therefore
\begin{align}
\label{eq:ll2}
L(g, U,s)&=\frac{1}{(4\pi)^j}
\binom{2j-k}{j}\frac{\Gamma(\frac{s}{2}+1+j)}{\Gamma(\frac{s}{2}+1)}
\big( \theta_P \otimes E_N(\tau,s;1+2j),\,
g\big)_{\Pet}.
\end{align}

The latter scalar product can be computed by means of the usual unfolding argument. We obtain
\begin{align*}
\big( \theta_P \otimes E_N(\tau,s;1+2j),\,
g\big)_{\Pet}&= \int_{\Gamma'_\infty\bs\H} \langle \theta_P\otimes v^{\frac{s}{2}-j}\phi_{0+N},\, \overline g\rangle v^{1+\frac{n}{2}+2j}
\, d\mu(\tau)\\
&=\sum_{\substack{\mu\in P'/P\\m>0}} r(m,\mu)\overline{b(m,\mu)}
\int_0^\infty e^{-4\pi m v} v^{\frac{s}{2}+\frac{n}{2}+j-1}\,dv\\
&=\sum_{\substack{\mu\in P'/P\\m>0}} r(m,\mu)\overline{b(m,\mu)}
(4\pi m)^{-\frac{s}{2}-\frac{n}{2}-j} \Gamma\left(\frac{s}{2}+\frac{n}{2}+j\right).
\end{align*}
Inserting this into \eqref{eq:ll2}, we obtain the claimed Dirichlet series representation.
%
\end{proof}

\begin{theorem}
\label{thm:fund}
Let $f\in H_{k-2j, \bar\rho_L}$.
The value of the higher Green function
$\Phi^j(z,h,f)$ at the CM cycle $Z(U)$ is given by
\begin{align*}
\frac{1}{\deg(Z(U))}\Phi^j(Z(U),f)&=
\CT\left(\langle
f^+,\, [\theta_P,\calE_N^+]_j\rangle\right)
-L'\left(\xi_{k-2j}(f),U,0\right).
\end{align*}
Here $\calE_N^+$ denotes the holomorphic part of the harmonic Maass form $E_N'(\tau,0;1)$, see
\eqref{eq:calE}. The Rankin-Cohen bracket is taken for the weights $(1,1)$.
\end{theorem}


\begin{proof}
According to Corollary \ref{cor:limit} we have
\begin{align*}
\Phi^j(Z(U),f)
&=\frac{\deg Z(U)}{2}\\
&\phantom{=}{}\times \lim_{T\to \infty}\left[\frac{1}{(4\pi)^j } \int_{\calF_T} \langle R_{k-2j}^j f(\tau),
\theta_P(\tau)\otimes E_N(\tau,0;-1) \rangle d\mu(\tau) -2A_0
\log(T)\right].
\end{align*}
We use the `self-adjointness' of the raising operator (which is a consequence of \eqref{eq:rd} and Stokes' theorem) to rewrite this as
\begin{align}
\label{eq:fund1}
\Phi^j(Z(U),f)
&=\frac{\deg Z(U)}{2}\\
\nonumber
&\phantom{=}{}\times \lim_{T\to \infty}\left[\frac{1}{(-4\pi)^j }\int_{\calF_T} \langle  f(\tau),
R_{-k}^j(\theta_P(\tau)\otimes E_N(\tau,0;-1)) \rangle d\mu(\tau) -2A_0
\log(T)\right].
\end{align}
Here the vanishing of the boundary terms in the limit $T\to \infty$ follows from Lemma~4.2 of \cite{Br} by inserting the Fourier expansions.
Since $R_{-1}E_N(\tau,0;-1)=0$ and because of \eqref{eq:l2}, we have
\begin{align*}
R_{-k}^j\left(\theta_P(\tau)\otimes E_N(\tau,0;-1)\right)&=\left(R_{1-k}^j \theta_P(\tau)\right)\otimes E_N(\tau,0;-1)\\
&=2 \left(R_{1-k}^j \theta_P(\tau)\right)\otimes \left(L_1 E_N'(\tau,0;1)\right).
\end{align*}
By Proposition \ref{prop:RC}, we find
\begin{align}
\label{eq:raisid}
R_{-k}^j\left(\theta_P(\tau)\otimes E_N(\tau,0;-1)\right)&=
2 (-4\pi)^j
L [\theta_P(\tau), E_N'(\tau,0;1)]_j.
\end{align}
Hence, we obtain for the integral
\begin{align}
I_T(f):=\frac{1}{(-4\pi)^j}\int_{\calF_T}  \langle  f(\tau),\,
R_{-k}^j(\theta_P(\tau)\otimes E_N(\tau,0;-1)) \rangle\,d\mu(\tau)
\end{align}
that
\begin{align*}
I_T(f)&=2\int_{\calF_T}   \langle  f,\,L [\theta_P, E_N'(\cdot,0;1)]_j
 \rangle\,d\mu(\tau)\\
&=2\int_{\calF_T}L   \langle  f,\, [\theta_P, E_N'(\cdot,0;1)]_j
 \rangle\,d\mu(\tau)-2\int_{\calF_T}   \langle ( Lf),\, [\theta_P, E_N'(\cdot,0;1)]_j
 \rangle\,d\mu(\tau).
\end{align*}
The second summand on the right hand side can be interpreted as the Petersson scalar product of $[\theta_P, E_N'(\cdot,0;1)]_j$ and the cusp form $\xi_{k-2j}(f)$.
The first summand can be computed by means of Stokes' theorem. We get
\begin{align*}
\int_{\calF_T}L   \langle  f,\, [\theta_P, E_N'(\cdot,0;1)]_j
 \rangle\,d\mu(\tau)&=-\int_{\calF_T} d\left(  \langle  f,\, [\theta_P, E_N'(\cdot,0;1)]_j
 \rangle\,d\tau\right)\\
&=-\int_{\partial \calF_T}   \langle  f,\, [\theta_P, E_N'(\cdot,0;1)]_j
 \rangle\,d\tau\\
&=\int_{\tau=iT}^{iT+1}   \langle  f,\, [\theta_P, E_N'(\cdot,0;1)]_j
 \rangle\,d\tau.
\end{align*}
Consequently,
\begin{align*}
\lim_{T\to \infty} \left[ I_T(f)-2A_0\log(T)\right]&=2\lim_{T\to \infty} \left[ \int_{\tau=iT}^{iT+1}   \langle  f,\, [\theta_P, E_N'(\cdot,0;1)]_j\rangle\, d\tau -A_0\log(T)\right]\\
&\phantom{=}{}-2\cdot \big( [\theta_P, E_N'(\cdot,0;1)]_j, \,\xi_{k-2j}(f)\big)_\Pet.
\end{align*}
As in the proof of \cite[Theorem 4.7]{BY} we find that the first term on the right hand side is equal to
\[
2\CT\left(\langle  f^+,\, [\theta_P, \calE_N^+]_j\rangle\right).
\]
Putting this into \eqref{eq:fund1} and inserting \eqref{eq:L}, we obtain
\begin{align*}
\Phi^j(Z(U),f)
&=\deg (Z(U) )\left(\CT\left(\langle  f^+,\, [\theta_P, \calE_N^+]_j\rangle\right)
 -L'(\xi_{k-2j}(f),U,0)\right).
\end{align*}
This concludes the proof of the theorem.
\end{proof}

In the same way one proves the following result for the values at individual CM points.

\begin{theorem}
\label{thm:fund2}
Let $f\in M^!_{k-2j, \bar\rho_L}$ and $h\in T(\A_f)$.
Then we have
\begin{align*}
 \Phi^j(z_U^\pm,h,f)&=
\CT\left(\langle
f,\, [\theta_P,\calG_{N}^+(\tau,h)]_j\rangle\right).
\end{align*}
Here, $\calG_{N}^+(\tau,h)$ denotes holomorphic part of \emph{any} harmonic Maass form $\calG_{N}(\tau,h) \in H_{1, \rho_N}^!$
satisfying $L_1 \calG_{N}(\tau,h) = \theta_N(\tau,h)$.
\end{theorem}

By taking the particulary nice preimages $\calG_{N}(\tau,h)$ from Theorem \ref{thm:preimages},
we obtain the following algebraicity statement.
We use the same notation as in Theorem \ref{thm:preimages},
in particular, $D<0$ denotes the discriminant of $N$, $\k_D = \Q(\sqrt{D})$,
and we write $H_D$ for the ring class field of the order $\calO_D \subset \k_D$
of discriminant $D$.

\begin{corollary}
\label{cor:alg}
Assume that $f \in M_{k-2j, \bar{\rho}_L}^!$ has integral Fourier  coefficients and that $Z^j(f)$ is disjoint from $Z(U)$.
For every $h \in \Cl(\calO_D)$ there is an $\alpha_{U,f}(h) \in H_D^\times$ such that:
   \begin{enumerate}
   	\item
		We have
   \[
     A^j \cdot \Phi^j(z_U^\pm,h,f) = -\frac{1}{r} \cdot \log |\alpha_{U,f}(h)|,
   \]
   where $A\in \Z_{>0}$ is the least common multiple of the levels of the lattices $P$ and $N$,
   and $r \in \Z_{>0}$ is a constant that only depends on $L$ and $D$ (but not on $f$, $h$ or $j$).
   \item The algebraic numbers $\alpha_{U,f}(h)$ satisfy the Shimura reciprocity law
   \[
   	 \alpha_{U,f}(h) = \alpha_{U,f}(1)^{[h,\, \k_D]}.
   \]
   \end{enumerate}
 \end{corollary}
 \begin{proof}
   Recall that we assume that $L = P \oplus N$ holds.
   According to \cite[Lemma 2.6]{madapusi-pera-spin}, the group
   $\GSpin(\hat{L})$ is the maximal subgroup of $H(\A_f)$
   that preserves $L$ and acts trivially on $L'/L$.
   Hence $f$ is $\GSpin(\hat{L})$-invariant and we
   can assume that $K = \GSpin(\hat{L})$.
   We now show that $K_T \cong \hat{\calO}_{D}^\times$.
   Consider the embedding $\iota: T \hookrightarrow H$,
   where $T$ acts trivially on $U^\perp$.
   Since $\GSpin(\hat{N}) = \hat\calO_{D}^\times$, we have $\iota(\hat{\calO}_{D}^\times) \subset K_T$.
   However, by the maximality of $\GSpin(\hat{N})$,
   the other inclusion follows as well.
   Now the assertion follows from Theorem \ref{thm:preimages}.
%
 \end{proof}

\subsection{General CM cycles}

Let $d \in \Z_{\geq 0}$, and let $F$ be a totally real number field of degree $d+1$ with real embeddings $\sigma_0, \dots, \sigma_d$. Let $(W, Q_F)$  be a quadratic space over $F$ of dimension $2$ with signature $(0, 2)$ at the place $\sigma_0$ and signature $(2,0)$ at the places $\sigma_1,\dots ,\sigma_d$.
%
%
Let $W_\Q=\Res_{F/\Q} W =(W, Q_\Q)$ be  the space $W$ viewed as a $\Q$-vector space with the $\Q$-valued quadratic form $Q_\Q(w) =\tr_{F/\Q}(Q_F(w))$. Then  $W_\Q$ is a quadratic space over $\Q$ of signature $(2d, 2)$. In this subsection we assume that there is an isometric embedding $i: (W_\Q, Q_\Q) \to (V, Q)$, which we fix throughout.
This gives an orthogonal decomposition
\begin{align*}
V \cong V_0 \oplus W_\Q.
\end{align*}

Let $T\subset H$ be the inverse image under the natural map of the subgroup $\Res_{F/\Q} \SO(W)$ of $\SO(W_\Q)$. Then $T$ is a  torus in $H$, fitting into the commutative diagramm
$$
\xymatrix{
1 \ar[r] &\mathbb G_m \ar[r] \ar[d] &T \ar[r] \ar[d]&\Res_{F/\Q} \SO(W) \ar[r] \ar[d] &1
\\
1 \ar[r] &\mathbb G_m \ar[r] &H \ar[r] &\SO(V) \ar[r] &1.
}
$$
The even Clifford algebra $C_F^0(W)$ of $W$ is a CM field $E$ over $F$.
It is easily checked that $T(\Q)\cong E^\times/F^1$, where $F^1$ denotes the group of norm $1$ elements in $F$.

The subspace $W_{\sigma_0} =W\otimes_{F, \sigma_0}\R \subset V_\R$ is a negative $2$-plane. Together with the choice of an orientation it determines two points $z_{\sigma_0}^\pm$ in $\calD$. The image of natural map
$$
T(\Q)\bs \{ z_{\sigma_0}^\pm\}  \times T(\A_f)/K_T  \longrightarrow X_K,
$$
where $K_T = T(\A_f) \cap K$, determines
a CM cycle $Z(W, \sigma_0)$ of dimension $0$, which is defined over $\sigma_0(F)$. Its Galois conjugate $\sigma_i \sigma_0^{-1} (Z(W, \sigma_0))$ is defined  over $\sigma_i(F)$ for $0 \le i \le d$. It  is equal to a certain Hecke translate of  the CM cycle $Z(W_i, \sigma_i)$, where $W_0=W$, and  $(W_i, Q_{F,i})$ is the quadratic space over $F$ such that $(W_{i, v}, Q_i) \cong (W_v, Q_F)$ for all  primes (finite and infinite)  $v \ne \sigma_0, \sigma_i$, and $W_{i, \sigma_0}$ is positive definite and $W_{i, \sigma_i}$ is negative definite. Notice that there is an isometry of quadratic spaces $W_\Q \cong W_{i, \Q}$ over $\Q$.
 The specific Hecke translate is given  in \cite[Section 2]{BKY} and is related to the choices of isomorphisms $W_\Q \cong W_{i, \Q}$ and $W_{f} \cong W_{i, f}$. We refer to \cite{BKY} for details.  Hence the CM cycle
\begin{equation}
Z(W) = \sum_{i=0}^d Z(W_i, \sigma_i)
\end{equation}
is defined over $\Q$. We remark  that  different $i$'s might give the same Galois conjugate, in such a case $Z(W)$ is a multiple of the formal sum of the Galois conjugates of  $Z(W, \sigma_0)$. When  $F=\Q$, $Z(W)$ is a small CM cycle as defined before. When  $V_0=0$, i.e.,
$V\cong W_\Q$, it is a big CM cycle studied in \cite{BKY}.  The general case is studied by Peng Yu in his thesis \cite{Yu-thesis}.

Let $N=L\cap W_\Q$ and $P=L\cap V_0$, and let $\theta_P(\tau)$ be the Siegel-theta function of weight $\frac{n}2 -d$ associated to $P$. Associated to $N \subset W_\Q=W$, there are $d+1$ coherent  Hilbert Eisenstein series $E_N(\vec\tau, s; \mathbf{1}(i))$ of weight $\mathbf{1}(i)$ over $F$ ($0\le i \le d$) and one incoherent  Hilbert  Eisenstein series $E_N(\vec\tau, s;\mathbf{1})$ of weight $\mathbf{1}=(1, \dots, 1)$.  Here $\mathbf{1}(i)$ is obtained from  $\mathbf{1}$ by replacing the $i$-th entry $1$ by $-1$ as in  \cite[Section 5]{BKY}. These Eisenstein series  are related by the identity
\begin{equation}
L_{1, i} E_N'(\vec\tau, 0; \mathbf{1}) = \frac{1}2  E_N(\vec\tau, 0; \mathbf{1}(i)).
\end{equation}
Here $L_{1, i}$ is the Maass lowering operator with respect to the variable $\tau_i$. In particular, if we denote by $\H\to \H^{d+1}$, $\tau\mapsto \tau^\Delta =(\tau, \dots, \tau)$ the diagonal embedding, we have
$$
L_1 E_N'(\tau^\Delta, 0; \mathbf{1}) = \frac{1}2 \sum_i E_N(\tau^\Delta, 0; \mathbf{1}(i)).
\quad
$$
The same argument as in Corollary \ref{cor:limit} (see also \cite{Yu-thesis} or \cite{BKY}) leads to the following proposition.

\begin{proposition}
Let the notation be as above, and let $f\in H_{k-2j, \bar\rho_L}$. Then the CM value of the higher Green function $\Phi^j(z,h,f)$ is given by
\begin{align*}
\Phi^j(Z(W), f) &= \frac{\deg Z(W, \sigma_0)}{2(4\pi)^j}  \int^\reg_\calF \big\langle R_{k-2j}^j f, \,\theta_P \otimes \sum_i E_N(\tau^\Delta, 0; \mathbf{1}(i)) \big\rangle \,d\mu(\tau)
 \\
  &=\frac{\deg Z(W, \sigma_0)}{(-4\pi)^j}\int^\reg_\calF \left\langle  f,\, R_{-k}^j(\theta_P \otimes L_1E_N'(\tau^\Delta, 0; \mathbf{1})) \right\rangle  \,d\mu(\tau).
\end{align*}
\end{proposition}

In order to derive an explicit formula for this CM value analogous to
Theorem \ref{thm:fund}, we need to find an explicit modular form $G$ on $\H$ (smooth and with possible `poles' at cusps)
such that
\begin{align}
\label{eq:wanted}
L_{2-k+2j} (G) =R_{-k}^j \left(\theta_P \otimes L_1E_N'(\tau^\Delta, 0; \mathbf{1}) \right).
\end{align}
In the case of small CM cyles, that is, for $d=0$ we could use
\eqref{eq:raisid} for this purpose.

There is one further case, in which we can determine a function $G$, this is the case when $d=1$ and $V_0=0$, which we assume for the rest of this subsection. These conditions imply that $F$ is real quadratic, $V=\Res_{F/\Q}(W)$ has signature $(2,2)$, $L=N$, and $k=0$. The function $G$ is obtained using a Cohen operator on Hilbert modular forms, which is a slight variant of the Rankin-Cohen bracket considered in Section \ref{sect:3.1}. Let $g:\H^2\to \C$ be a smooth Hilbert modular form of weight $(k_1,k_2)$ for some congruence subgroup of $\SL_2(F)$. Define the $j$-th Cohen operator as
\begin{align}
\label{eq:defCO}
\calC_j(g)(\tau)= \frac{1}{(2\pi i)^j}\sum_{s=0}^j (-1)^s \binom{k_1+j-1}{s} \binom{k_2+j-1}{j-s}  \left(\frac{\partial^{j-s}}{\partial \tau_1^{j-s}}
\frac{\partial^s}{\partial \tau_2^s} \,g\right) (\tau,\tau) .
\end{align}
Then $\calC_j(g)$ is a smooth function on $\H$ which is modular in weight $k_1+k_2+2j$ for some congruence subgroup. The following result generalizes Proposition \ref{prop:RC}.

\begin{proposition}
\label{prop:CO}
Let $g$ be a smooth Hilbert modular form of weight $(k_1,k_2)$ for the real quadratic field $F$. Assume that, as a function of the first variable, $g$ is annihilated by $\Delta_{k_1}$, and, as a function of the second variable, $g$ is annihilated by $\Delta_{k_2}$. Then
for any non-negative integer $j$ we have
\[
(-4\pi)^j L (\calC_j(g)) = \binom{k_2+j-1}{j} (R_{k_1,1}^j L_{k_2,2} g)(\tau,\tau)+ (-1)^j \binom{k_1+j-1}{j} (L_{k_1,1} R_{k_2,2}^j g)(\tau,\tau).
\]
\end{proposition}

Applying this for $g=E'_N(\tau_1,\tau_2,0,\mathbf{1})$ we find:

\begin{corollary}
\label{cor:CO}
Assume that $d=1$ and $V=\Res_{F/\Q}(W)$ as above. Then
\[
(-4\pi)^j L_{2+2j} \calC_j(E'_N(\cdot,0,\mathbf{1})) =R_{0}^j \left[\big(L_{1,2} E_N'(\cdot,0;\mathbf{1})\big)(\tau^\Delta)+ (-1)^j  \big(L_{1,1}  E_N'(\cdot,0;\mathbf{1})\big)(\tau^\Delta)\right].
\]
\end{corollary}

Hence the function $\calC_j(E'_N(\cdot,0,\mathbf{1}))$ on the left hand side has essentially the property that is required in \eqref{eq:wanted} for the function $G$, except for the sign $(-1)^j$ which appears in addition on the right hand side. However, this sign can be fixed by slightly redefining the CM cycle
by putting
\begin{align}
Z^j(W) = Z(W_1,\sigma_1)+(-1)^j Z(W_0, \sigma_0).
\end{align}
Now the analogue of Theorem \ref{thm:fund} in this case is as follows.

\begin{theorem}
\label{thm:fundbig}
Assume that $d=1$ and $V=\Res_{F/\Q}(W)$ as above. For $f\in H_{-2j,\bar\rho_L}$ we have
\begin{align*}
\frac{1}{\deg(Z(W,\sigma_0))}\Phi^j(Z^j(W),f)&=
\CT\left(\langle
f^+,\, \calC_j(\calE_L^+)\rangle\right)
- \big(\calC_j(E_L'(\cdot,0;\mathbf{1})), \,\xi_{-2j}(f)\big)_\Pet.
\end{align*}
Here $\calE_L^+$ denotes the `holomorphic part' of $E_L'(\vec\tau,0;1)$ (see  \cite[Proposition 4.6]{BKY}), that is, the part of the Fourier expansion which is indexed by totally positive $\nu\in F$ together with the holomorphic contribution of the constant term. The Cohen operator is taken with respect to the parallel weight $(1,1)$.
\end{theorem}

We omit the proof since it is analogous to the one of Theorem \ref{thm:fund} with Proposition \ref{prop:RC} replaced by Corollary \ref{cor:CO}.

The second term on the right hand side is the central derivative of the Rankin-Selberg type integral
\begin{align*}
L(g,W,s)=\big(\calC_j(E_L(\cdot,s;\mathbf{1})), \, g\big)_\Pet.
\end{align*}
for a cusp form $g\in S_{2-2j,\rho_L}$, similarly as in \cite[Section III]{GKZ}. When $f$ is weakly holomorphic, this contribution vanishes, and Theorem \ref{thm:fundbig} gives an explicit formula for the value of the higher Green function $\Phi^j(z,h,f)$ at the CM cycle $Z^j(W)$.
Using the explicit formulas for the coefficients of $E_L'$ of we see that
\begin{align}
d_F^{j/2}\Phi^j(Z^j(W),f) = C\cdot \log \alpha,
\end{align}
for a constant $C\in \Q$ only depending on $L$ and a positive rational number $\alpha$ whose prime factorization can be determined explicitly. Here $d_F$ denotes the discriminant of $F$.

It would be very interesting to generalize this result to general $d\geq 0$. The crucial point would be to obtain an analogue of Corollary \ref{cor:CO} or some other variant of \eqref{eq:wanted}.
While there are Cohen operators for higher degree Hilbert modular forms (see e.g.~\cite{Lee}) there does not seem to be a direct analogue of Corollary \ref{cor:CO}.

\section{The Gross-Zagier conjecture and higher Heegner cycles}

Here we consider examples of our main results for $n=1,2$. These can be used to prove certain cases of an algebraicity conjecture of Gross and Zagier and a higher weight version of the Gross-Kohnen-Zagier theorem.

\subsection{The resolvent kernel}
\label{sect:resolvent}

Let $V=\Mat_2(\Q)$ with the quadratic form $Q(X)=\det(X)$ of signature $(2,2)$. The corresponding bilinear form is $(X,Y)=\tr(XY^*)$, where
\[
 \abcd^* = \begin{pmatrix}
   d & -b \\
   -c & a
 \end{pmatrix}
\]
denotes the adjugate matrix.
We consider the even unimodular lattice  $L=\Mat_2(\Z)$ in $V$.
In this case
\begin{align*}
H&\cong \{ (g_1,g_2)\in \Gl_2\times \Gl_2:\; \det(g_1)=\det(g_2)\},
\end{align*}
and the natural map $H\to \SO(V)$ is given by $(g_1,g_2).X = g_1 X g_2^{-1}$. The space $\H\times \H$ can be identified with the Hermitian space $\calD^+$  by the map
\begin{align*}
z=(z_1,z_2)\mapsto \R\Re X(z)+\R\Im X(z),
\end{align*}
where
\begin{align}\label{eq:Xz}
X(z)=\zxz{z_1 }{-z_1z_2}{1}{-z_2}\in L_\C.
\end{align}
Under this identification, the action of $H(\R)$ on $\calD$ becomes the natural action by fractional linear transformations.
We fix the compact open subgroup $K=\GSpin(\hat L)\subset H(\A_f)$. Then the corresponding Shimura variety $X_K$ is isomorphic to $X(1)^2$, where $X(1)=\SL_2(\Z)\bs \H$.
If $\lambda=\kabcd\in L$, then
\begin{align*}
(\lambda,X(z))&= cz_1 z_2+dz_1-az_2-b,\\
(X(z),\overline{X(z)})&= -4y_1 y_2,\\
(\lambda_z,\lambda_z)&= 2\frac{|(\lambda,X(z))|^2}{(X(z),\overline{X(z)})}= -\frac{|cz_1 z_2+dz_1-az_2-b|^2}{2y_1y_2}.
\end{align*}

The automorphic Green function \eqref{eq:phim} can be interpreted as a function on $\H\times \H$. Using the fact that
\begin{align}
\label{eq:legendre}
Q_{s-1}(t)= \frac{\Gamma(s)^2}{2\Gamma(2s)} \left(\frac{2}{1+t}\right)^s
F(s,s,2s;\frac{2}{1+t}),
\end{align}
where $Q_{s-1}(t)= \int_0^\infty (t+\sqrt{t^2-1}\cosh(u))^{-s}\, du$ denotes the Legendre function, we obtain
\begin{align*}
\Phi_{m}(z,s) &=2\frac{\Gamma(s)}{\Gamma(2s)}
\sum_{\substack{\lambda\in L\\ Q(\lambda)=m}}
\left(\frac{m}{Q(\lambda_{z^\perp})}\right)^{s}
F\left(s,s,2s;\frac{m}{Q(\lambda_{z^\perp})}\right)\\
&= \frac{4}{\Gamma(s)}
\sum_{\substack{\lambda\in L\\ Q(\lambda)=m}}
Q_{s-1} \left( 1-\frac{2Q(\lambda_z)}{m}\right).
\end{align*}
Here we have dropped the subscript $\mu$ from the notation in $\Phi_{m,\mu}(z,h,s)$, since $L$ is unimodular, and the argument $h\in H(\A_f)$, since we have evaluated at $h=1$.
With the above formula for $(\lambda_z,\lambda_z)$ we get
\begin{align*}
\Phi_{m}(z,s)
&= \frac{4}{\Gamma(s)}
\sum_{\substack{\gamma=\kabcd\in L\\ \det(\gamma)=m}}
Q_{s-1} \left( 1+\frac{|cz_1 z_2+dz_1-az_2-b|^2}{2my_1y_2}\right)\\
&= \frac{4}{\Gamma(s)}
\sum_{\substack{\gamma\in L\\ \det(\gamma)=m}}
Q_{s-1} \left( 1+\frac{|z_1 -\gamma z_2|^2}{2m\Im(z_1)\Im(\gamma z_2)}\right).
\end{align*}
Hence $\Phi_m(z,s)= - \frac{2}{\Gamma(s)}G_s^m(z_1,z_2)$, where $G_s^m(z_1,z_2)$ denotes the Green function defined the introduction in \eqref{eq:greenintro}.
In particular, for $m=1$ we obtain the resolvent kernel $G_s = G_s^1$ for the hyperbolic Laplacian.
It has the following properties, see \cite{Hejhal}, \cite[Chapter~2.2]{GZ}.

\begin{enumerate}
\item[(i)] $G_s$ is smooth on $(X(1)\times X(1))\setminus Z(1)$, where $Z(1)$ denotes the diagonal;
\item[(ii)] it satisfies $\Delta_i G_s= s(1-s) G_s$, where $\Delta_i=-y_i^2\left(\frac{\partial^2}{\partial x_i^2}+ \frac{\partial^2}{\partial y_i^2}\right)$ is the hyperbolic Laplacian in the variable $z_i$ for $i=1,2$;
\item[(iii)]  we have $G_s(z_1,z_2)= e_{z_2} \log|z_1-z_2|^2+O(1)$ as $z_1\to z_2$, where $e_{z_2}$ denotes the order of the stabilizer of $z_2$ in $\operatorname{PSL}_2(\Z)$;
\item[(iv)]  $G_s(z_1,z_2)=O(y_1^{1-s})$ as $y_1\to \infty$ for fixed $z_2$.
\end{enumerate}
When $\Re(s)>1$, these properties characterize $G_s$ uniquely, see \cite[Section 3.4]{Zh1}. Here, property (iv) replaces the integrability condition in the analogous result Proposition~\ref{prop:greenchar}.

Let $j\in \Z_{>0}$. For $m>0$, let $f_m \in H_{-2j,\bar\rho_L}$ be the unique harmonic Maass form whose Fourier expansion starts with $f_m = q^{-m}+O(1)$ as $v\to\infty$. Then according to \eqref{eq:higherphi} and Proposition~\ref{prop:higherphi} we have
\begin{align}
\Phi^j(z,f_m)=\frac{1}{(4\pi )^j} \Phi(z,R^j_{-2j} f_m) = - m^j G_{1+j}^m(z_1,z_2).
\end{align}
It is easy to see that the two-dimensional, positive definite subspaces of $V(\Q)$
are in one-to-one correspondence to pairs $(z_1, z_2)$
of CM points lying in the same imaginary quadratic field.

Consequently, Theorem \ref{thm:fund} and Theorem \ref{thm:fund2} lead to formulas and algebraicity statements for CM values of the Green functions $G_{1+j,f}(z_1,z_2)$ defined in \eqref{eq:Gkf},
when evaluated at two different CM points $z_1$ and $z_2$, such that $\Q(z_1) = \Q(z_2)$.
In particular, Corollary~\ref{cor:alg} implies the following
strengthening of a result due to Viazovska \cite[Theorem~7]{Vi1}.

\begin{theorem}
\label{thm:via}
Let $z_1, z_2 \in X(1)$
be different CM points corresponding to CM orders of discriminant $D_1 = t_1^2 D_0$ and $D_2 = t_2^2 D_0$
in the same imaginary quadratic field $\Q(\sqrt{D_0})$.
Moreover, let $j>0$ and suppose that
$f \in M_{-2j}^!$ has integral Fourier coefficients.
Then there is an $\alpha \in H_D^\times$, where $D = \lcm(t_1, t_2)^2 D_0$, such that
\[
	|D|^j G_{1+j,f}(z_1^\sigma, z_2^\sigma) = - \frac{1}{r}\log |\alpha^\sigma|
\]
for every $\sigma \in \Gal(H_{D}/\Q(\sqrt{D_0}))$,
and where $r \in \Z_{>0}$ only depends on $D$, not on $j$ or $f$.
\end{theorem}

\subsection{Modular curves}
\label{sect:6.2}

While Theorem \ref{thm:via} gives the algebraicity of the values of higher Green functions at a pair of CM points for the {\em same} CM  field, we now consider evaluations at pairs of CM points for different CM fields. We obtain new results in this case by applying Theorems \ref{thm:fund} and \ref{thm:fund2} for a suitable quadratic space of signature $(1,2)$. Our argument also relies on a quadratic transformation formula for the Gauss hypergeometric function.

Fix a positive integer $\lev$. Here we let
\[
V=\Mat_2^0(\Q)=\{X\in \Mat_2(\Q):\; \tr(X)=0\}
\]
with the quadratic form $Q(X)=\lev \det(X)$ of signature $(1,2)$.
The corresponding bilinear form is $(X,Y)=\lev\tr(XY^*)$.
In this case $H\cong \Gl_2$ with the Clifford norm corresponding to the determinant,
and the natural map $H\to \SO(V)$ is given by $g.X = g X g^{-1}$.
The upper half plane $\H$ can be identified with the hermitean space $\calD^+$  by the map
$z\mapsto \R\Re X(z)+\R\Im X(z)$,
where
\begin{align}\label{eq:Xz12}
  X(z)=\zxz{z }{-z^2}{1}{-z}\in L_\C.
\end{align}
Let $L$ be the lattice
\begin{align}
\label{latticeN}
 L=\left\{\zxz{b}{-a/\lev}{c}{-b}:\quad a,b,c\in
\Z\right\}
\end{align}
in $V$. The dual lattice is given by
\begin{align}
\label{latticeN2}
 L'=\left\{\zxz{b/2\lev}{-a/\lev}{c}{-b/2\lev}:\quad
\text{$a,b,c\in \Z$} \right\}.
\end{align}
We frequently identify $\Z/2\lev\Z$ with $L'/L$ via $r \mapsto
\mu_r=\diag(r/2\lev, -r/2\lev)$. Here the quadratic form on $L'/L$ is
identified with the quadratic form $x\mapsto -x^2$ on $\Z/2\lev\Z$. The
level of $L$ is $4\lev$.
We fix the compact open subgroup $K=\GSpin(\hat L)\subset H(\A_f)$. Then the complex space of the corresponding Shimura variety $X_K$ is isomorphic to the modular curve $\Gamma_0(\lev)\bs \H$.

Let $m\in \Q_{>0}$ and let $\mu_r\in L'/L$ such that $Q(\mu_r) \equiv m
\pmod{1}$. Then $D:=-4\lev m\in \Z$ is a negative discriminant satisfying $D \equiv r^2 \pmod{
4\lev}$, and we have
$$
Z(m, \mu_r) =P_{D, r} + P_{D, -r},
$$
where $P_{D, r}$ is the Heegner divisor defined in \cite{GKZ}.
If $\lambda=\kzxz{b/2\lev}{-a/\lev}{c}{-b/2\lev}\in L'$ with $Q(\lambda)=m$,
we denote the associated CM point by
\[
  z_\lambda= \frac{b}{2\lev c}+\frac{\sqrt{b^2-4\lev ac}}{2\lev |c|}\in \H.
\]
Using a similar calculation as in Section \ref{sect:resolvent}, we obtain
\begin{align}
\label{eq:qlp}
\sqrt{\frac{Q(\lambda_{z^\perp})}{m}}  =\frac{|a-b\Re (z)+c|z|^2|}{2\sqrt{\lev m}y} = 1+\frac{|z-z_\lambda|^2}{2y\Im (z_\lambda)}.
\end{align}
The automorphic Green function \eqref{eq:phim} can be interpreted as the function on $\H$ given by
\begin{align}
\label{eq:phim12}
\Phi_{m,\mu}(z,s) &=2\frac{\Gamma(s-\frac{1}{4})}{\Gamma(2s)}
\sum_{\substack{\lambda\in \mu+L\\ Q(\lambda)=m}}
\left(\frac{m}{Q(\lambda_{z^\perp})}\right)^{s-\frac{1}{4}}
F\left(s-\frac{1}{4},s+\frac{1}{4},2s;\frac{m}{Q(\lambda_{z^\perp})}\right).
\end{align}
Here we have dropped the
argument $h\in H(\A_f)$, since we evaluated at $h=1$.

\begin{proposition}
\label{prop:green12}
Assume the above notation, and let  $G_{\lev,s}(z_1,z_2)=G_{\lev ,s}^1(z_1,z_2)$ denote the Green function for $\Gamma_0(\lev )$ as in \cite[Chapter 2.2]{GZ}. Then
\[
\Phi_{m,\mu}(z,s)
=-\frac{2}{\Gamma(s+\frac{1}{4})} G_{\lev ,2s-\frac{1}{2}}(z,Z(m,\mu)).
\]
\end{proposition}

\begin{proof}
We employ the quadratic transformation formula for the Gauss hypergeometric function
\begin{align*}
F(a,a+\frac{1}{2},c; w)=\left(\frac{1}{1+\sqrt{w}}\right)^{2a} F\left(2a,c-\frac{1}{2},2c-1;\frac{2\sqrt{w}}{1+\sqrt{w}}\right),
\end{align*}
see \cite[(15.2.20)]{AS}, to deduce
\begin{align*}
&\left(\frac{m}{Q(\lambda_{z^\perp})}\right)^{s-\frac{1}{4}} F\left(s-\frac{1}{4},s+\frac{1}{4},2s;\frac{m}{Q(\lambda_{z^\perp})}\right)\\
&=
\left(\frac{\sqrt{m}}{\sqrt{Q(\lambda_{z^\perp})}+\sqrt{m}}\right)^{2s-\frac{1}{2}}F\left(2s-\frac{1}{2},2s-\frac{1}{2}, 4s-1; \frac{2\sqrt{m}}{\sqrt{Q(\lambda_{z^\perp})}+\sqrt{m}}\right).
\end{align*}
Applying \eqref{eq:legendre} and \eqref{eq:qlp} to the right hand side we get
\begin{align}
\label{eq:hypid}
\left(\frac{m}{Q(\lambda_{z^\perp})}\right)^{s-\frac{1}{4}} F\left(s-\frac{1}{4},s+\frac{1}{4},2s;\frac{m}{Q(\lambda_{z^\perp})}\right)
&=\frac{2^{\frac{3}{2}-2s}\Gamma(4s-1)}{\Gamma(2s-\frac{1}{2})^2}Q_{2s-\frac{3}{2}}\left( \sqrt{\frac{Q(\lambda_{z^\perp})}{m}}   \right)\\
\nonumber
&=\frac{2^{\frac{3}{2}-2s}\Gamma(4s-1)}{\Gamma(2s-\frac{1}{2})^2}Q_{2s-\frac{3}{2}}\left(  1+\frac{|z-z_\lambda|^2}{2y\Im (z_\lambda)}    \right).
\end{align}
Inserting this into \eqref{eq:phim12}, we obtain
\begin{align*}
\Phi_{m,\mu}(z,s)
&= 2^{\frac{5}{2}-2s}\frac{\Gamma(s-\frac{1}{4})}{\Gamma(2s)}
\frac{\Gamma(4s-1)}{\Gamma(2s-\frac{1}{2})^2}
\sum_{\substack{\lambda\in \mu +L\\ Q(\lambda)=m}}
Q_{2s-\frac{3}{2}}\left(  1+\frac{|z-z_\lambda|^2}{2y\Im (z_\lambda)} \right)\\
&=\frac{4}{\Gamma(s+\frac{1}{4})}
\sum_{\substack{\lambda\in (\mu +L)/\Gamma_0(\lev )\\ Q(\lambda)=m}}
\sum_{\gamma\in \Gamma_0(\lev)/\Gamma_0(\lev)_\lambda}
Q_{2s-\frac{3}{2}}\left(  1+\frac{|z-\gamma z_\lambda|^2}{2y\Im (\gamma z_\lambda)}   \right)\\
&=-\frac{2}{\Gamma(s+\frac{1}{4})} G_{\lev ,2s-\frac{1}{2}}(z,Z(m,\mu)).
\end{align*}
This concludes the proof of the proposition.
\end{proof}

To state Theorem \ref{thm:fund} in the present case, we
first interpret the $L$-function $L(g,U,s)$ following \cite[Section 7.2]{BY}.
Let $x_0\in L'$ be primitive and assume that $m_0=Q(x_0)>0$.
We let $\mu_0=x_0+L\in L'/L$. Throughout we assume that $D_0=-4\lev m_0$ is a fundamental discriminant.
We consider the CM cycle $Z(U)$ associated with the negative definite subspace $U=V\cap x_0^\perp$.
The corresponding negative definite lattice $N=L\cap U$ has determinant $|D_0|$.
Since $D_0$ is fundamental, we have $Z(U)=Z(m_0, \mu_{0})$.

Let $S_{2+4j}^-(\Gamma_0(\lev))$ denote the subspace of cusp forms in $S_{2+4j}(\Gamma_0(\lev))$ which are invariant under the Fricke involution.
Recall that there is a Shimura lifting map
$\shim_{m_0,\mu_0}: S_{3/2+2j,\rho_L}\to S_{2+4j}(\Gamma_0(\lev))$ which is given in terms of Fourier series by
\begin{align} \label{eqY7.12}
g=\sum_\mu \sum_{m>0} b(m,\mu) q^m\phi_\mu \mapsto
\shim_{m_0,\mu_0}(g)= \sum_{n=1}^\infty \sum_{d\mid n}
d^{2j}\left(\frac{D_0}{d}\right) b\left(m_0 \frac{n^2}{d^2},
\mu_0\frac{n}{d} \right) q^n,
\end{align}
see \cite[ Section II.3]{GKZ}.
If we denote
the Fourier coefficients of $\shim_{m_0,\mu_0}(g)$ by $B(n)$, then
we may rewrite the formula for the image as the identity of Dirichlet series
\begin{align}
\label{eq:dirid}
L\left(\shim_{m_0,\mu_0}(g), s\right)= \sum_{n>0}
B(n)n^{-s} = L(\chi_{D_0},s-2j)\cdot \sum_{n>0} b\left(m_0 n^2, \mu_0
n\right)
 n^{-s}.
\end{align}
The maps $\shim_{m_0,\mu_0}$ are Hecke-equivariant and there is a
linear combination of them which determines an isomorphism of the subspaces of newforms of
$S_{3/2+2j,\rho_L}$ and $S_{2+4j}^-(\Gamma_0(\lev))$, see \cite{SZ88}.
If $g\in
S_{3/2+2j,\rho_L}$ is a newform that corresponds to the
normalized newform $G\in S_{2+4j}^-(\Gamma_0(\lev))$ under the Shimura
correspondence, then
\begin{align}
\label{eq:newform}
 L\left(\shim_{m_0,\mu_0}(g), s\right) =
b\left(m_0,\mu_0\right)\cdot L(G,s).
\end{align}

\begin{lemma}
\label{prop:l1} Let $m_0$, $\mu_0$, $D_0$, $U$ be as above.
If $g\in S_{3/2+2j,\rho_L}$ has real coefficients $b(m,\mu)$, then
\[
L(g, U, s)= C(s,j)\cdot (4\pi m_0)^{-\frac{s}{2}-\frac{1}{2}-j}
L(\chi_{D_0},s+1)^{-1} L\big(\shim_{m_0,\mu_0}(g),s+1+2j\big),
\]
where
\[
C(s,j)=\frac{2^{1-s}\Gamma(s+1+2j)\Gamma(2j+\frac{1}{2})}
{(4\pi)^j\Gamma(2j+1)\Gamma(\frac{s}{2}+1)}.
\]
In particular,
\begin{align*}
L'(g, U, 0) &=\frac{2^{2-4j}\pi^{-3/2-2j}\Gamma(2j+1/2)}{m_0^j\deg(Z(U))} b(m_0, \mu_0) L'(G,1+2j),
\end{align*}
if $g \in S_{3/2+2j, \rho_{L}}$ and $G\in S_{2+4j}(\Gamma_0(\lev))$ are
 further related by (\ref{eq:newform}).
\end{lemma}

\begin{proof}
According to Lemma \ref{lem:dirser}
we have
\begin{align*}
L(g,U,s)= \frac{1}{(4\pi)^j}
\binom{2j-k}{j}\frac{\Gamma(\frac{s}{2}+1+j)}{\Gamma(\frac{s}{2}+1)}\frac{\Gamma(\frac{s}{2}+\frac{1}{2}+j)}{(4\pi )^{\frac{s}{2}+\frac{1}{2}+j}}
\sum_{\substack{\mu\in P'/P\\m>0}} r(m,\mu)\overline{b(m,\mu)}
 m^{-\frac{s}{2}-\frac{1}{2}-j} ,
\end{align*}
where we view $g$ as a modular form with representation
$\rho_{P\oplus N}$ via Lemma \ref{sublattice}.
Using \eqref{eq:dirid} and the fact that
$b(Q(\lambda), \lambda)=0$ for $\lambda \in P'$ unless
$\lambda\in  P'\cap L'=\Z x_0 $,
we obtain
\begin{align*}
\sum_{\substack{\mu\in P'/P\\m>0}} r(m,\mu)\overline{b(m,\mu)}
 m^{-\frac{s}{2}-\frac{n}{2}-j}&=2\sum_{l=1}^\infty b(Q(lx_0),lx_0)Q(lx_0)^{-\frac{s}{2}-\frac{1}{2}-j}\\
&=2m_0^{-\frac{s}{2}-\frac{1}{2}-j}\sum_{l=1}^\infty b(l^2m_0,l\mu_0) l^{-s-1-2j}\\
&=2m_0^{-\frac{s}{2}-\frac{1}{2}-j}L(\chi_{D_0}, s+1)^{-1} L(\shim_{m_0,\mu_0}(g),s+1+2j).
\end{align*}
This implies the claimed formula with
\[
C(s,j)=2
\binom{2j-k}{j}\frac{\Gamma(\frac{s}{2}+1+j)\Gamma(\frac{s}{2}+\frac{1}{2}+j)}{(4\pi)^j\Gamma(\frac{s}{2}+1)}.
\]
Simplifying this gamma factor we obtain the first formula.
For the derivative,
we use \eqref{eq:newform} and the fact that
\[
\deg(Z(U))= \frac{2\sqrt{|D_0|}}{\pi} L(\chi_{D_0},1),
\]
see \cite[Lemma 6.3]{BY}, to deduce the assertion.
\end{proof}
Combining Theorem \ref{thm:fund} and Lemma \ref{prop:l1}, we obtain the following result.

\begin{theorem}
\label{thm:fund12}
Let $f\in H_{1/2-2j, \bar\rho_L}$ and use the above notation.
The value of the higher Green function
$\Phi^j(z,f)$ at the CM cycle $Z(U)$ is given by
\begin{align*}
\frac{1}{\deg(Z(U))}\Phi^j(Z(U),f)&=
\CT\left(\langle
f^+,\, [\theta_P,\calE_N^+]_j\rangle\right)\\
&\phantom{=}{}
-\frac{C(0,j)}{(4\pi m_0)^{\frac{1}{2}+j}
L(\chi_{D_0},1)} L'\big(\shim_{m_0,\mu_0}(\xi_{1/2-2j}f),1+2j\big)
.
\end{align*}
Here $\calE_N^+$ denotes the holomorphic part of the harmonic Maass form $E_N'(\tau,0;1)$, see
\eqref{eq:calE}.
\end{theorem}

Let $j\in \Z_{>0}$ and let $f_{m,\mu}\in H_{1/2-2j,\bar\rho_L}$ be the unique harmonic Maass form, whose Fourier expansion starts with $q^{-m}(\phi_{\mu}+\phi_{-\mu})+O(1)$ as $v\to \infty$.
Then according to
Proposition~\ref{prop:higherphi} and Proposition \ref{prop:green12} we have
\begin{align}
\label{eq:go12}
\Phi^j(z,f_{m,\mu})=\frac{1}{(4\pi )^j} \Phi(z,R^j_{1/2-2j} f_{m,\mu}) = -2 m^j G_{\lev,1+2j}(z,Z(m,\mu)).
\end{align}
Hence Theorem \ref{thm:fund12} and Theorem \ref{thm:fund2} can be applied to obtain algebraicity results for CM values of the higher Green function on the right hand side. We will come back to this topic and some refinements in Section \ref{sec:partial-averages}.

\subsection{A Gross-Kohnen-Zagier theorem for higher weight Heegner cycles}
\label{sect:gkz}

Here we employ Theorem \ref{thm:fund12} together with some results of \cite{Zh1} and \cite{Xue}, to prove a Gross-Kohnen-Zagier theorem for Heegner cycles on Kuga-Sato varieties over modular curves, see Theorem ~\ref{thm:modat}.

We begin by recalling some basic facts on Kuga-Sato varieties and their CM cycles following Zhang \cite{Zh1} and \cite{Xue}.
Let $\kappa>1$ be an integer, and let $D<0$ be a discriminant.
For an elliptic curve $E$ with complex multiplication by  $\sqrt{D}$, let $Z(E)$ denote the divisor class on $ E \times E$
of  $\Gamma - (E \times \{0\}) + D(\{0\} \times  E)$ , where $\Gamma$  is the graph of multiplication with $\sqrt D$. Then
$Z(E)^{\kappa-1}$ defines a cycle of codimension $\kappa -1$ in $E^{2\kappa-2}$. Denote by $S_\kappa(E)$ the cycle
$$
c \sum_{\sigma\in P_{2\kappa-2}}
\sgn(\sigma) \,   \sigma^*(Z(E)^{\kappa-1}),
$$
where $P_{2\kappa-2}$  denotes the symmetric group of $2\kappa-2$ letters which acts on $E^{2\kappa-2}$ by permuting
the factors, and $c$ is a real number such that the self-intersection of $S_\kappa(E)$ on each
fiber is $(-1)^{\kappa-1}$.

When $\lev$  is a product of two relatively prime integers bigger than $2$, it can be shown that the universal
elliptic curve over the non-cuspidal locus of the modular curve $\mathcal X(\lev)$ (over $\Z$) with full level $\lev$ can be extended uniquely to a regular
semi-stable elliptic curve $\mathcal E(\lev)$ over the whole $\mathcal X(\lev)$. The Kuga-Sato variety $\mathcal Y =\mathcal  Y_\kappa(\lev)$ is
defined to be a certain canonical resolution of the $(2\kappa-2)$-tuple fiber product of $\mathcal E(\lev)$ over $\mathcal X(\lev)$, see \cite[Section 2]{Zh1}.
If $y$ is a CM point on $\mathcal X(\lev)$, the CM-cycle $ S_\kappa(y)$ over $ y$  is defined to be $ S_\kappa(\mathcal E_y)$ in $\mathcal{Y}$.

For a general integer $\lev \ge 1$, we choose a  positive integer $\lev'$ such that $\lev |\lev'$ and $\lev'$ is the product of two co-prime integers bigger than $2$. Let $\pi:  \mathcal X(\lev') \rightarrow \mathcal X_0(\lev)$ be the natural projection.
If $x$ is a CM point on $\mathcal  X_0(\lev)$,
then  $\pi^*(x) = \frac{w(x)}{2} \sum_i x_i$ with $\pi(x_i) =x$ and $w(x) =|\Aut(x)|$. The CM-cycle
$S_\kappa(x)$ over $x$ is defined to be
$\sum_i  S_\kappa(x_i)/\sqrt{\deg \pi}$.

Let $X_0(\lev)$ and $Y$ be the generic fibers of $\mathcal X_0(\lev)$ and $\mathcal Y$. For a CM  point $x \in  X_0(\lev)$, let $\bar x$ be its Zariski closure in $\mathcal X_0(\lev)$.
It is proved in \cite{Zh1}
 that $S_\kappa(\bar x)$ has zero intersection with any cycle of dimension $\kappa$ in $\mathcal Y$ which is supported in the special fibers. Moreover, the class of $S_\kappa(x)$ in $ H^{2\kappa-2}(Y(\C),\C) $ vanishes, which implies that there is a Green
current $g_\kappa(x)$ on $Y(\C)$,  unique up to the image of $\partial$ and $\bar{\partial}$,  such that
$$
\frac{1}{\pi i}\partial \bar{\partial}\, g_\kappa(x) = \delta_{S_\kappa(x)},
$$
where the current on the right hand side is the Dirac current given by integration over $S_\kappa(x)$, and
$$
\int g_\kappa(x) \eta =0
$$
for any $\partial \bar{\partial}$-closed form $\eta$ on $Y(\mathbb C)$.
The arithmetic CM-cycle $ \hat S_\kappa(x)$ over $x$, in the sense of Gillet and Soul\'e \cite{GS}, is the arithmetic codimension $\kappa$ cycle on $\mathcal Y$ defined by
\begin{align}
\hat{S}_\kappa(x) = (S_\kappa(\bar x), g_\kappa(x)).
\end{align}

Now let $x$ and $y$ be two different CM points on $X_0(\lev)$.
Then the height pairing of the CM cycles $S_\kappa(x)$
and $S_\kappa(y)$ on $Y$ is defined as the arithmetic intersection
\begin{align}
\langle S_\kappa(x), S_\kappa(y)\rangle := (-1)^\kappa (\hat{ S}_\kappa(x)\cdot \hat{S}_\kappa(y))_{GS}.
\end{align}
According to \cite[Section 3.2]{Zh1}, it decomposes into local contributions
$$
\langle S_\kappa(x), S_\kappa(y)\rangle =
\langle S_\kappa(x), S_\kappa(y)\rangle_{fin} + \langle S_\kappa(x), S_\kappa(y)\rangle_{\infty},
$$
with
\begin{align}
\label{eq:finp}
\langle S_\kappa(x), S_\kappa(y)\rangle_{fin} &=
\sum_{p<\infty} \langle S_\kappa(x), S_\kappa(y)\rangle_{p}=
(-1)^\kappa \sum_{p<\infty} ( S_\kappa(\bar x)\cdot  S_\kappa(\bar y))_p,
\\
\label{eq:infp}
 \langle S_\kappa(x), S_\kappa(y)\rangle_{\infty} &=\frac{1}2 G_{\lev,\kappa}(x, y).
\end{align}
Here the last identity is \cite[Proposition 3.4.1]{Zh1}, and $G_{\lev,\kappa}(x, y)$ is the higher Green function defined in \cite[Eq.~(2.10)]{GZ}.
Let $U$ and $m_0$ be as in Section \ref{sect:6.2}. Following \cite{Xue}, we define higher Heegner divisors for $X_0(\lev)$ as
 $$
 Z_\kappa(U)= m_0^{\frac{\kappa-1}2} \sum_{x \in Z(U)}  S_\kappa(x),
 $$
and
$$
Z_\kappa(m, \mu)=m^{\frac{\kappa-1}2} \sum_{x \in Z(m, \mu)}  S_\kappa(x).
$$
It is our goal to compute the height pairing of these divisors in the case of proper intersection.

Assume that $D_0=-4\lev m_0$ is fundamental and coprime to $2M$ (in particular $ Z_\kappa(U)=Z_\kappa(m_0, \mu_0)$).
Let $m_1\in \Q_{>0}$ and let $\mu_1\in L'/L$ such that $Q(\mu_1) \equiv m_1
\pmod{1}$. Then $D_1:=-4\lev m_1\in \Z$ is a negative discriminant which we assume to be coprime to $D_0$.

\begin{theorem}
\label{thm:higherheight}
Let the notation be as above and assume that $\kappa=1+2j>1$ is an odd integer.  Let $f_{m_1,\mu_1}\in H_{3/2-\kappa,\bar\rho_L}$ be the unique harmonic Maass form, whose Fourier expansion starts with $q^{-m_1}(\phi_{\mu_1}+\phi_{-\mu_1})+O(1)$.
\begin{enumerate}
\item One has
\begin{align*}
\langle Z_\kappa(m_1, \mu_1),  Z_\kappa(U) \rangle_{fin}
&=\frac{\deg (Z(U))}{4} m_0^\frac{\kappa-1}{2} \cdot \CT\left(\langle f_{m_1, \mu_1}^+, [\theta_{P}, \calE_{N}^+]_\frac{\kappa-1}{2}\rangle\right).
\end{align*}

\item  The global height pairing is given by
\begin{align*}
\langle Z_\kappa(m_1, \mu_1),  Z_\kappa(U) \rangle
&=\frac{\sqrt \lev \Gamma(\kappa-1/2)}{(4\pi)^{\kappa-1}\pi^{3/2} }L'\big(\shim_{m_0,\mu_0}(\xi_{3/2-\kappa}f_{m_1, \mu_1}),\kappa\big).
\end{align*}
\end{enumerate}
\end{theorem}

\begin{proof}
By \eqref{eq:infp}, \eqref{eq:go12}, and Theorem \ref{thm:fund12}, we have
\begin{align*}
 \langle Z_\kappa(m_1, \mu_1), Z_\kappa(U) \rangle_\infty
&= \frac{(m_0 m_1)^{\frac{\kappa-1}{2}}}{2} G_{\lev,\kappa}(Z(U), Z(m_1, \mu_1))
\\
&=-\frac{m_0^{\frac{\kappa-1}{2}}}{4} \Phi^\frac{\kappa-1}{2}(Z(U),f_{m_1,\mu_1})\\
&=-m_0 ^{\frac{\kappa-1}{2}}\frac{\deg Z(U)}{4} \bigg(
\CT\left(\langle
f_{m_1, \mu_1}^+,\, [\theta_P,\calE_N^+]_\frac{\kappa-1}{2}\rangle\right)\\
&\phantom{= -}{}
-\frac{C(0,\frac{\kappa-1}{2})}{(4\pi m_0)^{\frac{\kappa}{2}}
L(\chi_{D_0},1)} L'\big(\shim_{m_0,\mu_0}(\xi_{3/2-\kappa}f_{m_1, \mu_1}),\kappa\big)
\bigg).
\end{align*}
Using Propositions \ref{prop:Coeff}, \ref{prop:FiniteIntersection} and Lemma \ref{lem:Com} below, we find that
\begin{align*}
\langle Z_\kappa(m_1, \mu_1),  Z_\kappa(U) \rangle_{fin}= m_0 ^{\frac{\kappa-1}{2}}\frac{\deg Z(U)}{4} \cdot
\CT\left(\langle
f_{m_1, \mu_1}^+,\, [\theta_P,\calE_N^+]_\frac{\kappa-1}{2}\rangle\right).
\end{align*}
Hence, we obtain for the global height pairing:
\begin{align*}
 \langle Z_\kappa(m_1, \mu_1), Z_\kappa(U) \rangle
&=m_0 ^{\frac{\kappa-1}{2}}
\frac{\deg( Z(U))C(0,\frac{\kappa-1}{2})}{4(4\pi m_0)^{\frac{\kappa}{2}}
L(\chi_{D_0},1)} L'\big(\shim_{m_0,\mu_0}(\xi_{3/2-\kappa}f_{m_1, \mu_1}),\kappa\big).
\end{align*}
Moreover, a simple calculation shows that
\begin{align*}
\deg Z(U) &=\frac{2 \sqrt{|D_0|}}{\pi} L(\chi_{D_0}, 1),\\
C(0, \tfrac{\kappa-1}{2})&=2(4\pi)^{-\frac{\kappa-1}{2}}\Gamma(\kappa-1/2).
\end{align*}
Inserting these expressions, we obtain the asserted formula.
\end{proof}

\begin{corollary}
\label{cor:higherheight}
Let the notation be as above and assume that $\kappa=1+2j>1$ is an odd integer.  Let $f\in H_{3/2-\kappa,\bar\rho_L}$ with Fourier coefficients $c^\pm(m,\mu)$ and define
\[
Z_\kappa(f) = \sum_{\substack{m>0\\ \mu \in L'/L}} c^+(-m,\mu) Z_\kappa(m,\mu).
\]
If $c^+(-m,\mu)=0$ for all $m>0$ with $(4\lev m,D_0)\neq 1$,
then
\begin{align*}
\langle Z_\kappa(f),  Z_\kappa(U) \rangle
&=\frac{2 \sqrt \lev \Gamma(\kappa-1/2)}{(4\pi)^{\kappa-1}\pi^{3/2} }L'\big(\shim_{m_0,\mu_0}(\xi_{3/2-\kappa}f),\kappa\big).
\end{align*}
\end{corollary}

\begin{proof}
Since $Z_\kappa(f_{m_1,\mu_1})= 2Z_\kappa(m_1,\mu_1)$, the corollary directly follows from Theorem \ref{thm:higherheight} by linearity.
\end{proof}

We now provide the three auxiliary results that were used in the proof of Theorem \ref{thm:higherheight}.
\begin{proposition}
\label{prop:Coeff}
Let the notation be as above
 and assume that $\kappa=1+2j>1$ is odd.
As in  \eqref{eq:calE} let $\kappa(m, \mu)$ denote the $(m, \mu)$-th Fourier coefficient of $\calE_{N}^+(\tau)$.
Then the $(m_1, \mu_1)$-th Fourier coefficient of $[\theta_{P}, \calE_{N}^+]_j$ is equal to
\begin{align*}
\CT\left(\langle
f_{m_1, \mu_1}^+,\, [\theta_P,\calE_N^+]_j\rangle\right)
&=2 m_1^j\sum_{\substack{ n \equiv r_0 r_1 \pod{2\lev} \\ n^2 \le D_0 D_1}}
  \kappa\left(\frac{D_0D_1 -n^2}{4\lev |D_0|}, \frac{\tilde 2 n}{\sqrt{D_0}} \right) \beta_j\left(\frac{n}{\sqrt{D_0D_1}}\right) ,
\end{align*}
where $\kappa(m, \mu)$ denotes the $(m, \mu)$-th Fourier coefficient of $\calE_{N}^+(\tau)$ as in  \eqref{eq:calE}. Moreover,
$\tilde 2$ is an integer with $2 \cdot \tilde 2 \equiv 1 \pmod {D_0}$, and
\begin{align*}
\beta_j(x) &=\sum_{s=0}^j  \begin{pmatrix} j-\frac{1}2 \\ s \end{pmatrix} \begin{pmatrix} j \\ s \end{pmatrix}  x^{2j-2s} (x^2 -1)^{s}.
\end{align*}
\end{proposition}

\begin{proof}
Let $a(n, \nu)$ be the $(n, \nu)$-th Fourier coefficient of the weight $1/2$ theta series  $\theta_{P}(\tau)$. By definition, we have
\begin{align*}
[\theta_{P}, \calE_{N}^+]_j
 &= \sum_{\substack{n_1,n_2\\ \nu_1,\nu_2}} a(n_1, \nu_1) \kappa(n_2, \nu_2) \sum_{s=0}^j (-1)^s \begin{pmatrix} j-\frac{1}2 \\ s \end{pmatrix} \begin{pmatrix} j \\ s \end{pmatrix}
   n_1^{j-s} n_2^s q^{n_1 +n_2} \phi_{\nu_1} \phi_{\nu_2}.
\end{align*}
Now the same argument as in the proof of \cite[Lemma 7.13]{BY} shows that the $(m_1, \mu_1)$-th coefficient of $[\theta_{P}, \calE_{N}^+]_j$ is equal to
\begin{align*}
&\sum_{\substack{n_1+ n_2 =m_1 \\ \nu_1 +\nu_2 \equiv \mu_1 \pmod L}}
  a(n_1, \nu_1) \kappa(n_2, \nu_2) \sum_{s=0}^j  (-1)^s\begin{pmatrix} j-\frac{1}2 \\ s \end{pmatrix} \begin{pmatrix} j \\ s \end{pmatrix}
   n_1^{j-s} n_2^s
   \\
   &=\sum_{\substack{n \equiv r_0 r_1 \pod{2\lev } \\ n^2 \le D_0D_1 }}
     \kappa(\frac{D_0D_1-n^2}{4\lev |D_0|},\frac{ \tilde 2 n }{\sqrt{D_0}})
     \left( \frac{1}{4\lev |D_0|}\right)^j
      \sum_{s=0}^j  \begin{pmatrix} j-\frac{1}2 \\ s \end{pmatrix} \begin{pmatrix} j \\ s \end{pmatrix}
          n^{2j-2s} (n^2-D_0D_1)^s
  \\
   &= m_1^j
 \sum_{\substack{ n \equiv r_0 r_1 \pod{2\lev} \\ n^2 \le D_0 D_1}}
  \kappa(\frac{D_0D_1 -n^2}{4\lev |D_0|}, \frac{\tilde 2 n}{\sqrt{D_0}} ) \beta_j(\frac{n}{\sqrt{D_0D_1}}).
\end{align*}
The $(m_1,-\mu_1)$-th coefficient gives the same contribution.
\end{proof}

\begin{lemma} \label{lem:Com} Let
$$
P_n(x) = \frac{1}{2^n n!} \frac{d^n}{dx^n}(x^2-1)^n =\frac{1}{2^n} \sum_{k=0}^{[\frac{n}2]} (-1)^k \begin{pmatrix} 2n-2k \\ n \end{pmatrix}
   \begin{pmatrix} n \\ k \end{pmatrix}  x^{n-2k}
$$
be the $n$-th Legrendre polynomial. Then
$$
\beta_j(x) = P_{2j}(x).
$$
\end{lemma}
\begin{proof} Since $P_{2j}$ satisfies the differential equation
$$
(1-x)^2P'' -2xP' +2j(2j+1)P=0,
$$
$P_{2j}(x)$ is the unique polynomial
$$
\sum_{m=0}^{j} a_{2m} x^{2m}
$$
satisfying the recursion formula
\begin{equation} \label{eq:Recursion}
\frac{a_{2m+2}}{a_{2m}} =-\frac{(2j-2m)(2j+2m+1)}{(2m+2)(m+1)}, \quad a_{0}=(-1)^j\frac{1}{2^{2j}} \binom{2j}{j}.
\end{equation}
A simple calculation shows that we can write
$$
\beta_j(x) = \sum_{m=0}^j b_{2m} x^{2m},
$$
with
$
b_{2m} =(-1)^{j-m} \sum_{s=j-m}^j x_s^{(m)},
$
and
$$
x_s^{(m)} = \frac{1}{2^{2s}}
 \begin{pmatrix} 2j \\ 2j-2s \end{pmatrix} \begin{pmatrix} 2s \\ s \end{pmatrix} \begin{pmatrix} s \\ j-m \end{pmatrix}
 = \frac{1}{2^{2s}} \frac{(2j)!}{(2j-2s)! s! (s-j+m)! (j-m)!}.
$$
Clearly, $b_0=(-1)^j\frac{1}{2^{2j}} \binom{2j}{j}$. To show that $b_{2m}$ satisfies the recusion formula (\ref{eq:Recursion}),  notice that
\begin{align*}
\frac{x_s^{(m)}}{x_{s+1}^{(m)} }
 &= \frac{4(s+1)(s+1-j+m)}{(2j-2s)(2j-2s-1)},
 \\
 \frac{x_s^{(m+1)}}{x_s^{(m)}}
  &=\frac{j-m}{s+1-j+m},
  \\
  \frac{x_{s-1}^{(m+1)}}{x_s^{(m)}}
    &=\frac{4s(j-m)}{(2j-2s+2)(2j-2s+1)}.
\end{align*}
With these formulas,  we can split
\begin{align*}
x_{s}^{(m+1)}&= x_{s, +}^{(m+1)} + x_{s, -}^{(m+1)},
\\
x_{s, +}^{(m+1)} &= \frac{4(j-m)(s+1)}{(2m+2)(2m+1)} x_{s+1}^{(m)},\quad  j-m-1 \le s \le j-1, \quad  x_{j, +}^{(m+1)}=0,
\\
 x_{s, -}^{(m+1)} &=\frac{(2j-2m)(2j+2m-2s+1)}{(2m+2)(2m+1)} x_s^{(m)}, \quad  j-m \le s \le j, \quad  x_{j-m-1, -}^{(m)}=0.
\end{align*}
Now it is straightforward to verify  the following identities  for  $j-m \le s \le j$:
$$
x_{s-1,+}^{(m+1)} +x_{s, -}^{(m+1)} = \frac{(2j-2m)(2j+2m+1)}{(2m+2)(m+1)}{x_s^{(m)}}.
$$
Notice that the right hand side is independent of $s$. Adding them together,  we see  that $b_{2m}$ satisfies the recursion formula (\ref{eq:Recursion}). This proves that $\beta_j(x) =P_{2j}(x)$.
\end{proof}

We thank Ruixiang Zhang for showing one of us (T.Y.) the proof of the above lemma.

\begin{proposition} \label{prop:FiniteIntersection}
Let the notation be as before. In particular, assume that $D_0$ is fundamental and $(D_0, 2\lev D_1)=1$.
Then
\begin{align*}
&\langle Z_\kappa(m_1, \mu_1),  Z_\kappa(U) \rangle_{fin}
\\
&=(m_0m_1)^{\frac{\kappa-1}2}\frac{\deg (Z(U))}2 \sum_{ \substack{n\equiv r_0 r_1 \pod{2\lev } \\ n^2 \le D_0D_1}} P_{\kappa-1}\left(\frac{n}{\sqrt{D_0D_1}}\right) \kappa\left(\frac{D_0D_1-n^2}{4\lev |D_0|},\frac{\tilde 2 n}{ \sqrt{D_0}}\right).
\end{align*}
\end{proposition}

To prove this proposition, we need some preparations.
Recall that $ D=-4Mm$ is a negative discriminant and $\mu =\mu_r= \diag(\frac{r}{2M},  -\frac{r}{2M})$ where $D \equiv r^2 \pmod{4M}$.
The ideal  $\mathfrak n= [M, \frac{r+\sqrt D}2]$ in the quadratic order $ \OO_D$ of discriminant $D$  has index $M$. It is invertible if the conductor of $\calO_D$ is coprime to $M$.
Following \cite[Section 7.3]{BY},
let  $\mathcal Z(m, \mu)$ be the moduli stack over $\Z$ assigning
to  a scheme $S$ over $\Z$  the groupoid
of  pairs $(\pi: E \rightarrow E',\iota)$ where
\begin{itemize}
\item[(i)]
$\pi: E \rightarrow E'$ is a cyclic isogeny of elliptic curves over $S$
of degree $\lev$,
\item[(ii)]
$\iota: \OO_D \hookrightarrow \End(\pi)
$
is an $\OO_D$-action on $\pi$ such that
$\iota(\mathfrak n) \ker \pi =0$.
\end{itemize}
Then in the complex fiber we have
$
\mathcal Z(m, \mu)_\C  = Z(m,\mu)=P_{D, r} + P_{D, -r}
$
where $P_{D, r}$ is the Heegner divisor defined in \cite[(1) and (2) on page 542]{GKZ}.

\begin{lemma}
\label{lem:hegmod}
Let $c$ be the conductor of the order $\calO_D$ and put $M'=(M,c)$.
\label{lem:moduli}
Let $\bar{ Z}(m, \mu)$ be the Zariski closure of $Z(m, \mu)$ in $\mathcal X_0(M)$. Then we have
$$
\mathcal Z(m, \mu) \cong \bar Z(m, \mu)
$$
as stacks over $\Z[1/M']$.
\end{lemma}

We thank Ben Howard for communicating the following proof to us.

\begin{proof}
We begin by  showing that $\mathcal Z(m, \mu)$ defines a Cartier divisor on $\mathcal X_0(M)$. This can be checked \'etale locally.
Fix a geometric point
$z \to \mathcal{Z}(m,\mu)$ of characteristic $p$, and let
\begin{equation}
\label{eq:localrings}
\widehat{\mathcal{O}}^{et}_{ \mathcal{X}_0(M),z} \to \widehat{\mathcal{O}}^{et}_{ \mathcal{Z}(m,\mu),z }
\end{equation}
be the canonical morphism of completed local rings.  This morphism is surjective, and we need to show that the kernel is a principal ideal.

We first assume that $p$ is coprime to $c$ so that the order $\calO_D$ is maximal at $p$. The completed local ring  $\widehat{\mathcal{O}}^{et}_{ \mathcal{Z}(m,\mu),z }$ classifies deformations of the elliptic curve $E_z$ corresponding to $z$ together with its $\calO_D$-action.   By the Serre-Tate theorem, these are the same as deformations of the $p$-divisible group $E_z[p^\infty]$ together with its action of the maximal $\Z_p$-order $\mathcal{O}_D \otimes \Z_p$.
The classification of such deformations is a special case of \cite[Theorem~2.1.3]{Ho12}.
If $p$ is unramified in $\Q(\sqrt{D})$
it implies that $ \widehat{\mathcal{O}}^{et}_{ \mathcal{Z}(m,\mu),z }$ is isomorphic to the Witt ring $W$  of $\bar \F_p$. If $p$ is ramified in $\Q(\sqrt{D})$
it implies that
\[
\widehat{\mathcal{O}}^{et}_{ \mathcal{Z}(m,\mu),z }
\cong W\otimes_\Z \calO_D.
\]
Hence in both cases $ \widehat{\mathcal{O}}^{et}_{ \mathcal{Z}(m,\mu),z }$ is a discrete valuation ring and \eqref{eq:localrings} is a surjective morphism of regular local rings of dimensions $2$ and $1$. This implies that the kernel is principal, see e.g.~Lemma~10.105.4 of \cite{stacks-project}.

Now assume that $p$ is not coprime to $c$, and write $c=c'p^t$ with $c'$ coprime to $p$ and $t\in \Z_{>0}$.
The hypothesis of the lemma implies $p\nmid M$.
By the Serre-Tate theorem, $R=\widehat{\mathcal{O}}^{et}_{\mathcal{X}_0(M) ,z}$ classifies deformations of the isomorphism
\[
\pi_z : E_z[p^\infty] \to E_z'[p^\infty].
\]
This is equivalent to deformations of $E_z[p^\infty]$ alone, and so $R \cong W[[T]]$ (noncanonically) where $W$ is the Witt ring of $\bar\F_p$.
 The quotient
 \[
 R/I=\widehat{\mathcal{O}}^{et}_{ \mathcal{Z}(m,\mu),z }
 \]
  classifies those deformations of $E_z[p^\infty]$ for which the action of $\mathcal{O}_D$ also deforms.

If $p$ is split in $\Q(\sqrt{D})$ then $E_z$ is ordinary.
For any complete local $W$-algebra $A$, the Serre-Tate coordinates establish a bijection between the lifts of $E_z[p^\infty]$ and the elements of $1+\mathfrak{m}_A$, where $\mathfrak{m}_A\subset A$ is the maximal ideal.  Under this bijection, the lifts of $E_z[p^\infty]$ with its $\mathcal{O}_D\otimes\Z_p$-action correspond to the roots of unity   $\mu_{p^t} \subset 1+\mathfrak{m}_A$.  It follows that  $I$ is the principal ideal generated by
\[
(T+1)^{p^t}-1 .
\]

Now suppose $p$ is nonsplit in $\Q(\sqrt{D})$ so that $E_z$ is supersingular.
In this case, according to \cite[Proposition~5.1]{Wewers},
the deformation locus inside $R$ of any non-scalar endomorphism of $E_z$ is a Cartier divisor. In particular, $I\subset R$ is a principal ideal.

Finally, we note that the Cartier divisor $\calZ(m,\mu)$ cannot have any vertical components, because over any field there are only finitely many isomorphism classes of elliptic curves with complex multiplication by $\calO_D$. Hence $\calZ(m,\mu)$ has to agree with $ \bar Z(m, \mu)$.
\end{proof}

Assume from now on that $D = D_0 = -4Mm_0$ is fundamental and $(D_0, 2M)=1$.
Moreover, recall that $D_0 \equiv r_0^2 \pmod{4M}$.
As in \cite[Section 6]{BY},
let  $\mathcal C$ be the moduli stack over $\Z$ assigning to a scheme $S$ over $\Z$  the groupoid of pairs $(E, \iota)$ where $E$ is an elliptic curve over $S$ and $\iota:\calO_{D}\hookrightarrow \End_S(E)=:\calO_E$ is a homomorphism such that the main involution on $\calO_E$ gives complex conjugation on $\calO_{D}$.

According to \cite[Lemma~7.10]{BY}, we have an isomorphism of stacks
\begin{equation}
j:  \mathcal C \cong \mathcal Z(m, \mu),   \quad (E, \iota) \mapsto (\pi:E \rightarrow E_{\mathfrak n}=E/E[\iota(\mathfrak n)], \iota).
\end{equation}
Moreover, this map gives rise to a closed immersion (assuming $M>1$ without loss of generality)
\begin{equation}
j: \mathcal C \longrightarrow \mathcal X_0(M).
\end{equation}

Now let $D_1$ be a discriminant which is coprime to $D_0$ as before.
Throughout, we write $\mathfrak n_i = [M, \frac{r_i+\sqrt D_i}2]$
for the ideal in $\calO_{D_i}$ of index $M$ corresponding to $D_i$ and $r_i$.
It is then easy to see that
$j^*\mathcal Z(m_1,  \mu_1) =\mathcal Z(m_0, \mu_0)\times_{\mathcal X_0(M)} \mathcal Z(m_1, \mu_1) $ is the intersection of  $\mathcal Z(m_0, \mu_0)$ and $\mathcal Z(m_1, \mu_1)$. It represents
triples $(E, \iota, \phi)$ where $(E, \iota) \in \mathcal C$ and
\begin{equation}
\phi: \OO_{D_1} \rightarrow \OO_{E, \mathfrak n_0} =\End (E\rightarrow E_{\mathfrak n_0})
\end{equation}
such that $\phi(\mathfrak n_1) E[\mathfrak n_0] =0$.
The intersection is a stack of dimension zero which is supported in finitely many closed fibers.
Moreover, Lemma \ref{lem:hegmod} implies that
\begin{equation}\label{eq:zariski-pullback}
j^*\mathcal Z(m_1, \mu_1) \cong \bar Z(m_0, \mu_0) \times_{\mathcal X_0(M)} \bar Z(m_1, \mu_1) .
\end{equation}
Indeed, it suffices to check this in the fiber above $p$ for every prime $p$. If $p\nmid M$ it is an immediate consequence of the lemma.
On the other hand, if $p\mid M$ then $p$ is split in $\Q(\sqrt{D_0})$, and hence both sides of \eqref{eq:zariski-pullback} vanish. Indeed, the points $(E, \iota) \in \mathcal C(\bar{\mathbb F}_p)$ are given by ordinary elliptic curves, which do not admit additional complex multiplication by $\calO_{D_1}$.

 To describe the intersection further, we recall the special cycles $\mathcal  Z(m, \mathfrak a, \mu)$ in $\mathcal C$ defined in \cite[Section 6]{BY} (see also
\cite[Section 2.4]{KY13}), where $\mathfrak a$ is an ideal of $\OO_{D_0}$ with quadratic form $Q(x) = -x \bar x/\norm(\mathfrak a)$, $\mu \in  \frac{1}{\sqrt{D_0}} \mathfrak a/\mathfrak a$, and $Q(\mu) \equiv m \pmod 1$. It assigns to every
scheme $S$  the groupoid of triples
$(E, \iota, \boldbeta)$ where
\begin{itemize}
\item[(i)]
$(E, \iota) \in \mathcal C(S)$ and
\item[(ii)]
$\boldbeta \in L(E, \iota)\frac{1}{\sqrt{D_0}}\mathfrak a$ such that
$\norm(\boldbeta) =m \norm(\mathfrak a)$ and $\mu + \boldbeta \in \OO_E
\mathfrak a$.
\end{itemize}
Here
$$
L(E, \iota) = \{ x \in \OO_E:\;  \iota(\alpha)  x = x
\iota(\bar\alpha), \alpha\in \OO_{D_0}\}
$$
is the lattice of special endomorphisms of $(E, \iota)$.

\begin{lemma} \label{lem:pullback} Let the notation and assumption be as above. Then there is an isomorphism
$$
j^* \mathcal Z(m_1,\mu_1) \cong  \bigsqcup_{ \substack{n\equiv r_0 r_1 \pod {2\lev } \\ n^2 \le D_0D_1}}
   \mathcal Z\left(\frac{D_0D_1-n^2}{4\lev |D_0|}, \mathfrak n_0,  \frac{n+r_1 \sqrt{D_0}}{2 \sqrt{D_0}}\right),
 \quad   (E, \iota, \phi) \mapsto (E, \iota, \boldbeta)
$$
with
$$
  2n =  \phi(\sqrt{D_1}) \iota(\sqrt{D_0}) + \iota( \sqrt{D_0}) \phi(\sqrt{D_1})
$$
and
$$
\boldbeta  = \phi\left(\frac{r_1 +\sqrt{D_1}}2\right) - \frac{n+r_1 \sqrt{D_0}}{2 \sqrt{D_0}}.
$$
\end{lemma}
\begin{proof}
The lemma is proved in \cite[Lemma 7.12]{BY} on $\bar{\mathbb F}_p$-points, the proof goes through in general.
Tracing back the proof, we obtain the stated formula for $2n$.
\end{proof}

\begin{proof}[Proof of  Proposition \ref{prop:FiniteIntersection}.]
Since $D_0$ is fundamental and coprime to $M$ we have $Z_\kappa(U)=Z_\kappa(m_0,\mu_0)$.
To compute the local contribution at $p$ to $\langle S_\kappa( x_1), S_\kappa( x_0) \rangle_{fin}$,
let $x_i \in Z(m_i, \mu_i)$ and denote by $\bar{x}_i$ the Zariski closure of $x_i$ in $\calX_0(M)$.
Let $(E, \iota, \phi) \in j^*\calZ(m_1, \mu_1)(\bar{\F}_p)$
correspond to the intersection of the divisors $\bar{x}_1$ and $\bar{x}_0$ in the fiber above $p$
and put for convenience $n = n((E, \iota, \phi)) = \left(\phi(\sqrt{D_1}) \iota(\sqrt{D_0}) + \iota( \sqrt{D_0}) \phi(\sqrt{D_1})\right) / 2$.
Then $\calO_{E, \frakn_0}$ contains the Clifford order $S_{[D_0, 2n, D_1]}$ defined in \cite[Chapter I.3]{GKZ}.
By means of \cite[Proposition~3.1]{Xue} we obtain
\begin{equation}
\label{eq:intersection}
\langle S_\kappa(x_1), S_\kappa( x_0) \rangle_{p} = (-1)^\kappa
( S_\kappa( \bar x_1)\cdot  S_\kappa( \bar x_0) )_{p}
= - P_{\kappa-1}\left(\frac{n}{\sqrt{D_0D_1}}\right) (\bar{x}_1 \cdot \bar{x}_0)_p.
\end{equation}
Using the identification \eqref{eq:zariski-pullback} together with \eqref{eq:intersection}, we find
\begin{align*}
&(m_0m_1)^{\frac{1-\kappa}2}\langle Z_\kappa(m_1, \mu_1), Z_\kappa(m_0, \mu_0) \rangle_{fin}  \\
&= - \sum_{p < \infty} \sum_{x \in j^*\calZ(m_1,\mu_1)(\bar{\F}_p)} P_{\kappa-1}\left( \frac{n(x) }{\sqrt{D_0D_1}}\right)\frac{i_p(x)}{|\Aut(x)|} \log p,
\end{align*}
where $i_p(x) = i_p(j^*\calZ(m_1,\mu_1), x)$ is the length of the local ring of $j^*\calZ(m_1,\mu_1)$ at $x$.
By virtue of Lemma \ref{lem:pullback} and \cite[Theorem 6.4]{BY}, we obtain
 \begin{align*}
  &(m_0m_1)^{\frac{1-\kappa}2}\langle Z_\kappa(m_1, \mu_1), Z_\kappa(U) \rangle_{fin}
   \\
   &= -\sum_{ \substack{n\equiv r_0 r_1 \pod{2\lev} \\ n^2 \le D_0D_1}}  P_{\kappa-1}\left(\frac{n}{\sqrt{D_0D_1}}\right)\cdot
     \widehat{\deg}\left(\mathcal Z\left(\frac{D_0D_1-n^2}{4\lev |D_0|}, \mathfrak n_0,  \frac{n+r_1 \sqrt{D_0}}{2 \sqrt{D_0}}\right)\right)
   \\
   &= \frac{\deg (Z(U))}2 \sum_{ \substack{n\equiv r_0 r_1 \pod{2\lev} \\ n^2 \le D_0D_1}}  P_{\kappa-1}\left(\frac{n}{\sqrt{D_0D_1}}\right) \cdot \kappa\left(\frac{D_0D_1-n^2}{4\lev |D_0|},\frac{n+r_1 \sqrt{D_0}}{2 \sqrt{D_0}}\right).
\end{align*}
Since $\frac{n+r_1\sqrt{D_0}}{2\sqrt{D_0}} \equiv \frac{\tilde 2 n}{\sqrt{D_0}} \pmod{\mathcal O_{D_0}}$, we have proved the  proposition.
\end{proof}

As before, let $\kappa=1+2j>1$ be an odd integer. We consider the generating series
\begin{align}
A_\kappa(\tau,U) &= \sum_{m,\mu}
\langle Z_\kappa(m,\mu), Z_\kappa(U)\rangle \cdot q^m\phi_\mu.
\end{align}
In analogy with the Gross-Kohnen-Zagier theorem \cite{GKZ} it is expected that $A_\kappa(\tau,U)$ is a cusp form in $S_{\kappa+1/2,\rho_L}$, or equivalently a cuspidal Jacobi form of weight $\kappa+1$ and index $\lev$ for the full Jacobi group. Note that the height pairings $\langle Z_\kappa(m,\mu), Z_\kappa(U)\rangle$ may involve improper intersections of higher Heegner cycles on Kuga-Sato varieties when $(4\lev m,D_0)\neq 1$, a technical problem which we do not consider in the present paper. Here we prove the following version of the Gross-Kohnen-Zagier theorem.

\begin{theorem}
\label{thm:modat}
Assume the above notation. In particular, let $D_0$ be a fundamental discriminant which is coprime to $2M$. There is a cusp form
$g=\sum_{m,\mu} b(m,\mu) q^m\phi_\mu$ in $S_{\kappa+1/2,\rho_L}$ whose Fourier coefficients $b(m,\mu)$ satisfy
\[
b(m,\mu) =   \langle Z_\kappa(m,\mu), Z_\kappa(U)\rangle
\]
for all $\mu\in L'/L$ and $m\in Q(\mu)+\Z$ with $(4\lev m,D_0)= 1$.
\end{theorem}

\begin{proof}
The theorem is a direct consequence of Corollary \ref{cor:higherheight} and the modularity criterion in Proposition \ref{prop:mod} below.
\end{proof}

By means of Lemma \ref{lem:adm0} below, we obtain the following consequence.

\begin{corollary}
\label{cor:modat}
The generating series
\begin{align}
\tilde A_\kappa(\tau,U) &= \sum_{\substack{m,\mu\\ (4\lev m,D_0)=1}}
\langle Z_\kappa(m,\mu), Z_\kappa(U)\rangle \cdot q^m\phi_\mu
\end{align}
belongs to
$S_{\kappa+1/2,\rho_L}(\Gamma_0(D_0^2))$.
\end{corollary}

\begin{lemma}
\label{lem:adm0}
Let $g=\sum_{m,\mu} b(m,\mu) q^m\phi_\mu\in S_{\kappa+1/2,\rho_L}$, and let $R\in \Z_{>0}$.
Then
\[
\tilde g= \sum_{\substack{m,\mu\\ (4\lev m,R)=1}} b(m,\mu) q^m\phi_\mu
\]
belongs to $S_{\kappa+1/2,\rho_L}(\Gamma_0(R^2))$.
\end{lemma}

\begin{proof}
Using the isomorphism between $S_{\kappa+1/2,\rho_L}$ and Jacobi forms of weight $\kappa+1$ and index $\lev$, the assertion follows from
\cite[Lemma 2.4]{Schwagenscheidt}.
\end{proof}

We now turn to the modularity criterion required for the proof of Theorem \ref{thm:modat}.
We start with the following lemma.

\begin{lemma}
\label{lem:adm1}
Let $g=\sum_{m,\mu} b(m,\mu) q^m\phi_\mu\in S_{\kappa+1/2,\rho_L}$, and let $D_0$ be a discriminant which is coprime to $\lev$. If $b(m,\mu)=0$ for all $\mu\in L'/L$ and  $m\in Q(\mu)+\Z$ with $(4\lev m,D_0)=1$, then $g=0$.
\end{lemma}

\begin{proof}
This can be proved in the same way as \cite[Proposition 3.1]{Schwagenscheidt}.
\end{proof}

Let $D_0$ be a discriminant which is coprime to $\lev$. We call a harmonic Maass form $f\in H_{3/2-\kappa,\bar\rho_L}$ with Fourier coefficients $c^\pm(m,\mu)$ \emph{admissible} for $D_0$ if
$c^+(-m,\mu)=0$ for all $m>0$ with $(4\lev m,D_0)\neq 1$.

It is a consequence of Lemma \ref{lem:adm1} that for every $g\in S_{\kappa+1/2,\rho_L}$ there exists an $f\in H_{3/2-\kappa,\bar\rho_L}$ which is admissible for $D_0$ such that $\xi(f)=g$.

\begin{proposition}
\label{prop:mod}
Let  $D_0$ be a discriminant which is coprime to $\lev $. Let
\[
h=\sum_{\mu\in L'/L}\sum_{\substack{m\in Q(\mu)+\Z\\ m> 0}} a(m,\mu) q^m\phi_\mu
\]
be a $\C[L'/L]$-valued formal $q$-series satisfying $a(m,\mu)=a(m,-\mu)$ for all $(m,\mu)$. If
\[\CT(\langle f,h\rangle) =0\]
for every $f\in M^!_{3/2-\kappa,\bar\rho_L}$ which is admissible for $D_0$, then there is a
$g=\sum_{m,\mu} b(m,\mu) q^m\phi_\mu$ in $S_{\kappa+1/2,\rho_L}$ whose Fourier coefficients $b(m,\mu)$ satisfy
\[
b(m,\mu) =   a(m,\mu)
\]
for all $\mu\in L'/L$ and $m\in Q(\mu)+\Z$ with $(4\lev m,D_0)= 1$.
\end{proposition}

\begin{proof}
For $\mu\in L'/L$ and  $m\in Q(\mu)+\Z$ positive, we define coefficients $b(m,\mu)$ as follows:
If $(4\lev m,D_0)\neq 1$ we consider the harmonic Maass form $f_{m,\mu}=q^{-m}(\phi_\mu+\phi_{-\mu})+O(1)$ in $H_{3/2-\kappa, \bar\rho_L}$.
Choose an $f_{m,\mu}'\in H_{3/2-\kappa, \bar\rho_L}$ which is admissible for $D_0$ such that
\begin{align}
\label{eq:xiid}
\xi(f_{m,\mu}')=\xi(f_{m,\mu}).
\end{align}
If $(4\lev m,D_0)=1$, then $f_{m,\mu}$ is already admissible for $D_0$. In this case we simply put $f'_{m,\mu}=f_{m,\mu}$.
In both cases we define
\[
b(m,\mu)=\frac{1}{2}\CT(\langle f_{m,\mu}^{\prime,+},h\rangle),
\]
where $f_{m,\mu}^{\prime,+}$ denotes the holomorphic part of $f_{m,\mu}^{\prime}$.
Note that we have $b(m,\mu)=a(m,\mu)$ when $(4\lev m,D_0)=1$.

Now the generating series
\[
g=\sum_{\mu\in L'/L}\sum_{\substack{m\in Q(\mu)+\Z\\ m>0}} b(m,\mu) q^m\phi_\mu
\]
satisfies $\CT(\langle f,g\rangle) =0$ for {\em every} $f\in M^!_{3/2-\kappa,\bar\rho_L}$.
 In fact, to see this we denote the Fourier coefficients of $f$ by $c(n,\nu)$ and write
 \begin{align*}
 f&= \frac{1}{2}\sum_{\substack{\mu\in L'/L\\m>0}}
 c(-m,\mu)f_{m,\mu}\\
 &= \frac{1}{2}\sum_{\substack{\mu\in L'/L\\m>0}}
 c(-m,\mu)f'_{m,\mu} +\frac{1}{2}\sum_{\substack{\mu\in L'/L\\m>0}}
 c(-m,\mu)(f_{m,\mu}-f_{m,\mu}').
 \end{align*}
 The second sum on the right hand side is weakly holomorphic because of \eqref{eq:xiid}. Hence the first sum also has to be weakly holomorphic. Since the first sum is in addition admissible for $D_0$, we find by the hypothesis that
 \begin{align*}
 \CT(\langle f,g\rangle) &= \frac{1}{2}\sum_{\substack{\mu\in L'/L\\m>0}}
 c(-m,\mu)\CT(\langle f_{m,\mu}^+-f_{m,\mu}^{\prime,+},g\rangle).
 \end{align*}
 Because the $f_{m,\mu}'$ are admissible for $D_0$, we have
 \begin{align*}
 \CT(\langle f_{m,\mu}^{\prime,+},g\rangle)=\CT(\langle f_{m,\mu}^{\prime,+},h\rangle)= 2b(m,\mu)= \CT(\langle f_{m,\mu}^+,g\rangle).
 \end{align*}
 Consequently, $\CT(\langle f,g\rangle)=0$.
Therefore Borcherds'  modularity criterion \cite[Theorem~3.1]{Bo:Duke} implies that $g\in S_{\kappa+1/2,\rho_L}$.
\end{proof}

\section{Partial averages}
\label{sec:partial-averages}
In this section, we will use the higher automorphic Green functions for $\SO(1,2)$
to evaluate certain partial averages of the resolvent kernel for $\SL_2(\Z)$ at positive integral spectral parameter as considered in Section \ref{sect:resolvent}.
To this end we employ and generalize and Theorem \ref{thm:fund2}
and Theorem \ref{thm:fund12} of Section \ref{sect:6.2}.
Throughout this section, we let $V$, $L$, and $K$ be as in Section \ref{sect:6.2}
but we restrict to level $1$ for simplicity. Thus, $X_K$ is isomorphic
to the modular curve $\SL_2(\Z) \bs \H$.

The general idea of this section is to fix a  fundamental discriminant $d_1$
and to consider the \emph{partial average}
\[
	G_{j+1, f} (C(d_1), z_2),
\]
where we use the same notation as in the Introduction.
We shall prove that at any CM point $z_2$ of discriminant $d_2$
the CM value $G_{j+1, f} (C(d_1), z_2)$ is equal to $(d_1 d_2)^{\frac{j-1}{2}} \log|\alpha|$ for some $\alpha \in \bar{\Q}$ (see Corollary \ref{cor:partial-average-algebraic-general} below).
This result proves Conjecture \ref{conj:alg} in the case when the
class group of $\Q(d_1)$ is trivial.

\subsection{Twisted special divisors}
To obtain a stronger result,
and to make this approach work at all for odd $j$ as well,
we also consider twisted partial averages, which we will now define.

\begin{definition}\label{def:chi}
	Let $\Delta \in \Z$ be a fundamental discriminant and $p$ be a prime.
	For $\lambda \in V(\Q_p)$,
	we put $\chi_{\Delta, p}(\lambda) = 0$ unless we have
	$\lambda = \kzxz{b/2}{-a}{c}{-b/2} \in L_p' := L' \otimes_\Z \Z_p$ with
	$4Q(\lambda) \in \Delta\Z_p$. In the latter case, we let
	\[
		\chi_{\Delta, p}(\lambda) :=
			\begin{cases}
				(n, \Delta)_p, &\text{ if } \gcd(a,b,c,\Delta) = 1 \text{ and } [a,b,c] \text{ represents } n \text{ with } \gcd(\Delta, n) = 1, \\
				0, &\text{otherwise.}
			\end{cases}
	\]
	Here, $(a,b)_p$ denotes the $p$-adic Hilbert symbol.
	Finally, for $\lambda = (\lambda_p)_p \in V(\A_f)$, we put
	\[
		\chi_\Delta(\lambda) = \prod_{p < \infty} \chi_{\Delta, p}(\lambda_p).
	\]
\end{definition}

The following lemma is a local variant of
\cite[I.2, Proposition 1]{GKZ} and we leave the
(completely analogous) proof to the reader.

\begin{lemma}\label{lem:chiform}
    The function $\chi_{\Delta, p}$ is well-defined.
    Moreover, for $p \mid \Delta$ and $\lambda = \kzxz{b/2}{-a}{c}{-b/2} \in L_p'$
	with $4Q(\lambda) \in \Delta\Z_p$, we have the following explicit formula:
	\[
		\chi_{\Delta, p}(\lambda) =
		\begin{cases}
			(a, \Delta)_p, &\text{if } p \nmid a,\\
			(c, \Delta)_p, &\text{if } p \nmid c,\\
			0 &\text{otherwise.}
		\end{cases}
	\]
\end{lemma}

Using this formula it is easy to see that
for $\lambda \in L'$, we have
$\chi_\Delta(\lambda) = 0$ if $\Delta \nmid b^2-4ac$. Moreover, when
$\Delta \mid b^2-4ac$, we obtain
\[
  \chi_\Delta(\lambda) =
  \begin{cases}
    \legendre{\Delta}{n} &\text{ if } \gcd(a,b,c,\Delta) = 1 \text{ and } [a,b,c] \text{ represents } n \text{ with } \gcd(\Delta,n)=1\\
     0              &\text{ otherwise. }
  \end{cases}
\]
This shows that on $L$, our definition of $\chi_\Delta$ agrees with
the (generalized) genus character as in \cite{GKZ, BrO}.
In the following lemma, we write $h \cdot \lambda = h \lambda h^{-1}$ for
the action of $h \in \GSpin \cong \GL_2$ on $\lambda \in V$.

\begin{lemma}\label{lem:chitrafo}
	For $h \in \GSpin(L_p) \cong \GL_2(\Z_p)$ and $\lambda \in L_p'$, we have
	\[
		\chi_{\Delta, p}(h \cdot \lambda) = (\det(h), \Delta)_p \cdot \chi_{\Delta, p}(\lambda).
	\]
\end{lemma}
\begin{proof}
	Let $S=\kzxz {0} {1} {-1} {0}$ and put $\tilde{L} = LS$.
	For $\lambda \in L_p'$, we write and $\tilde{\lambda} = \lambda S$.
	Then $\tilde{\lambda}$ is symmetric and $Q(\tilde{\lambda}) = Q(\lambda)$.
    The group $\GL_2(\Z_p)$ acts on $\tilde{L}$ via $h \cdot \tilde\lambda = h  \tilde\lambda h^t$
    for $h \in \GL_2(\Z_p)$.
	Moreover, we have
	$$
	h\cdot \tilde\lambda  = (\det h)\, \widetilde{h\cdot \lambda}.
	$$
	It is clear that $n$ is represented by the quadratic form associated to $\tilde\lambda$  (i.e., $[a, b, c]$) if and only if $n$ is repesented by the form associated to $h\cdot \tilde\lambda$ for any $h \in \GL_2(\Z_p)$.
	This implies the statement of the lemma. \qedhere
\end{proof}

\begin{definition}
For $\lambda = \kzxz{b/2}{-a}{c}{-b/2} \in L'$ with $Q(\lambda)>0$, we write $\lambda>0$ if
\[
\left(\lambda,\zxz{0}{0}{1}{0}\right)>0.
\]
This is the case if and only if the corresponding binary quadratic form $[a,b,c]$ is {\em positive} definite.
\end{definition}

We let
\[
	K_\Delta := \{ h \in K\ \mid\ (\det(h), \Delta)_{\A_f} = 1 \}.
\]
By Lemma \ref{lem:chitrafo}, the function $\chi_\Delta: V(\A_f) \to \{\pm 1\}$ is $K_\Delta$-invariant.
For $\Delta \neq 1$, we have $K_\Delta \neq K = \GSpin(\hat{L})$, and then
$K_\Delta$ has index $2$ in $K$. Now assume that $\Delta \neq 1$.
Since $H(\A_f) = H(\Q)^+K$, the  Shimura variety $X_{K_\Delta}$ then has two connected components
which are both isomorphic to $X_K \cong \SL_2(\Z) \bs \H$.

To describe this isomorphism, let $\Gamma_1 = K_\Delta \cap H(\Q)^+ = \SL_2(\Z)$,
choose $\xi \in K$ such that $\xi \not\in K_\Delta$ and put $\Gamma_{\xi} = (\xi K_\Delta \xi^{-1}) \cap H(\Q)^+ = \SL_2(\Z)$.
Then $(\det(\xi), \Delta)_{\A_f} = -1$ and
\[
	H(\A_f) = H(\Q)^+ K_\Delta \sqcup H(\Q)^+\xi K_\Delta.
\]
We obtain an isomorphism
\[
	\Gamma_{1} \bs \calD^+ \sqcup \Gamma_{\xi} \bs \calD^+ \to X_{K_\Delta}
	= H(\Q) \bs (\calD \times H(\A_f) / K_{\Delta}),
\]
$\Gamma_1 z \mapsto H(\Q)(z,1)K_\Delta$ and $\Gamma_\xi z \mapsto H(\Q)(z,\xi)K_\Delta$.

For $h \in H(\A_f)$, we denote by $c(\lambda, h)$ the ``connected'' cycle \cite{Ku:Duke}
corresponding to $\lambda$ on the component corresponding to the class of $h$.
That is, we obtain the cycle (which is really just a weighted point in our case) $c(\lambda, h)$
as the image of
\[
	\Gamma_{h, \lambda} \bs \calD_\lambda^+ \to \Gamma_h \bs \calD^+, \quad \Gamma_{h, \lambda} z \to \Gamma_h z
\]
where $\Gamma_{h, \lambda} \subset \Gamma_h$ is the stabilizer of $\lambda$
in $\Gamma_h$ and $\calD_\lambda^+$ is the unique point in $\calD^+$, such that $(\calD_\lambda^+, \lambda) = 0$.
Each point in the image is weighted by $2/|\Gamma_{h, \lambda}|$.
We keep the same notation for $\Delta=1$, where $X_{K_\Delta} = X_K$ has only one connected component.

\begin{definition}
We define the following twisted divisors on $X_{K_\Delta}$ (cf.~also \cite{BrO}).
Let $r \in \Z$ with $\Delta \equiv r^2 \bmod{4}$
and let $\mu \in L'/L$ with $\sgn(\Delta)Q(\mu) \equiv m \bmod{\Z}$. Define
\begin{equation}
  \label{eq:ZD}
  Z_\Delta(m, h) := \sum_{\substack{\lambda \in V(\Q) \bmod{\Gamma_h} \\ Q(\lambda)
               = |\Delta|m \\ \lambda > 0}} \chi_\Delta(h^{-1}\lambda) \phi_{r\mu}(h^{-1}\lambda) c(\lambda, h).
\end{equation}
\end{definition}
Note that since we restricted to the level one case,
$\mu \in L'/L$ is uniquely determined by the condition
$\sgn(\Delta)Q(\mu) \equiv m \bmod{\Z}$
and therefore we dropped $\mu$ from the notation.
Also note that by definition, the cycle $Z_\Delta(m,h)$ is supported on the connected component
corresponding to the class of $h$ in $H(\Q) \bs H(\A_f)/K_\Delta$.

\begin{remark}
  For $\Delta=1$, we have $Z_1(m,h) = C(D) = P_{D,s}$,
  where $D=-4m$, $s^2 \equiv D \bmod{4}$ and $P_{D,s}$ is the
  Heegner divisor defined in \cite{GKZ}.
\end{remark}

\subsection{Twisted Siegel and Millson theta functions}
For the partial averages in the case of odd $j$, we also need the Millson theta function.
For
\[
   \lambda = \zxz{b/2}{-a}{c}{-b/2} \in V(\R)
\]
and $z \in \H \sqcup \bar{\H} \cong \calD$ we let
\begin{equation}
  \label{eq:pzV}
  p_{z}(\lambda) = -2^{-\frac{1}{2}}(\lambda, X_1(z)) = \frac{-1}{2y}(c|z|^2 - bx + a),
\end{equation}
where
\[
	X_1(z) := \frac{1}{\sqrt{2}y} \zxz{-x}{x^2+y^2}{-1}{x}
\]
is a normalized (i.e., $(X_1(z), X_1(z)) = 1$)
generator of the positive line $X(z)^\perp$, and $X(z)$ is defined in \eqref{eq:Xz12}.
Now suppose that $M$ is an even lattice
of signature $(1,2)$, and fix an isometric embedding $\sigma: M \otimes \R \to V(\R)$.
Then we define
\begin{align}\label{eq:millson1}
     \theta_{M}^\calM(\tau, z, h) = v \sum_{\mu \in M'/M} \sum_{\lambda \in V(\Q)} p_z(\sigma(\lambda)) \phi_{\mu}(h^{-1}\lambda) e\left(Q(\lambda_{z^\perp}) \tau + Q(\lambda_z) \bar\tau\right) \phi_\mu,
\end{align}
for $\tau \in \H, z \in \calD$, and $h \in H(\A_f)$.
The Millson theta function has weight $1/2$ in $\tau$ and transforms with the representation $\rho_M$.

For the twisted partial averages, we also need twisted variants
of the Siegel and the Millson theta functions. Let $r \in \Z$ with $\Delta \equiv r^2 \pmod{4}$.
If $M$ is any lattice with quadratic form $Q$,
we write $M_\Delta$ for the
rescaled lattice $M_\Delta = \Delta M$ with the quadratic form $Q_\Delta=\frac{Q}{|\Delta|}$.
Note that we have $M_{\Delta}' = M'$ and thus $M_\Delta'/M_\Delta = M'/\Delta M$.
Following \cite{Alfes:2013ff}, we let
\begin{equation} \label{eq:psi}
	\psi_{\Delta}: S_L \to S_{L_\Delta}, \quad \phi_\mu \mapsto \sum_{\substack{\delta \in L'/\Delta L \\ Q_\Delta(\delta) \equiv \sgn(\Delta) Q(\mu)\, (\Z) \\ \delta \equiv r\mu\, (L)}} \chi_\Delta(\delta) \phi_\delta.
\end{equation}
 If $\Delta>0$, this map is an intertwining operator for the Weil representation
  $\rho_L$ on $S_L$ and $\rho_{L_\Delta}$ on $S_{L_\Delta}$. If $\Delta<0$, it intertwines
  $\bar\rho_L$ on $S_L$ and $\rho_{L_\Delta}$ on $S_{L_\Delta}$ (see \cite[Proposition 3.2]{Alfes:2013ff} and \cite[Proposition 4.2]{BrO}).

The twisted Siegel theta function for the lattice $L$ is defined as
\[
	\theta_{L,\Delta}(\tau, z, h) := \sum_{\mu \in L'/L} \langle \theta_{L_\Delta}(\tau, z, h), \psi_\Delta(\phi_\mu) \rangle \phi_\mu.
\]
By the intertwining property of $\psi_\Delta$, it transforms as
a vector valued modular form of weight $-1/2$ in $\tau$
for the Weil representation $\rho_L$ if $\Delta > 0$ and for
$\bar\rho_L$ if $\Delta<0$.
Explicitly, we have
      \begin{equation*}
        \label{eq:Siegeltheta12}
        \theta_{L,\Delta}(\tau,z, h) = v \sum_{\mu \in L'/L}\!\!\!\! \sum_{\substack{\lambda \in V(\Q) \\ Q_\Delta(\lambda) \equiv \sgn(\Delta)Q(\mu)\, (\Z)}}\!\!\!\! \chi_\Delta(h^{-1}\lambda)  \phi_{r\mu}(h^{-1}\lambda) e\left(\frac{Q(\lambda_{z^\perp})}{|\Delta|} \tau + \frac{Q(\lambda_z)}{|\Delta|} \bar\tau\right) \phi_\mu.
      \end{equation*}
By Lemma \ref{lem:chitrafo} it is easy to see that $h \mapsto \theta_{L,\Delta}(\tau,z, h)$ is invariant under $K_\Delta$.

To define the twisted Millson theta function, we embed
$(L_\Delta, Q_\Delta)$ isometrically into $V(\R)$ via $\sigma(\lambda) = \frac{1}{\sqrt{|\Delta|}}\lambda$ and let $\theta_{L_\Delta}^\calM(\tau,z)$ be as in \eqref{eq:millson1}.
We then define the twisted Millson theta function as
\[
	\theta_{L,\Delta}^\calM(\tau, z, h) =
	          \sum_{\mu \in L'/L} \langle \theta_{L_\Delta}^\calM(\tau, z, h), \psi_\Delta(\phi_\mu) \rangle \phi_\mu.
\]
It transforms of weight $1/2$ with $\bar\rho_L$ if $\Delta<0$ and with $\rho_L$ if $\Delta > 0$
and is also $K_\Delta$-invariant.

\begin{remark}
 Note that in our level $1$ case, the function
 $\theta_{L,\Delta}$ vanishes if $\Delta < 0$.
 Similarly, $\theta^\calM_{L, \Delta} = 0$ if $\Delta > 0$.
 For higher level, this is not the case.
 Moreover, note that for $\Delta = 1$,
 the map $\psi_\Delta$ is the identity on $L'/L$.
\end{remark}

A straightforward calculation shows that
the theta lift against any of the twisted
theta functions can in fact be obtained by twisting the input function
using $\psi_\Delta$, which simplifies many calculations.
\begin{lemma} \label{lem:twisting}
For any $v \in \C[L'/L]$, we have
	\[
		\langle v, \theta_{L,\Delta}(\tau, z, h) \rangle = \langle \psi_\Delta(v), \theta_{L_\Delta}(\tau, z, h) \rangle.
	\]
	The analogous formula holds with $\theta_{L,\Delta}$ replaced by $\theta_{L,\Delta}^\calM$
	and $\theta_{L_\Delta}$ replaced by $\theta_{L_\Delta}^\calM$.
\end{lemma}

\subsection{Twisted theta lifts}
Let $j$ be an \emph{even} positive integer, and let $\Delta >0$ be a fundamental discriminant. For $f\in H_{1/2-j,\bar \rho_L}$ we consider the twisted theta lift
\begin{equation}\label{eq:Phijtilde-even}
   \tilde\Phi^j_{\Delta}(z,h, f) := \frac{1}{(4\pi)^{j/2}} \int_{\calF}^\reg \langle R_{1/2-j}^{j/2}f(\tau), \theta_{L, \Delta}(\tau,z,h)\rangle\, d\mu(\tau).
\end{equation}
Let  $f_m\in H_{1/2-j,\bar\rho_L}$ be the unique harmonic Maass form whose Fourier expansion starts with $f_m=q^m\phi_\mu+O(1)$ as $v\to\infty$.
We now identify the theta lift with a twisted partial average of the higher Green function.
Fix $h \in H(\A_f)$.
We then identify the connected component $\Gamma_h \bs \calD^+$ of $X_{K_\Delta}$
with $\SL_2(\Z) \bs \H$.
The divisor $Z_\Delta(m, h)$ is supported on this component and corresponds to the divisor
$(\det(h), \Delta)_{\A_f} Z_\Delta(m)$ on $\SL_2(\Z) \bs \H$, where
\begin{equation}\label{eq:ZDelta-classical}
  Z_\Delta(m) = \sum_{\substack{\lambda \in L' \bmod \SL_2(\Z)\\ Q(\lambda) = m |\Delta| \\ \lambda > 0}} \frac{2}{w(\lambda)}z_\lambda,
\end{equation}
$z_\lambda \in \H$ is the CM point corresponding to $\lambda$,
and $w(\lambda)$ is the order of the stabilizer of $\lambda$
in $\SL_2(\Z)$.
We denote by $G_{1+j}(Z_\Delta(m,h), z)$ the function on $\Sl_2(\Z) \bs \H$
defined by
\[
	G_{1+j}(Z_\Delta(m,h), z) =
	(\det(h), \Delta)_{\A_f} G_{1+j}(Z_\Delta(m), z).
\]

We can evaluate this function at points $(z,h)$ that lie on the connected component corresponding to $h$
as follows: write $h = \gamma h_0 k$ with $\gamma \in H(\Q), k \in K_\Delta$, such that $\gamma^{-1}z \in \calD^+ \cong \H$
and $h_0 = 1$ or $h_0 = \xi$ and put
\[
	G_{1+j}(Z_\Delta(m, h), (z,h)) = G_{1+j}(Z_\Delta(m, h), \gamma^{-1}z)
	= (\det(h), \Delta)_{\A_f}  G_{1+j}(Z_\Delta(m), \gamma^{-1}z).
\]

Using Lemma \ref{lem:chitrafo},
it is straightforward to check that the analogous identity to \eqref{eq:go12} holds:
\begin{equation}\label{eq:go12-twisted}
  \tilde\Phi^{j}_{\Delta}(z, h, f_m) = -2m^{j/2} G_{1+j}(Z_\Delta(m,h), (z,h)),
\end{equation}
where the additional factor $2$ is a result of the condition $\lambda > 0$ in the definition of the twisted divisor \eqref{eq:ZD}. See also the analogous proof of Theorem \ref{thm:PhiMlift} below
which takes the twist into account.

\medskip
We now turn to the case of \emph{odd} positive $j$.
For a negative fundamental discriminant $\Delta<0$ and a weak Maass form of weight $-1/2$
with representation $\rho_L$ we may consider the regularized theta lift
\begin{align}\label{eq:phiDM1}
\Phi_\Delta^\calM(z,h,f) = \int_{\calF}^\reg \langle f(\tau), \theta_{L,\Delta}^\calM(\tau, z, h)\rangle\, d\mu(\tau).
\end{align}
For $\mu\in L'/L$ and $m\in \Z-Q(\mu)$ with $m>0$, let $F_{m,\mu}(\tau,s,-1/2)$
be the Hejhal Poincar\'e series of weight $-1/2$ defined in
\eqref{eq:hps} but with $\bar\rho_L$ replaced by $\rho_L$. We put
\begin{align}\label{eq:phiDM2}
\Phi^\calM_{\Delta, m}(z,h,s) = \Phi_\Delta^\calM(z,h,F_{m,\mu}(\tau,s,-1/2)).
\end{align}
By the usual argument it can be shown that the regularized theta integral is well defined and smooth outside the special divisor
$Z_\Delta(m,1) + Z_\Delta(m, \xi)$. The following result gives an explicit formula for it analogous to Proposition \ref{prop:green12}.

\begin{theorem}\label{thm:PhiMlift}
For $z \in \H \cong \calD^+$, we have
\begin{align*}
\Phi^\calM_{\Delta,m}(z,h,s)
&=-\frac{16\sqrt{m}}{\Gamma(s-1/4)}
\sum_{\substack{\lambda\in V(\Q)\\ Q(\lambda)=m|\Delta|\\ \lambda>0}}  \phi_{r\mu}(h^{-1}\lambda)\chi_\Delta(h^{-1}\lambda)
Q_{2s-\frac{3}{2}}\left(  1+\frac{|z-z_\lambda|^2}{2y\Im (z_\lambda)}    \right)\\
&=\frac{8\sqrt{m}}{\Gamma(s-1/4)}
G_{2s-\frac{1}{2}}(Z_\Delta(m,h), (z,h)).
\end{align*}
\end{theorem}

\begin{proof}
Inserting the definition of the Poincar\'e series, we obtain by the unfolding argument
\begin{align*}
\Phi^\calM_{\Delta, m}(z,h,s) &=\frac{1}{\Gamma(2s)} \int_{\Gamma_\infty'\bs\H}^\reg \left\langle \calM_{s,-\frac{1}{2}}(4\pi m v)e(-mu)(\phi_\mu+\phi_{-\mu}),\, \theta_{L,\Delta}^\calM(\tau, z, h)\right\rangle\,d\mu(\tau)\\
&=\frac{2}{\Gamma(2s)\sqrt{|\Delta|}} \sum_{\substack{\lambda \in V(\Q)\\ Q(\lambda)=m|\Delta|}}\chi_\Delta(h^{-1}\lambda) \phi_{r\mu}(h^{-1}\lambda) p_z(\lambda) \\
&\phantom{=}{}\times\int_{0}^\infty \calM_{s,-\frac{1}{2}}(4\pi m v)
\exp\left( -2\pi \frac{Q(\lambda_{z^\perp})}{|\Delta|} v + 2\pi \frac{Q(\lambda_z)}{|\Delta|} v \right) \frac{d v}{v},
\end{align*}
where we also used $p_z(-\lambda)=-p_z(\lambda)$ and
$\chi_\Delta(-\lambda)=-\chi_\Delta(\lambda)$ (since $\Delta < 0$).

To compute the latter integral, which is a Laplace transform, we use \cite[p.215~(11)]{EMOT} and obtain
\begin{align*}
&\int_{0}^\infty \calM_{s,-\frac{1}{2}}(4\pi m v)
\exp\left( -2\pi \frac{Q(\lambda_{z^\perp})}{|\Delta|} v + 2\pi \frac{Q(\lambda_z)}{|\Delta|} v \right) \frac{d v}{v}\\
&=(4\pi m)^{\frac{1}{4}}\int_{0}^\infty M_{\frac{1}{4}, s-\frac{1}{2}}(4\pi m v)
\exp\left( -2\pi m v + 4\pi \frac{Q(\lambda_z)}{|\Delta|} v \right) v^{-3/4}\, d v\\
&= (4\pi m)^{s+\frac{1}{4}}\Gamma(s+\frac{1}{4})(4\pi m-4\pi \frac{Q(\lambda_z)}{|\Delta|})^{-s-\frac{1}{4}} F(s+\frac{1}{4}, s-\frac{1}{4},2s ;\frac{4\pi m|\Delta|}{4\pi m|\Delta|-4\pi Q(\lambda_z)})\\
&=\Gamma(s+\frac{1}{4})
\left(\frac{m|\Delta|}{ Q(\lambda_{z^\perp})}\right)^{s+\frac{1}{4}}
 F(s+\frac{1}{4}, s-\frac{1}{4},2s ;\frac{m|\Delta|}{ Q(\lambda_{z^\perp})}).
\end{align*}
Inserting this, we find
\begin{align*}
\Phi^\calM_{\Delta, m}(z,h,s) &=\frac{2\Gamma(s+\frac{1}{4})}{\sqrt{|\Delta|}\Gamma(2s)} \sum_{\substack{\lambda \in V(\Q)\\ Q(\lambda)=m|\Delta|}} \chi_\Delta(h^{-1}\lambda)\phi_{r\mu}(h^{-1}\lambda)  p_z(\lambda) \\
&\phantom{=}{}\times \left(\frac{m|\Delta|}{ Q(\lambda_{z^\perp})}\right)^{s+\frac{1}{4}}
 F(s+\frac{1}{4}, s-\frac{1}{4},2s ;\frac{m|\Delta|}{ Q(\lambda_{z^\perp})}).
\end{align*}
Moreover, using $2\sqrt{Q(\lambda_{z^\perp})}=|p_z(\lambda)|$, we obtain
\begin{align*}
\Phi^\calM_{\Delta, m}(z,h,s) &=\frac{4\sqrt{m}\Gamma(s+\frac{1}{4})}{\Gamma(2s)}
\sum_{\substack{\lambda \in V(\Q)\\ Q(\lambda)=m|\Delta|}} \chi_\Delta(h^{-1}\lambda)\phi_{r\mu}(h^{-1}\lambda)  \\
&\phantom{=}
\times \frac{p_z(\lambda)}{|p_z(\lambda)|}
\!\!\left(\frac{m|\Delta|}{ Q(\lambda_{z^\perp})}\right)^{s-\frac{1}{4}}
 F(s+\frac{1}{4}, s-\frac{1}{4},2s ;\frac{m|\Delta|}{ Q(\lambda_{z^\perp})}).
\end{align*}
Using \eqref{eq:hypid}, we obtain
\begin{align*}
\Phi^\calM_{\Delta, m}(z,h,s)
&=\frac{8\sqrt{m}}{\Gamma(s-1/4)}
\sum_{\substack{\lambda \in V(\Q)\\ Q(\lambda)=m|\Delta|}} \chi_\Delta(h^{-1}\lambda)\phi_{r\mu}(h^{-1}\lambda)   \frac{p_z(\lambda)}{|p_z(\lambda)|}
Q_{2s-\frac{3}{2}}\left(  1+\frac{|z-z_\lambda|^2}{2y\Im (z_\lambda)}    \right).
\end{align*}
Note that $\frac{p_z(\lambda)}{|p_z(\lambda)|}$
is constant on $\calD^+$,
and it is in fact equal to $-1$ if $\lambda > 0$ and equal to $1$ if $\lambda < 0$.
Using $\chi_\Delta(-\lambda)=-\chi_\Delta(\lambda)$ again, we finally obtain
\begin{align*}
\Phi^\calM_{\Delta, m}(z,h,s)
&=-\frac{16\sqrt{m}}{\Gamma(s-1/4)}
\sum_{\substack{\lambda \in V(\Q)\\ Q(\lambda)=m|\Delta| \\ \lambda >0}} \chi_\Delta(h^{-1}\lambda)\phi_{r\mu}(h^{-1}\lambda)
Q_{2s-\frac{3}{2}}\left(  1+\frac{|z-z_\lambda|^2}{2y\Im (z_\lambda)}    \right)\\
&= \frac{8\sqrt{m}}{\Gamma(s-1/4)} G_{2s-\frac{1}{2}}(Z_\Delta(m,h), (z,h)).
\end{align*}
This concludes the proof of the theorem.
\end{proof}

In particular, at the harmonic point $s=5/4$ (note that the input form has weight $-1/2$), we get
\[
\Phi^\calM_{\Delta,m}(z,h,5/4)=8\sqrt{m}
G_{2}(Z_\Delta(m,h), (z,h)).
\]
Let $f_m\in H_{1/2-j,\rho_L}$ be the unique harmonic Maass form whose Fourier expansion starts with $f_m=q^m\phi_\mu+O(1)$ as $v\to\infty$.
Then the analogue of \eqref{eq:go12} for odd $j$ states
\begin{align}
\label{eq:go12odd}
\frac{1}{(4\pi)^{\frac{j-1}{2}}}
\Phi_\Delta^\calM(z,h,R_{1/2-j}^{\frac{j-1}{2}} f_m) &=
m^{\frac{j-1}{2}}\Gamma\left(\frac{j+1}{2}\right)\Phi_{\Delta,m}^\calM\left(z,h,\frac{5}{4}+\frac{j-1}{2}\right)\\
\nonumber
&= 4 m^{\frac{j}{2}}
G_{1+j}(Z_\Delta(m, h), (z,h)).
\end{align}
Completing the definiton of $\tilde\Phi^j_{\Delta}(z,h,f_m)$ in \eqref{eq:Phijtilde-even}
we define the twisted theta lift of $f \in H_{1/2-j}$ for odd $j$ as
\begin{equation}\label{eq:Phijtilde-odd}
  \tilde\Phi^j_{\Delta}(z,h,f) := \frac{1}{(4\pi)^{\frac{j-1}{2}}}
\Phi_\Delta^\calM(z,h,R_{1/2-j}^{\frac{j-1}{2}} f).
\end{equation}

The following theorem is a combination of a result of Duke and Jenkins \cite{duke-jenkins-integral} (where the maps $\za_d^j$ are called Zagier lifts)
and the generalization (using a theta lift called the Millson lift)
to harmonic Maass forms by Alfes-Neumann and Schwagenscheidt \cite[Theorem 1.1]{Alfes-Schwagenscheidt:Shintani}.
To state the result, put
\begin{align}
\label{eq:tilderho}
	\tilde\rho_L =
		\begin{cases}
	  	\rho_L & \text{if $j$ is odd,} \\
	  	\bar\rho_L & \text{if $j$ is even.}
		\end{cases}
\end{align}
\begin{theorem}\label{thm:millson-lift}
  Let $j \in \Z_{>0}$ and
  let $d$ be a fundamental discriminant with $(-1)^j d < 0$.
  There is a linear map $\za_{d}^j: H_{-2j} \to H_{\frac{1}{2} - j,\tilde\rho_L}$,
  such that $\za_{d}^j(f)$
  is the unique harmonic Maass form in $H_{\frac12 - j,\tilde\rho_L}$
  with principal part (not including the constant term) given by
  \begin{equation}
     \label{eq:MillsonPP}
        |d|^{-j/2}\sum_{m>0}c_f^+(-m)\sum_{n\mid m}\legendre{d}{n} n^{j} q^{-\frac{|d| m^2}{4 n^2}}\phi_{\frac{d m^2}{n^2}}.
   \end{equation}
   Here, for any $x \in \Z$ we write $\phi_{x} = \phi_{(x \bmod{2})}$.
  Furthermore,
  \begin{enumerate}
  	\item If $f$ is weakly holomorphic, then so is $\za_{d}^{j}(f)$.
  	\item More precisely, $\za_{d}^j(f)$ is weakly holomorphic if and only if $L(\xi_{-2j}(f),\chi_d, j + 1) = 0$.
  	\item Finally, if $f \in H_{-2j}$ with $L(\xi_{-2j}(f),\chi_d, j + 1) = 0$ and
  	      the coefficients of the principal part of $f^+$ are all contained in $\Z$,
  	      then \emph{all} Fourier coefficients of
  	      $|d|^{j/2} \za_{d}^j(f)$ are contained in $\Z$.
  \end{enumerate}
\end{theorem}
\begin{proof}
	First of all, we note that both, \cite[Theorem 1.1]{Alfes-Schwagenscheidt:Shintani} and \cite[Theorem 1]{duke-jenkins-integral} are stated for scalar-valued modular forms.
	Using the isomorphism
	\[
	  H_{\frac{1}{2} - j,\tilde\rho_L} \to H_{\frac{1}{2} - j}^+(\Gamma_0(4)), \quad f_0(\tau) \phi_0 + f_1(\tau) \phi_1 \mapsto f_0(4\tau) + f_1(4\tau),
	\]
	we obtain the translation of their results to our vector-valued setting.

    By Theorem 1.1 in \cite{Alfes-Schwagenscheidt:Shintani}, the definition of $\za_{d}^j$ then agrees
    with the $d$th Millson theta lift, up to the normalizing factor $|d|^{-j/2}$.
    Restricted to weakly holomorphic forms, it agrees with the $d$th Zagier lift defined by
    Duke and Jenkins. Therefore, (1) follows from \cite{duke-jenkins-integral}
    and the generalization (2) follows from \cite{Alfes-Schwagenscheidt:Shintani}.

    The third item follows a bit indirectly from \cite{duke-jenkins-integral}:
    If $L(\xi_{-2j}(f),\chi_d, j + 1) = 0$, then $|d|^{j/2} \za_{d}^j(f)$
    is weakly holomorphic by (2) and has only integral Fourier coefficients in
    its principal part.

    Write $-2j = 12\ell + k'$, where $\ell \in \Z$ and $k' \in \{0,4,6,8,10,14\}$
    are uniquely determined.
    Let $A = 2\ell$ if $\ell$ is even and $A = 2\ell - (-1)^j$ if $\ell$ is odd.
    In Section 2 of \cite{duke-jenkins-integral}, a basis $\{f_m \mid\; m \geq -A, (-1)^{j-1}m \equiv 0,1 \bmod{4}\}$ for $M_{1/2 - j}^!$ is constructed, where
    each basis element $f_m$ has a Fourier expansion of the form
    \[
      f_m(\tau) = q^{-\frac{m}{4}} \phi_{m} + \sum_{\substack{n > A \\ (-1)^{j} n \equiv 0,1 \bmod{4}}} a(m,n) q^{n/4} \phi_n
    \]
    and it is shown that $a(m,n) \in \Z$ for all $m$ and $n$.

    Since the principal part of $|d|^{j/2} \za_{d}^j(f)$ contains only integer coefficients,
    it must be an integral linear combination of the $f_m$, and thus all Fourier coefficients are integral.
\end{proof}

\begin{theorem}
\label{thm:twisted-average}
Let $j\in \Z_{\geq 0}$. Let $d_1$ and $\Delta$ be fundamental discriminants  with $(-1)^jd_1 < 0$
and $(-1)^j\Delta >0$, and put $m_1 = |d_1|/4$.
For $f \in H_{-2j}$ we have
  \[
     G_{j+1,f}(Z_\Delta(m_1, h), (z, h)) = - 2^{j-1} \, \tilde\Phi_{\Delta}^{j} \left(z, h, \za_{d_1}^{j}(f)\right).
  \]
\end{theorem}

\begin{proof}
        By \eqref{eq:MillsonPP} together with \eqref{eq:go12-twisted} for even $j$ and
        \eqref{eq:go12odd} for odd $j$, we obtain that
        \begin{align*}
        \tilde \Phi_{\Delta}^{j} \left(z, h, \za_{d_1}^j(f)\right) &= -2^{1-j}\sum_{m>0}c_f^+(-m)m^{j} \\
        	&\phantom{=}\quad\times\sum_{n\mid m}\legendre{d_1}{n}G_{j+1}\left(Z_\Delta\left(\frac{m_1 m^2}{n^2}, h\right), (z,h)\right).
        \end{align*}
       We need to compare the sum on the right-hand side with $G_{j+1,f}(Z_\Delta(m_1), (z,h))$, which is by definition equal to
       \[
         \sum_{m>0} c_f^+(-m) m^{j} G_{j+1}(Z_\Delta(m_1, h) \mid T_m, (z, h)).
       \]
       According to \cite[p. 508]{GKZ}, we have
       \begin{equation}
       	Z_\Delta(m_1, h) \mid T_m = \sum_{n \mid m} \legendre{d_1}{n} Z_\Delta\left(\frac{m_1 m^2}{n^2}, h \right).
       \end{equation}
       and this finishes the proof.
\end{proof}

\subsection{CM values}
From here on, fix a (not necessarily fundamental) discriminant $d_2 < 0$ and
$r_2 \in \{0, 1\}$ such that $d_2 \equiv r_2^2  \pmod{4}$, and put
\[
	x_2 := \begin{pmatrix}
	   \frac{r_2}{2} & 1 \\
	   \frac{d_2-r_2^2}{4} & -\frac{r_2}{2}
	\end{pmatrix}\in L'
\]
and $m_2= Q(x_2)=-d_2/4$.
We let $U = V \cap x_2^\perp$ and consider the corresponding CM cycle $Z(U)$ on $X_{K_\Delta}$
defined in \eqref{eq:defzu}.

Note that we have $U \cong \Q(\sqrt{d_2})$
and both CM points $(z_U^+, 1)$ and $(z_U^+, \xi)$
correspond to the point
\[
    \frac{-r_2 + \sqrt{d_2}}{2}
\]
on each connected component of $X_{K_\Delta}$ if identified with $\SL_2(\Z) \bs \H$.
As in Section \ref{sect:cm},
we have the two lattices $P = L \cap U^\perp$ and $N = L \cap U$.
Explicitly, they are given by
\[
  N = \Z \begin{pmatrix}
     		 1 & 0 \\ -r_2 & -1
          \end{pmatrix}
       \oplus
       \Z \begin{pmatrix}
       	0 & 1 \\ \frac{r_2^2-d_2}{4} & 0
       \end{pmatrix},\qquad
				P = \Z \frac{2}{2-r_2} x_2,
\]
which implies that $P' = \Z\frac{2-r_2}{|d_2|} x_2$.
Hence, the discriminant group
$P'/P$ is cyclic of order $2|d_2|$ if $r_2 = 1$
and of order $|d_2|/2$ if $r_2=0$.
The lattice $N$ has discriminant $d_2$ and is described below in Lemma \ref{lem:N},
which is a special case of Lemma 7.1 in \cite{BY}.

\begin{lemma}\label{lem:N}
  Let $\calO_{d_2} \subset \Q(\sqrt{d_2})$ be the order of discriminant $d_2$
  in $\Q(\sqrt{d_2})$.
  Then $\calO_{d_2} = \Z \oplus \Z \frac{r_2 + \sqrt{d_2}}{2}$
  and this defines
  a $2$-dimensional lattice of discriminant $d_2$
  with the quadratic form $Q(z) = - \Norm_{\Q(\sqrt{d_2})/\Q}(z)$.
  The map
  \[
    f: (\calO_{d_2}, Q) \to (N, Q), \quad x+y\frac{r_2 + \sqrt{d_2}}{2}
    \mapsto x \begin{pmatrix}
     		 1 & 0 \\ -r_2 & -1
          \end{pmatrix} + y \begin{pmatrix}
       	0 & -1 \\ \frac{d_2-r_2^2}{4} & 0
       \end{pmatrix}
  \]
  is an isometry.
  Both lattices are equivalent to the negative definite
  integral binary quadratic form
  $[-1,-r_2,\frac{d_2 - r_2^2}{4}]$.
\end{lemma}

Recall that we put $T = \Gspin(U)$ and $K_T = K \cap T(\A_f)$.
We have $K_T \cong \hat\calO_{d_2}^\times$,
where $\calO_{d_2} \subset \k_{d_2} = \Q(\sqrt{d_2})$ is
the order of discriminant $d_2$ in $\k_{d_2}$
and $\hat\calO_{d_2} = \calO_{d_2} \otimes_\Z \hat\Z$.
This can either be seen as in Corollary \ref{cor:alg}
or alternatively verified using
the embedding given in Lemma \ref{lem:N}.
Consequently, the cycle $Z(U)_K$ on $X_K$ is
in bijection to two copies of the ring class group $\Cl(\calO_{d_2})$,
see Section \ref{sect:bintheta}.

We are now able to obtain a formula for the twisted partial averages at CM points.
At the point $z_U$, we obtain two definite lattices $N_\Delta = \Delta L \cap U$
and $P_\Delta = \Delta L \cap U^\perp$, both equipped with the quadratic form $\frac{Q}{|\Delta|}$.
\begin{theorem}
  \label{thm:partial-average-value}
    Let $(z_U^+,h) \in Z(U)$ be a CM point and let $\calG_{N_\Delta}(\tau, h) \in H_{1, \rho_{N_\Delta}}$ such that $L_1(\calG_{N_\Delta}(\tau,h)) = \theta_{N_\Delta}(\tau,h)$.
	\begin{enumerate}
		\item If $j \in \Z_{>0}$ is even,
		    \begin{align*}
		     G_{j+1,f}(Z_\Delta(m_1,h), (z_U^+,h)) &= -2^{j-1}
		     \CT \left(\langle \psi_\Delta(\za_{d_1}^{j}(f))_{P_\Delta \oplus N_\Delta}, [\theta_{P_\Delta}(\tau), \calG_{N_\Delta}^+(\tau,h)]_{\frac{j}{2}}\rangle\right),
		    \end{align*}
		    \item and if $j \in \Z_{>0}$ is odd, we have
		    \begin{align*}
		     G_{j+1,f}(Z_\Delta(m_1, h), (z_U^+,h)) &= -2^{j-1}\CT \left(\langle \psi_\Delta(\za_{d_1}^{j}(f))_{P_\Delta \oplus N_\Delta}, [\tilde\theta_{P_\Delta}(\tau), \calG_{N_\Delta}^+(\tau,h)]_{\frac{j-1}{2}}\rangle\right),
		    \end{align*}
		    where $\tilde\theta_{P_\Delta}(\tau)$ is the weight $3/2$ theta function
		    \[
		       \tilde\theta_{P_\Delta}(\tau) = \frac{1}{\sqrt{|\Delta|}}
		       		\sum_{\mu \in P'/\Delta P} \sum_{\lambda \in \Delta P + \mu} p_{z_U^+}(\lambda) e\left(\frac{Q(\lambda)}{|\Delta|}\tau\right) \phi_\mu \in M_{3/2, \rho_{P_\Delta}}.
		    \]
	\end{enumerate}
\end{theorem}

\begin{remark}
    Note that in both cases, the right-hand side in Theorem \ref{thm:partial-average-value}
    is well-defined even if $(z_U^\pm, h)$ is contained in any of the divisors $Z_\Delta(m_1,h) \mid T_m$ for
	$m > 0$ with $c_f(-m) \neq 0$.
	These are in fact the values of the (non-continuous)
	extension of the higher Green function to the divisor
	obtained from realizing it as a regularized theta lift in Theorem
	\ref{thm:twisted-average}.
\end{remark}

\begin{proof}[Proof of Theorem \ref{thm:partial-average-value}]
The proof is analogous to Theorem \ref{thm:fund}. First suppose that $j$ is even.
By virtue of Theorem \ref{thm:twisted-average}, \eqref{eq:Phijtilde-even}, and Lemma \ref{lem:twisting}, we obtain
\begin{align*}
  G_{j+1,f}(Z_\Delta(m_1), (z_U^+,h))
  &= - 2^{j-1} \, \tilde\Phi_{\Delta}^{j} \left(z_U^+, h, \za_{d_1}^{j}(f)\right)\\
  &= - \frac{2^{j-1}}{(4\pi)^{\frac{j}{2}}} \, \Phi_{\Delta} \left(z_U^+, h, R_{1/2-j}^{\frac{j}{2}}\za_{d_1}^{j}(f)\right)\\
  &= - \frac{2^{j-1}}{(4\pi)^{\frac{j}{2}}} \, \int_{\calF}^\reg \langle R_{1/2-j}^{\frac{j}{2}}\za_{d_1}^{j}(f)(\tau), \theta_{L,\Delta}(\tau, z, h)\rangle\, d\mu(\tau)\\
  &= - \frac{2^{j-1}}{(4\pi)^{\frac{j}{2}}} \, \int_{\calF}^\reg \langle \psi_{\Delta}(R_{1/2-j}^{\frac{j}{2}}\za_{d_1}^{j}(f)(\tau)), \theta_{L_\Delta}(\tau, z, h)\rangle\, d\mu(\tau).
\end{align*}
From here, the proof continues parallel to the one of Theorem \ref{thm:fund}.
For odd $j$, we can perform the analogous calculation using the definition
in \eqref{eq:Phijtilde-odd} and the splitting of the Millson theta function
  \[
    \theta_{P_\Delta \oplus N_\Delta}^\calM (\tau, z_U^+, h) = \tilde\theta_{P_\Delta}(\tau) \otimes \theta_{N_\Delta}(\tau, h). \qedhere
\]
\end{proof}
\begin{corollary}
\label{cor:partial-average-algebraic-general}
Let $j\in \Z_{\geq 0}$. Let $d_1$ and $\Delta$ be fundamental discriminants  with $(-1)^jd_1 < 0$
and $(-1)^j\Delta >0$,
and assume that $d_1d_2\Delta$ is not a square of an integer.
Put $m_1 = |d_1|/4$, and
let $f \in H_{-2j}$ with integral principal part and such that $L(\xi_{-2j}(f), \chi_{d_1}, j+1) = 0$. Then we have for any $(z_U^+,h) \in Z(U)$ that
the value $|d_1d_2 \Delta|^{j/2} G_{j+1,f}( Z_\Delta(m_1, h), (z_U^+,h))$
can be expressed as a finite \emph{integral} linear combination of Fourier coefficients of $\calG_{N_\Delta}^+$.
In particular,
   \[
     |d_1d_2 \Delta|^{j/2} G_{j+1,f}( Z_\Delta(m_1, h), (z_U^+,h)) =
                 - \frac{1}{t}\log| \alpha_{U,f,\Delta}(h) |,
   \]
   where $\alpha_{U,f,\Delta}(h) \in H_{d_2}(\sqrt{\Delta})^\times$
   and $t \in \Z_{>0}$ only depends on $d_2$ and $\Delta$.
Moreover, we have
\[
  \alpha_{U,f,\Delta}(h) = \alpha_{U,f,\Delta}(1)^{[h, k_{d_2}]}.
\]
\end{corollary}

    Note that even when $d_1d_2\Delta$ is a square, the statement of the corollary remains valid if $(z_U^+,h)$ is not contained in any of the Hecke translates
  	$Z_\Delta(m_1, h) \mid T_m$ for $m > 0$ with $c_f^+(-m) \neq 0$. For the proof of Corollary \ref{cor:partial-average-algebraic-general}, we need the following Lemma.
\begin{lemma}\label{lem:quadext}
    Let
  	\[
  		K_{T,\Delta} = \{h \in K_T\ \mid\ (\det(h), \Delta)_{\A_f} = 1\}.
  	\]
	Using the identification $T(\A_f) \cong\A_{k_{d_2},f}^\times$ as before,
    $K_{T, \Delta}$ is identified with
    \[
    	\{h \in \hat\calO_{d_2}^\times\ \mid\ (\Norm(h), \Delta)_{\A_f} = 1\}
   	\]
  	and its fixed field under the Artin map is $H_{d_2}(\sqrt{\Delta})$.
\end{lemma}

\begin{proof}[Proof of Corollary \ref{cor:partial-average-algebraic-general}]
  By Theorem \ref{thm:millson-lift} (3),
  all Fourier coefficients of $|d_1|^{j/2}\za_{d_1}^{j}(f)$ are integral.
  The Fourier coefficients of the Rankin-Cohen brackets
  $[\theta_{P_\Delta}(\tau), \calG_{N_\Delta}^+(\tau,h)]_{j/2}$
  and $\sqrt{|d_2 \Delta|}[\tilde\theta_{P_\Delta}(\tau), \calG_{N_\Delta}^+(\tau,h)]_{(j-1)/2}$
  can be expressed as \emph{rational}
  linear combinations of the Fourier coefficients of $\calG_{N_\Delta}^+(\tau,h)$.
  The denominator of the rational numbers appearing
  in this linear combination can be bounded by
  $(4|d_2 \Delta|)^{j/2}$ when $j$ is even
  and by $(4|d_2 \Delta|)^{\frac{j-1}{2}} \cdot 2$
  when $j$ is odd (the additional factor $2$ is obtained from $p_z(\lambda)$ if $r_2=1$).
  In any case, taking into account the factor $2^{j-1}$
  in Theorem \ref{thm:partial-average-value},
  and the factor $|d_1 d_2 \Delta|^{j/2}$
  in the statement of the corollary,
  we are left with a factor of $2$ in the denominator.

  This remaining $2$ in the denominator is also cancelled
  which can be seen as follows:
  If we write the constant term on the right-hand side of Theorem
  \ref{thm:partial-average-value} as a sum of the form
  \[
    \sum_{\mu \in (P_\Delta \oplus N_\Delta)'/(P_\Delta \oplus N_\Delta)}
     \sum_{m} a(m, \mu)b(-m,\mu),
  \]
  where $a(m,\mu)$ are the Fourier coefficients of
  $\psi_\Delta(\za_{d_1}^{j}(f))_{P_\Delta \oplus N_\Delta}$ and $b(m,\mu)$ the Fourier coefficients
  of $[\theta_{P_\Delta}(\tau), \calG_{N_\Delta}^+(\tau,h)]_{\frac{j}{2}}$
  or $[\tilde\theta_{P_\Delta}(\tau), \calG_{N_\Delta}^+(\tau,h)]_{\frac{j-1}{2}}$,
  then we can rewrite this sum as
  \[
    \sum_{\mu \in (P_\Delta \oplus N_\Delta)'/(P_\Delta \oplus N_\Delta) / \{\pm 1\}}
     \sum_{m} (a(m, \mu)b(-m,\mu) + a(m,-\mu)b(m,-\mu)),
  \]
  and since $a(m, \mu)b(-m,\mu) + a(m,-\mu)b(m,-\mu) = 2a(m, \mu)b(-m,\mu)$,
  we obtain another factor of $2$.

  Collecting all factors, we obtain that $|d_1d_2 \Delta|^{j/2} G_{j+1,f}((z_U^\pm,h), Z_\Delta(m_1))$ is equal to an integral linear combination
  of the Fourier coefficients $c_{N_\Delta}^+(h, m, \mu)$ of $\calG_{N_\Delta}^+(\tau,h)$.
  By Theorem~\ref{thm:preimages}, we have
  \[
  	c_{N_\Delta}^+(h, m, \mu) = -\frac{1}{t'} \log|\alpha_{N_\Delta}(h, m, \mu)|
  \]
  for all $(m, \mu) \neq (0,0)$ and with $\alpha_{N_\Delta}(h, m, \mu) \in H_{d_2\Delta^2}^\times$
  and $n \in \Z_{>0}$. Moreover, $t'$ only depends on $N_\Delta$
  which means it only depends on $d_2$ and $\Delta$.
  Thus, we obtain that
  \[
     |d_1d_2 \Delta|^{j/2} G_{j+1,f}( Z_\Delta(m_1, h), (z_U^\pm,h)) =
                 - \frac{1}{t'}\log| \tilde\alpha(h) |
  \]
  with $\tilde\alpha(h) \in H_{d_2 \Delta^2}^\times$.
  However, the left-hand side is invariant under the $h \mapsto h'h$ for $h' \in K_{T,\Delta}$
  and the field $H_{d_2}(\sqrt{\Delta})$ is fixed by these elements according to Lemma
  \ref{lem:quadext}.
  By virtue of the Shimura reciprocity law (Theorem \ref{thm:preimages}, (3))
  and the invariance under $K_{T,\Delta}$, we obtain
  \[
  	|\tilde\alpha(h)^{[h',\, \k_D]}| = |\tilde\alpha(hh')| = |\tilde\alpha(h)|
  \]
  for all $h' \in K_{T,\Delta}$.
  Moreover, we have $\hat\calO_{d_2 \Delta^2}^\times \subset K_{T,\Delta} \subset K_T = \hat\calO_{d_2}^\times$.
  This implies that for all $\sigma \in \Gal(H_{d_2\Delta^2}/H_{d_2}(\sqrt{\Delta}))$,
  there is a root of unity $\zeta_\sigma$ such that
  \[
  	\frac{\sigma(\tilde\alpha(h))}{\tilde\alpha(h)} = \zeta_\sigma.
  \]
  Therefore, the constant term of the minimal polynomial of $\tilde\alpha(h)$
  over $H_{d_2}(\sqrt{\Delta})$ is equal to $\tilde\alpha(h)^m \zeta$
  where $m$ is the degree of $\tilde\alpha(h)$ over $H_{d_2}(\sqrt{\Delta})$
  and $\zeta$ is some root of unity.
  Note that $m$ is bounded by the degree $[H_{d_2 \Delta^2}: H_{d_2}]$
  and thus only depends on $d_2$ and $\Delta$.
  By putting $\alpha_{U, f, \Delta}(h) := \tilde\alpha(h)^m \zeta$,
  and $t = t' m$, we obtain the statement of the corollary.
\end{proof}

We finish this section by rewriting the CM cycle $Z(U)$ in classical terms
and give a proof of Conjecture \ref{conj:alg} when one of the class groups
has exponent $2$.
\begin{lemma} \label{cor:order-d2}
  The image of $Z(U)_K$ on $X_K \cong \SL_2(\Z)\bs\H$ is given as follows.

  Let $\mathcal{Q}_{d_2}^0$ be the set of primitive integral binary quadratic forms
  $[a,b,c]$ of discriminant $d_2$.
  For each such $Q=[a,b,c]$ we denote by $z_Q$
  the unique root of $az^2+bz+c=0$ in $\H$. For simplicity, we also denote its image in $\SL_2(\Z) \bs \H$ by $z_Q$.
  We have
	\[
	Z(U)_K = \frac{4}{w_{d_2}} \sum_{Q \in \SL_2(\Z) \bs \mathcal{Q}^0_{d_2}} z_Q,
	\]
	where $w_{d_2}$ denotes the number of roots of unity contained in $\calO_{d_2}$.
\end{lemma}

\begin{corollary}\label{cor:exponent2}
  Let $D' < 0$ be a fundamental discriminant
  and assume that the class group of $\calO_{D'}$ is trivial or
  has exponent $2$.
  Let $f \in M_{-2j}^!$
  with integral principal part and let $z \in \H$ be any
  CM point of discriminant $D < 0$ (not necessarily fundamental)
  and $z' \in \H$ be any CM point of discriminant
  $D'$, where $z \neq z'$ if $D = D'$.
  Then, there is an $\alpha_{f}(z) \in (H_D \cdot H_{D'})^\times$ such that
   \[
     |DD'|^{j/2} G_{j+1, f} (z, z') = -\frac{1}{t}\log |\alpha_f(z)|,
   \]
   where $t \in \Z_{>0}$ only depends on $D$ and $D'$ (but not on $f$ or $j$).
\end{corollary}

\begin{proof}
  For the proof, we work with one of the connected components of $X_{K_\Delta}$
  and identify it with $\SL_2(\Z) \bs \H$.
  Thus, we can work with the divisors $Z_\Delta(m_1)$ defined in \eqref{eq:ZDelta-classical}.
  For each decomposition of $D'$ into $D' = d_1 \Delta$
  with $d_1$ and $\Delta$ fundamental discriminants
  and $(-1)^j d_1 < 0$ as well as $(-1)^j \Delta > 0$,
  we have shown in Corollary \ref{cor:partial-average-algebraic-general}
  that for any $z$ of discriminant $D$, we have
  \[
    |DD'|^{j/2} G_{j+1, f} (Z_\Delta(m_1), z) = -\frac{1}{t_\Delta} \log |\alpha_\Delta(z)|,
  \]
  where $\alpha_\Delta(z) \in H_{D}(\sqrt{\Delta})^\times$ and $t_\Delta \in \Z_{>0}$.

  We let $C$ be the class group of $\calO_{D'}$.
  For any fundamental discriminant $\Delta \mid D'$,
  we let $d_1 = D'/\Delta$ and
  and $m_1 = |d_1| / 4$. The splitting
  $D' = \Delta\,  d_1$ determines a genus character
  $\chi_{\Delta}: C \to \{\pm{1}\}$.
  The twisted CM cycle $Z_\Delta(m_1)$ is equal to
  \[
    Z_\Delta(m_1) = \sum_{[\fraka] \in C} \chi_\Delta([\fraka]) z_{[\fraka]},
  \]
  where we write $z_{[\fraka]}$ for the CM point corresponding to $\fraka$
  on $\SL_2(\Z)\bs\H$.

  Since $C$ has exponent $1$ or $2$ by assumption,
  its order is exactly $2^{s-1}$, where $s$ is the number of prime divisors of $D'$
  and every class group character can be obtained as a genus character.
  Note that there are exactly $s$ splittings $D' = \Delta \tilde{\Delta}$
  where $\Delta$ and $\tilde{\Delta}$ are both fundamental discriminants
  and $(-1)^j\Delta > 0$ since $D'<0$.
  We denote the fundamental discriminants satisfying these criteria
  by $\Delta_1, \ldots, \Delta_s$.
  Note that $\{\chi_{\Delta_1}, \ldots, \chi_{\Delta_s}\}$
  is a full set of representatives of the class group characters
  of $k_{D'}$.
  Hence, the individual value corresponding to an ideal class $[\fraka]$
  can be obtained as
  \begin{align*}
   G_{j+1, f} (z_{[\fraka]}, z) &= \frac{1}{h_{D'}}
    \sum_{i=1}^s \chi_{\Delta_i}([\fraka])
      G_{j+1, f} (Z_\Delta(m_1), z) \\
      &= -\frac{1}{h_{D'}}
    \sum_{i=1}^s \chi_{\Delta_i}([\fraka])
      \frac{1}{t_{\Delta_i}} \log |\alpha_{\Delta_i}(z)|\\
      &= -\frac{1}{t} \log |\alpha(z)|,
  \end{align*}
  where $t = h_{D'} \cdot \lcm \{t_{\Delta_i} \}$ and
  \[
    \alpha(z) =
    \prod_{i=1}^s
      \alpha_{\Delta_i}(z)^{\frac{t}{t_{\Delta_i}}\chi_{\Delta_i}([\fraka])}
  \]
  is contained in $H_D(\sqrt{\Delta_1}, \ldots, \sqrt{\Delta_s})$.
  Since the class group of $\k_{D'}$ has exponent $2$,
  we have $H_{D'} = \k_{D'}(\sqrt{\Delta_1}, \ldots, \sqrt{\Delta_s})$
  and the claim follows.
\end{proof}

\section{Numerical examples}
\label{sect:ex}

Here we provide some numerical examples to
illustrate the results of Section \ref{sec:partial-averages}.
In particular we demonstrate how our main results
in Section \ref{sec:partial-averages} and the Appendix
can be implemented
to obtain explicit formulas for the algebraic
numbers determining the CM values of higher Green functions.

\subsection{Example 1}
\label{sec:ex1}
We start with an example for $j=2$, which is a bit simpler than $j=1$
since we can work with $\Delta=1$.

Note that for $k = j + 1 = 3$,
we have that $S_{2k} = S_{6} = \{ 0 \}$.
Therefore, the algebraicity conjecture concerns the individual values of $G_3$
at pairs of CM points in this case. The function $G_3(z_1, z_2)$
is obtained as the higher theta lift $\tilde\Phi^1(f, z)$,
where $f \in M_{-4}^!$ is the unique weakly holomorphic modular form
for the full modular group whose Fourier expansion starts with
$f = q^{-1} + O(1)$. It is explicitly given by
\[
   f = E_{8}(\tau)/\Delta(\tau),
\]
where $E_8(\tau) \in M_8$ is the normalized Eisenstein series
of weight $8$ for $\SL_2(\Z)$ and $\Delta(\tau) \in S_{12}$ is
the discriminant function.

For our first explicit example, we let $d_1=-4$, $d_2=-23$, and $\Delta=1$.
Since the class group of $\Q(\sqrt{-1})$ is trivial,
the CM-cyle $Z(1/4)$ only contains one point, represented by $i \in \H$
and counted with multiplicity $1/2$.
Using the usual isomorphism $M_{1/2, \rho_L}^! \cong M_{1/2}^{!,+}(\Gamma_0(4))$,
we can identify the Zagier lift $\za_{-4}^{2}(f)$ of $f$ with the form
\[
 4\za_{-4}^{2}(f) = q^{-4} - 126 - 1248q - 263832q^{4} - 666664q^{5} + O(q^{6}),
\]
which can be constructed in a similar fashion as the weight $1/2$ forms in \cite{zagier-traces}; see also \cite{duke-jenkins-integral}.

We consider the case $d_2 = -23$.
In the notation of the previous section, we have $r=1$ and
\[
	x_2 := \begin{pmatrix}
	   \frac{1}{2} & 1 \\
	   -6 & -\frac{1}{2}
	\end{pmatrix}\in L'.
\]
The CM cycle $Z(U)$ for $U = V(\Q) \cap x_2^\perp$ then consists of the three points
\[
   z_1 = (z_U^\pm, 1) = \frac{-1+i\sqrt{23}}{2}, \quad z_{2/3} = \frac{\pm 1 + i\sqrt{23}}{4}.
\]
The discriminant group $P'/P$ is isomorphic to $\Z/46\Z$ with quadratic form
$x^2/92$ and, according to Lemma \ref{lem:N}, the lattice $N$ is isomorphic to the ring of integers
$\calO_{-23} \subset \k_{-23} = \Q(\sqrt{-23})$ with quadratic form given by the negative of the norm form.
Numerical approximations for the CM values can be obtained by using the Fourier expansion of $G_3$, for instance
\[
  G_3(i, z_1) \approx -1.000394556341.
\]

Note that, given $m \not\in \Z$, there are exactly
two cosets $\pm \mu_m$, such that $m \equiv Q(\mu_m)$.
If $m \in \Z$, then $\mu = 0$ is the only possibility.
We now let $\calG_N^+(\tau) = \sum_{m} c(m) q^m \phi_m$
be the holomorphic part of a harmonic Maass form
$\calG_N \in H_{1, \rho_N}$ with the property
that $L_1 (\calG_N(\tau)) = \theta_N(\tau,1)$.
To lighten the notation, we drop
the index of the component $\mu$ and simply write
$c(m)$ for $c(m,\mu)$ and $\phi_m = \phi_{\mu_m} + \phi_{-\mu_m}$
for $m \not\in \Z$ and $\phi_m = \phi_0$ for $m \in \Z$.
We require the additionally that $c(m) = 0$ for $m < -1/23$,
which can be satisfied because the space $S_{1, \bar\rho_N}$ is one-dimensional
and spanned by a cusp form whose Fourier expansion starts with $q^{1/23}$.
These conditions then characterize $\calG_N$ uniquely, since $M_{1, \rho_N} = \{ 0 \}$.

Theorem \ref{thm:partial-average-value} now gives the formula
\begin{align*}
  \frac{1}{2} G_{3}(i, z_1) &= -\frac{25}{46} c(7/23) - \frac{4}{46} c(14/23) +  \frac{11}{46} c(19/23) \\
                  &\phantom{=}+ \frac{20}{46} c(22/23) + \frac{1}{4} c(23/23)
                  + \frac{378}{46} c(-1/23).
\end{align*}
The coefficients of $\calG_N^+$ can be obtained
as follows. We let $h_{23}$ be the normalized Hauptmodul for $\Gamma_0^+(23)$,
the extenson of $\Gamma_0(23)$ by the Fricke involution.
Its Fourier expansions starts with
\[
  h_{23}(\tau) = q^{-1} + 2 + 4q + 7q^2 + 13q^3 + 19q^4 + 33q^5 +O(q^6).
\]
It is shown in \cite[Section 6]{Eh} that the
algebraic numbers $\alpha(1, m, \mu) =: \alpha(m)$ occuring in Theorem \ref{thm:preimages}
can all be obtained as certain rational expressions in the CM value
\[
   \alpha := -h_{23}\left( \frac{-23 + i \sqrt{23}}{46} \right) \approx 1.324717957244.
\]
By CM theory, this value is contained in the Hilbert class field $\Hilb_{-23}$ of $\k_{-23}$.
In fact, $\alpha$ is the unique real root of $X^3 - X - 1$ and a fundamental unit of $\Hilb_{-23}$.
The expressions for the relevant $\alpha(m)$ are given in the second column of Table \ref{tab:coeffs23}.
Note that the Hilbert class field $H_{-23}$ has class number one, so its ring of integers $\Rhilb_{-23}$ is  a unique facorization domain.
In the third column we list the corresponding prime valuations
(see Theorem 1.3 of \cite{Eh}) using the following convention.
If $p$ is a rational prime which is non-split in $\k_{-23}$,
then there is a unique prime ideal $(\pi_p)$
of $\Hilb_{-23}$ above $p$ that divides $\alpha(p/23)$.
If the valuation of $\alpha(p/23)$ at $(\pi_p)$
is equal to $\nu$ and the element $\alpha(m)$
is of the form $x^\nu$ for some $x \in \Rhilb_{-23}$, we choose $\pi_p = x$.
For instance, we put $\pi_7 = \alpha^2 + \alpha - 2$.
In our example this does not work for $m=23$.
In this case, we can take the prime element
$\pi_{23} = \frac{1}{\sqrt{-23}}\left( -9 \alpha^{2} + 2\alpha + 6 \right) \in \Rhilb_{-23}$.
Then $\alpha(23)/\pi_{23}^4$ is a unit.
Finally, we then let $\pi_{p}^{\prime}, \pi_p^{\prime\prime}$ be the two Galois conjugates
of $\pi_p^{(1)}$ over $\k_{-23}$,
so that $(p) = (\pi_{p})(\pi_{p}^{\prime})(\pi_{p}^{\prime\prime})$
in $\Hilb_{-23}$.
For instance, for $m=14$,
the entry $(\pi_7^{\prime},2), (\pi_7^{\prime\prime}, 2)$ in the third column
means that the principal ideal $(\alpha(14))$ factors as $(\alpha(14)) = (\pi_7^{\prime})^2 (\pi_7^{\prime\prime})^2$.

\begin{table}
  \centering
\begin{tabular}[h]{c|c|c|c}
  $23m$ & $ \alpha(m)$ & primes, valuations & $c(m) = -\log|\alpha(m)|$ (12 digits)\\\hline
  $7$  & $(\alpha^2 + \alpha - 2)^2$ & $(\pi_7, 2)$ & $-0.153173096659$ \\\hline
  $11$ & $(2\alpha^2 - \alpha)^2$ & $(\pi_{11}, 2)$ & $-1.563265867556$ \\\hline
  $14$ & $(\alpha^2 - 2\alpha + 3)^2$ & $(\pi_7^\prime,2), (\pi_7^{\prime\prime}, 2)$  & $-1.489050606868$ \\\hline
  $19$ & $(3\alpha^2 + \alpha)^2$ & $(\pi_{19},2)$  & $-3.770909708871$ \\\hline
  $22$ & $(3\alpha^2 + 7\alpha + 6)^2$ & $(\pi_{11}^{\prime},2), (\pi_{11}^{\prime\prime},2)$ & $-6.04452042127$\\\hline
  $23$ & $(8\alpha^2 + 12\alpha + 7)^2$ & $(\pi_{23},4)$ & $-7.218353704778$ \\\hline
  $-1$ & $\alpha^{-2}$ & & $0.562399148646$ \\\hline
\end{tabular}
\caption{ Coefficients of $\calG_N^+(\tau)$.
\label{tab:coeffs23}}
\end{table}

Taking the units into account, we obtain the precise value
\[
  G_3(i, z_1) = - \frac{1}{23} \log \left| \alpha^{-294} \cdot \frac{(\pi_{11}^{\prime}\pi_{11}^{\prime\prime})^{40}\, \pi_{19}^{22}\, \pi_{23}^{46}}{\pi_7^{50} (\pi_7^{\prime}\pi_7^{\prime\prime})^{8}} \right|.
\]
Note that the algebraic number in the logarithm is in fact contained in the real subfield $\Q(\alpha) \subset H_{-23}$
and thus has degree $3$ over $\Q$ (this is visible in Table \ref{tab:coeffs23}).

Moreover, it follows that we obtain $G_3(i, z_2) = G_3(i, z_3)$
by applying any Galois automorphism $\sigma \in \Gal(\Hilb_{-23}/\k_{-23})$ of order $3$
to the numbers $\alpha(m)$ in Table \ref{tab:coeffs23}.
Note that the two possible choices for $\sigma$
lead to complex conjugate algebraic numbers
and thus $\log|\alpha(m)^\sigma|$ is independent of this choice.
Numerically, we have $G_3(i, z_2) = G_3(i, z_3) \approx -3.854054384748$.
Finally, we obtain the factorization for the average value
\[
  G_3(i, z_1) + G_3(i, z_2) + G_3(i, z_3) = -\frac{1}{23} \log\left(\frac{11^{80}19^{22}23^{23}}{7^{66}}\right) \approx -8.708503325837,
\]
which alternatively follows directly from Theorem \ref{thm:fund}.

\subsection{Example 2}
\label{sect:8.2}
One of the interesting features of the CM value formula is that
the same harmonic Maass form $\calG_N$ occurs for all $j$.
To illustrate this, consider the same CM points
for $d_1 = -4$ and $d_2 = -23$ but now take $j=4$ and $j=6$.
For these values, we still have $S_{2+2j}(\SL_2(\Z)) = \{0\}$.
Theorem \ref{thm:partial-average-value} shows that the Fourier coefficients of the same function $\calG_N^+$
occur. In the case $j=4$, the numerical value is approximately
\[
   G_5(i, z_1) \approx -0.0869366459199.
\]
By Theorem \ref{thm:partial-average-value}, we obtain
\begin{align*}
  \frac{1}{2} G_{5}(i, z_1) &= \frac{493}{4232} c(7/23) + \frac{447}{1058} c(14/23)
                  +  \frac{613}{4232} c(19/23) \\
                  &\phantom{=}- \frac{233}{1058} c(22/23) - \frac{3}{16} c(23/23)
                  - \frac{5775}{2116} c(-1/23),
\end{align*}
which gives a precise value of
\[
  G_5(i, z_1) = -\frac{1}{4 \cdot 23^2}\log\left\lvert \alpha^{-15594} \cdot \frac{\pi_7^{986}\, (\pi_7^{\prime}\pi_7^{\prime\prime})^{3576}\, \pi_{19}^{1226}}{ (\pi_{11}^{\prime}\pi_{11}^{\prime\prime})^{1864}\,\pi_{23}^{3174}} \right\rvert.
\]
In the case $j=6$, we get an approximate value of
\[
  G_7(i, z_1) \approx -0.0101643901834.
\]
And Theorem \ref{thm:partial-average-value} yields
\begin{align*}
  \frac{1}{2} G_{7}(i, z_1) &= -\frac{80659}{194672} c(7/23) + \frac{2578}{24334} c(14/23)
                  +  \frac{60209}{194672} c(19/23) \\
                  &\phantom{=}- \frac{1538}{24334} c(22/23) - \frac{5}{32} c(23/23)
                  - \frac{42273}{97336} c(-1/23),
\end{align*}
which gives a similar precise algebraic formula as above.

\subsection{Example 3}
In this section we give an explicit example for $j=1$.
We obtain an explicit,
finite and algebraic expression for the value
\[
  \frac{1}{3}G_2\left(\frac{1 + \sqrt{-3}}{2}, \frac{1 + \sqrt{-7}}{2} \right) = G_2(Z_{-3}(1/4), z_{x_2}),
\]
where the CM point $z_{x_2}$ of discriminant $-7$ corresponds to the vector
\[
  x_2 = \begin{pmatrix}
  	\frac12 & 1 \\
  	-2 & -\frac12
  \end{pmatrix} \in L'.
\]
The parameters for this example are $d_1=1, d_2=-7, \Delta = -3$.

Numerically, we have
\[
   \frac{1}{3}G_2\left(\frac{1 + \sqrt{-3}}{2}, \frac{1 + \sqrt{-7}}{2} \right) \approx -2.928818048619.
\]
The unique weakly holomorphic modular form
$f$ of weight $-2$ with a principal part starting with $q^{-1}$
that corresponds to $G_2(z_1, z_2)$ is given by
\[
	f(\tau) = E_{10}(\tau)/\Delta(\tau),
\]
where $E_{10}(\tau) \in M_{10}$ is the normalized Eisenstein
series of weight $10$ for $\SL_2(\Z)$
and $\Delta(\tau)$ is the discriminant function.
Its Zagier lift $\za_1^1(f)$ for $d_1=1$
can be identified with a scalar-valued form in
$M_{-1/2}^{!,+}(\Gamma_0(4))$ given by
\[
  q^{-1} + 10 - 64q^{3} + 108q^{4} - 513q^{7} + 808q^{8} - 2752q^{11} + 4016q^{12} - 11775q^{15} + O(q^{16}).
\]

The lattice $P$ is spanned by $2x_2$ and $P'$ by $x_2/7$.
According to Theorem \ref{thm:partial-average-value}, we have
\begin{align*}
   \frac{1}{3}G_{2}\left( \frac{1 + \sqrt{-3}}{2}, \frac{1 + \sqrt{-7}}{2} \right) &=
       \frac{1}{2}\CT \left(\langle \psi_{-3}(\za_{1}^{1}(f))_{P_{-3} \oplus N_{-3}}, \tilde\theta_{P_{-3}}(\tau) \otimes \calG_{N_{-3}}^+(\tau)\rangle\right).
\end{align*}

The lattice $N_{-3}$ is isomorphic to the order
$\calO_{-63}$ in $\k_{-7}=\Q(\sqrt{-7})$
of discriminant $-63$.
We take the basis $(1, 3\frac{1+\sqrt{-7}}{2})$
of $\calO_{-63}$, which has the Gram matrix
\[
\begin{pmatrix}
  -2 & -3 \\
  -3 & -36
\end{pmatrix}.
\]
The dual of $\calO_{-63}$ is given by the fractional ideal generated by $\frac{1}{3\sqrt{-7}}$.
The discriminant group is isomorphic to $\Z/21\Z \times \Z/3\Z$.
We write $c(m, \mu)$ for the $(m,\mu)$-th Fourier coefficient of the holomorphic part of
$\calG_{N_{-3}}(\tau)$ for simplicity
and we write $\mu = (a,b)$ with $a \in \Z/21\Z$ and $b \in \Z/3\Z$.
Then a little calculation, which we carried out using \sage, shows that we have explicitly
\begin{align*}
    G_{2}\left( \frac{1 + \sqrt{-3}}{2}, \frac{1 + \sqrt{-7}}{2} \right)
       &= \frac{3}{\sqrt{21}}\Bigg( -25 c\left(\frac{-1}{21}, (1,0)\right) + 25 c\left(\frac{-1}{21}, (1,1)\right) \\
       &\phantom{= -\frac{1}{\sqrt{21}}\Big( } - 25 c\left(\frac{-1}{21}, (8,0)\right) + 5 c\left(\frac{-1}{21}, (8,2)\right) \\
       &\phantom{= -\frac{1}{\sqrt{21}}\Big( } + c\left(\frac{5}{21}, (4,0)\right) - c\left(\frac{5}{21}, (4,1)\right) \\
       &\phantom{= -\frac{1}{\sqrt{21}}\Big( } + c\left(\frac{5}{21}, (10,0)\right) - c\left(\frac{5}{21}, (10,1)\right) \Bigg).
\end{align*}
We implemented the algorithm outlined in the Appendix in \sage
to compute these coefficients numerically.

We remark that the computations are \emph{much}
harder than the previous examples for several reasons:
1) the algorithm in the appendix is computationally more
expensive because the twist results in much larger discriminant groups.
2) The coefficients are obtained as CM values
of meromorphic modular functions on $\Gamma_0(63)$, which has genus $5$.
In our first example, we used the fact that
the corresponding modular forms are rational functions
on $\Gamma_0^+(23) \bs \H$, which
has genus $0$, to obtain all CM values in terms
of just one CM value of the hauptmodul of $\Gamma_0^+(23)$.

Each of the coefficients $c(m,\mu)$
is of the form $-\frac{1}{3}\log|\alpha(m,\mu)|$ with
$\alpha(m, \mu)$ an algebraic integer contained in the ring class field
$\Hilb_{-63}$ of $\calO_{-63}$.
We can use this information to determine $\alpha(m, \mu)$
exactly from the numerical computations.
We have $\Hilb_{-63} = \Q(\alpha_1)$,
where the minimal polynomial of $\alpha_1$
is given by
$x^8 + x^6 - 3x^4 + x^2 + 1$.
We write $\Rhilb_{-63}$
for the ring of integers in $\Hilb_{-63}$.
We fix the embedding of $\Hilb_{-63}$ into $\C$, such that
$\alpha_1 \approx -0.9735614833 - 0.22842512587i$.
The field $\Hilb_{-63}$ has a total of $4$ pairs of complex-conjugate embeddings.
Hence, the rank of the unit group is $3$ and is
generated by $\alpha_1$, $\alpha_2$ and $\alpha_3$, where
\[
  \alpha_2 = \frac{1}{2} \alpha_1^{7} + \frac{1}{2} \alpha_1^{5} - \alpha_1^{3} + \frac{1}{2} a^{2} + a + \frac{1}{2} \approx 0.562638276594 - 0.324839360448i
\]
has minimal polynomial
$x^8 - 3x^7 + 4x^6 - 3x^5 + 3x^4 - 3x^3 + 4x^2 - 3x + 1$,
and
\[
  \alpha_3 = \frac{1}{2} \alpha_1^{6} + \frac{1}{2} \alpha_1^{4} - \frac{1}{2} \alpha_1^{3} - \frac{3}{2} \alpha_1^{2} + \frac{1}{2} \approx -0.0626382765944 + 0.541186043336i
\]
with minimal polynomial $x^8 - x^7 + 2x^6 - x^5 - 5x^4 + x^3 + 2x^2 + x + 1$.

Note that only $m=-1/21$ and $m=5/21$
occur. This suggests that only
primes above $5$ should occur in the factorization of the CM value,
which is indeed the case.
The prime $5$ is inert in $\Q(\sqrt{-7})$ and splits completely in $\Rhilb_{-63}$
into $5\Rhilb_{-63} = \frakp_1\frakp_2\frakp_3\frakp_4$, where
\begin{align*}
  \frakp_1 = (\pi_1), \text{ with } \pi_1 &= -\alpha_{1}^{7} + \frac{1}{2} \alpha_{1}^{6} - \alpha_{1}^{5} + \alpha_{1}^{4} + 3 \alpha_{1}^{3} - \alpha_{1}^{2} - \frac{1}{2} \alpha_{1}, \\
  \frakp_2 = (\pi_2), \text{ with } \pi_2 &= \alpha_1^{5} + \alpha_1^{3} - 2 \alpha_1, \\
  \frakp_3 = (\pi_3), \text{ with } \pi_3 &= \frac{1}{2} \alpha_{1}^{7} + \frac{1}{2} \alpha_{1}^{5} - \frac{1}{2} \alpha_{1}^{4} - \frac{3}{2} \alpha_{1}^{3} - \alpha_{1}^{2} - \frac{1}{2} \alpha_{1} + 1, \text{ and }\\
  \frakp_4 = (\pi_4), \text{ with } \pi_4 &= -\alpha_1^7 - \alpha_1^5 + \frac12 \alpha_1^4 + \frac52 \alpha_1^3 + \frac12 \alpha_1^2 - \frac12 \alpha_1 - \frac12.
\end{align*}
The values for the algebraic numbers $\alpha(m, \mu)$,
are recorded in Table \ref{tab:coeffs-3-7}. Here, we wrote down exactly what we obtained
by implementing the method of the Appendix, even if $|\alpha(m, \mu)| = 1$,
which yields a vanishing Fourier coefficient.
\begin{table}
  \centering
\begin{tabular}[h]{c|c|c|c}
  $21m$ & $\mu$ & $ \alpha(m, \mu)$ & $c(m, \mu) = -\frac{1}{3}\log|\alpha(m, \mu)|$ (12 digits)\\\hline
  $-1$ & $(1,0)$ & $\alpha_1^{-4}\alpha_2^{2}\alpha_3^{2}$ & $0.692410519993$ \\\hline
  $-1$ & $(1,1)$ & $\alpha_1^{-6}$ & $0$ \\\hline
  $-1$ & $(8,0)$ & $\alpha_1^{-4}\alpha_2^{-4}\alpha_3^{2}$ & $-0.170144107668$ \\\hline
  $-1$ & $(8,2)$ & $1$ & $0$ \\\hline
  $5$ & $(4,0)$ & $\pi_1^{6} \alpha_1^{-4}\alpha_2^{8}\alpha_3^{2}$ & $0.255860917422$ \\\hline
  $5$ & $(4,1)$ & $\pi_2^{6}\alpha_1^{-12}$ & $-0.582934829024$ \\\hline
  $5$ & $(10,0)$ & $\pi_3^{6}\alpha_1^{-4}\alpha_2^{-10}\alpha_3^{2}$ & $-1.786600671916$ \\\hline
  $5$ & $(10,1)$ & $\pi_4^{6}$ & $-0.582934829024$ \\\hline
\end{tabular}
\caption{
\label{tab:coeffs-3-7}}
\end{table}
Summarizing, we obtain the following expression for the value of the higher Green function
\[
  G_{2}\left( \frac{1 + \sqrt{-3}}{2}, \frac{1 + \sqrt{-7}}{2} \right)
  = -\frac{3}{\sqrt{21}} \log\left| \frac{\alpha_1^{18}\alpha_2^{16}\pi_1^{2}\pi_3^{2}}{\alpha_3^{32}\pi_2^{2}\pi_4^2} \right|.
\]
As predicted, we can check that the algebraic number in the logarithm is contained in
$\Q(\sqrt{21})$. With a little computation, we obtain the surprisingly simple expression
\[
	G_{2}\left( \frac{1 + \sqrt{-3}}{2}, \frac{1 + \sqrt{-7}}{2} \right)
  = -\frac{3}{\sqrt{21}} \log\left| \frac{(32+7\sqrt{21})^4}{25} \right|.
\]

\appendix

\section{Preimages of theta functions}
\label{sec:preim-theta}
Following the strategy of \cite{Eh} with a few modifications,
we will now give a proof of Theorem \ref{thm:preimages}.
In contrast to \cite{Eh}, we will not consider the prime ideal factorizations
of the algebraic numbers, which allows for some simplifications.
In this regard, the results of \cite{Eh} are stronger than Theorem \ref{thm:preimages}.
However, Theorem \ref{thm:preimages} is much more general as it does not
put any restriction on $N$, whereas in \cite{Eh} the assumption was that the discriminant of $N$
is an odd fundamental discriminant.

\subsection{Weakly holomorphic modular forms}
\label{sec:cusp-forms-weight}
In this section, we basically follow \cite[Section 4.3]{Eh}
to define a convenient basis of the space of weakly holomorphic modular forms.
The setup for this section is more general than
for the rest of the appendix. For simplicity, we make the following assumptions:
we let $(N,Q)$ be any even lattice of even signature and let $k \in \Z$
such that $2k \equiv \sgn(N) \bmod{4}$.

First consider the space of holomorphic modular forms $M_{k,\rho_N}(\Q)$
with rational coefficients and its dual $M_{k,\rho_N}(\Q)^\vee$.
Let $(m_1, \mu_1), \ldots, (m_r, \mu_r) \in \Q_{\geq 0} \times N'/N$ such that
the linear maps $\alpha_1, \ldots, \alpha_r \in M_{k, \rho_N}(\Q)^\vee$ defined by
\[
  \alpha_i: M_{k, \rho_N}(\Q) \to \Q, \quad f \mapsto c_f(m_i, \mu_i)
\]
form a basis of $M_{k,\rho_N}(\Q)^\vee$.
We fix these indices once and for all and let $G_1, \ldots, G_r \in M_{k,\rho_N}(\Q)$
be the dual basis, i.e., $G_i$ satisfies $c_{G_i}(m_j, \mu_j) = \delta_{i,j}$.

In the same way, we fix indices $(\tilde{m}_1, \tilde{\mu}_1), \ldots, (\tilde{m}_s, \tilde{\mu}_s) \in \Q_{\geq 0} \times N'/N$
for the space $M_{2-k,\bar{\rho}_N}(\Q)$ such that the linear maps
\[
  \beta_i: M_{2-k, \bar\rho_N}(\Q) \to \Q, \quad f \mapsto c_f(\tilde{m}_i, \tilde{\mu}_i)
\]
form a basis of $M_{2-k,\bar{\rho}_N}(\Q)^\vee$.
As before, we let $F_1, \ldots, F_s \in M_{2-k,\bar\rho_N}(\Q)$ be the dual basis.

Now we define special bases for the spaces $M_{k,\rho_N}^!(\Q)$ and $M_{2-k, \bar\rho_N}^!(\Q)$.
For $M_{2-k, \bar\rho_N}^!(\Q)$, we define a basis $\{ f_{m, \mu} \}$ as follows.
First, for $(m, \mu) = (-\tilde{m}_i, \pm \tilde{\mu}_i)$ with $i \in \{1, \ldots, s\}$
we let $f_{-\tilde{m}_i, \pm \tilde{\mu}_i} = F_i$.
Then, for $(m, \mu) \in \Q_{>0} \times N'/N$ with $m \equiv Q(\mu) \bmod{\Z}$
and $(m, \mu) \neq (m_i, \pm \mu_i)$ for all $i$, we let $f_{m,\mu} \in M_{2-k,\bar{\rho}_N}^!(\Q)$
be the unique weakly holomorphic modular form such that
\begin{align} \label{eq:fm}
	f_{m, \mu}(\tau) &= \frac12 q^{-m} (\phi_\mu + \phi_{-\mu}) - \frac12 \sum_{i=1}^r c_{G_i}(m,\mu) q^{-m_i}(\phi_{\mu_i} + \phi_{-\mu_i})\\ &\phantom{=} + \sum_{\nu \in N'/N} \sum_{n\geq 0} a_{m, \mu}(n, \nu) q^n (\phi_\mu + \phi_{-\mu}) \nonumber
\end{align}
with
\begin{equation}\label{eq:vanishingcoeffs1}
	a_{m,\mu}(\tilde{m}_1, \pm\tilde{\mu}_1) = \ldots = a_{m,\mu}(\tilde{m}_s, \pm\tilde{\mu}_s) = 0.
\end{equation}
It is clear that the forms $f_{-\tilde{m}_1, \tilde{\mu}_1}, \ldots, f_{-\tilde{m}_s, \tilde{\mu}_s}$ together with
\[
  \{f_{m, \mu}\ \mid m \in \Q_{>0}, \mu \in (N'/N)/\{\pm 1\}, m \equiv Q(\mu) \bmod{\Z} \}
\]
form a basis of $M_{2-k, \bar\rho_N}^!(\Q)$.

For $M_{k, \rho_N}^!(\Q)$, we define a basis $\{ g_{m, \mu} \}$ in the same way:
If $(m, \mu) = (-m_i, \pm \mu_i)$ then $i \in \{1, \ldots, r\}$
we let $g_{-m_i, \pm \mu_i} = G_i$.
Then, for $m \in \Q_{>0}$ and $\mu \in N'/N$ with $m \equiv -Q(\mu) \bmod{\Z}$,
we let $g_{m, \mu}$ be the unique weakly holomorphic modular form in $M_{k, \rho_N}^!(\Q)$
with
\begin{align}\label{eq:gm}
	g_{m, \mu}(\tau) &= \frac12 q^{-m} (\phi_\mu + \phi_{-\mu}) - \frac12 \sum_{i=1}^s c_{F_i}(m,\mu) q^{-\tilde{m}_i}(\phi_{\tilde{\mu}_i} + \phi_{-\tilde{\mu}_i})\\
	\phantom{=} &+ \sum_{\nu \in N'/N}\sum_{n \geq 0} b_{m, \mu}(n, \nu) q^n (\phi_\mu + \phi_{-\mu}) \nonumber
\end{align}
satisfying
\begin{equation}\label{eq:vanishingcoeffs2}
	b_{m,\mu}(m_1, \pm\mu_1) = \ldots = b_{m,\mu}(m_r, \pm\mu_r) = 0.
\end{equation}
We obtain a basis of $M_{k, \rho_N}^!(\Q)$ consisting of $g_{-m_1, \mu_1}, \ldots, g_{-m_r, \mu_r}$
and
\[
	\{g_{m, \mu}\ \mid m \in \Q_{>0}, \mu \in (N'/N)/\{\pm 1\}, m \equiv -Q(\mu) \bmod{\Z} \}.
\]

\begin{lemma}
  \label{lem:whforms}
  The conditions above characterize the forms $f_{m, \mu} \in M_{2-k,\bar\rho_N}^!(\Q)$
  and $g_{m, \mu} \in M_{k,\rho_N}^!(\Q)$ uniquely.
  They satisfy the duality relation
  \begin{equation}\label{eq:duality}
  	a_{m, \mu}(n, \nu) = - b_{n, \nu}(m, \mu)
  \end{equation}
  for all $m, n \in \Q$ and $\mu, \nu \in N'/N$ with $m \equiv Q(\mu) \bmod{\Z}$
  and $n \equiv -Q(\nu) \bmod{\Z}$.
\end{lemma}
\begin{proof}
  Existence follows from the second exact sequence in Corollary 3.8 of \cite{BF}
  and uniqueness is clear.
  Using \eqref{eq:fm}--\eqref{eq:vanishingcoeffs2}, it is easy to see that
  \[
  	\CT(\langle f_{m, \mu}, g_{n, \nu} \rangle) = a_{m, \mu}(n, \nu) + b_{n, \nu}(m, \mu).
  \]
  Note that $\langle f_{m, \mu}, g_{n, \nu} \rangle \in M_2^!$,
  which implies that its constant term vanishes.
\end{proof}

The relation \eqref{eq:duality} and the fact that
weakly holomorphic modular forms with rational Fourier coefficients
have bounded denominators implies the following lemma.
\begin{lemma}\label{lem:bound1}
  For every $n_0 \in \Q$ there is an $A \in \Z_{>0}$, only depending on $N$, $k$, and $n_0$,
  such that for all $n \leq n_0$, and all $\nu, m, \mu$, we have $A \cdot a_{m, \mu}(n, \nu) \in \Z$ and $A \cdot b_{n, \nu}(m, \mu) \in \Z$.
\end{lemma}
In the following, we assume that $N$ has signature $(0,q)$
with $q$ even.
Then $\theta_N(\tau, h) = v^{q/2}\overline{\theta_{N(-1)}(\tau, h)}$
has weight $-q/2$
and the theta function $\theta_{N(-1)}(\tau, h) \in M_{q/2, \bar\rho_N}$
is a holomorphic modular form.
We dropped the variable $z$
since the space $U = N \otimes \Q$ is definite and the symmetric domain only consists of the two points
$z_U^\pm$, yielding the same function.
Let $k:= 2 - q/2$.
We will then simply write $\Phi_N(h, f)$ for
the regularized theta lift of $f \in M^!_{2-k, \bar\rho_N}$
against $\theta_N$.
We will assume that $(\tilde{m}_1, \tilde{\mu}_1) = (0,0)$ for convenience, which we can do
because $\theta_{N(-1)} \in M_{2-k,\bar{\rho}_N}$ has a non-vanishing
constant term of index $(0,0)$.
In particular, $a_{m, \mu}(0,0) = 0$ for $(m, \mu) \neq (0,0)$ by \eqref{eq:vanishingcoeffs1}.

Write $T = \GSpin(U)$ and let $K_T \subset T(\A_f)$ be a suitable compact open such that $h \mapsto \theta_N(\tau, h)$
defines a function on $Z(U) = T(\Q) \bs T(\A_f) / K_T$.
\begin{lemma}\label{lem:GNexist}
  For every $h \in Z(U)$, there is a unique harmonic Maass form
  $\calG_N(\tau,h) \in H_{k,\rho_N}^!$
  with $L_{k}(\calG_N(\tau, h)) = \theta_N(\tau,h)$
  and holomorphic part
  \[
      \calG_N^+(\tau, h) = \sum_{\mu \in N'/N}\sum_{m \gg -\infty} c_N^+(h, m, \mu) q^m \phi_\mu
  \]
  satisfying the following properties:
  \begin{enumerate}
  \item For $m \leq 0$, we have $c_N^+(h, m, \mu) = 0$
  		unless $(m,\mu) = (-\tilde{m}_i, \pm\tilde{\mu}_i)$ for
    	some $i \in \{1,\ldots, s\}$.
  \item We have $c_N^+(h, m_i, \pm \mu_i) = 0$ for $i = 1,\ldots,r$.
  \item For every $m \in \Q$ and $\mu \in N'/N$ with $Q(\mu) \equiv m \bmod{\Z}$, we have
        \[
            \Phi_N(h, f_{m,\mu}) = c_N^+(h, m, \mu).
        \]
  \end{enumerate}
\end{lemma}
\begin{proof}
	Arguing as in Proposition 2.12 of \cite{EhlenSankaran}, there is a $\tilde{G} \in H_{k}^!$
	with $L_k(\tilde{G}(\tau)) = \theta_N(\tau,h)$ satisfying (1).
    To ensure that $\tilde{G}$ satisfies (2), we
  	can subtract suitable multiples of the
  	$G_i$ from $\tilde{G}$ without changing the image under
  	the lowering operator.

  Now let $\calG_N(\tau,h)$ be a harmonic Maass form
  with $L_{k}(\calG_N(\tau, h)) = \theta_N(\tau,h)$
  satisfying (1) and (2). Note that these conditions uniquely characterize $\calG_N(\tau, h)$.
  By Theorem \ref{thm:fund2} we have for $m = -\tilde{m}_i$ and $\mu = \tilde{\mu}_i$ that
  \begin{equation}
  	\Phi_N(h, f_{-\tilde{m}_i, \tilde{\mu}_i}) = \CT\left(\langle f_{-\tilde{m}_i, \tilde{\mu}_i},\, \calG_N(\tau,h) \rangle\right) = c_N^+(h, -\tilde{m}_i, \tilde{\mu}_i).
  \end{equation}
  Here, we have used $c_N^+(h, -\tilde{m}_i, \tilde{\mu}_i) = c_N^+(h, -\tilde{m}_i, -\tilde{\mu}_i)$
  and $a_{-\tilde{m}_i, \tilde{\mu}_i}(\tilde{m}_i, \pm \tilde{\mu}_i) = 1/2$.
  Similarly, for $m>0$ and $\mu \in N'/N$ with $(m, \mu) \neq (m_i, \pm \mu_i)$, we get
  \begin{align*}
    \Phi_N(h, f_{m, \mu})    &=  \CT\left(\langle f_{m,\mu},\, \calG_N(\tau,h) \rangle\right) \\
                             &=  c_N^+(h, m, \mu)
                             -  \sum_{i=1}^r c_N^+(h, m_i, \mu_i)c_{G_i}(m,\mu)
                             + \sum_{\substack{n\leq 0 \\\nu \in N'/N}} c_N^+(h, n, \nu)a_{m, \mu}(-n,\nu)  \\
                             &=  c_N^+(h, m, \mu)
								-  \sum_{i=1}^r c_N^+(h, m_i,\mu_i)c_{G_i}(m,\mu)
                                + 2\sum_{i=1}^s c_N^+(h, -\tilde{m}_i,\tilde{\mu}_i) a_{m, \mu}(\tilde{m}_i,\tilde{\mu}_i),
  \end{align*}
  where we have used (1) in the second line.
  By condition (2), the first sum on the right-hand side vanishes.
  Finally, by \eqref{eq:vanishingcoeffs1},
  the second sum on the right-hand side vanishes as well
  and this finishes the proof.
\end{proof}

\subsection{Special preimages of binary theta functions}
\label{sec:see-saw-identity}

In this section we restrict to the case $q=2$, i.e., $N$ has signature $(0,2)$ and $k=1$.
Let $U = N \otimes \Q$ be the corresponding rational quadratic space
and write $\theta_N(\tau,h)$ for the Siegel theta function
attached to $N$. We put $T := \GSpin(U)$.
As in Section \ref{sect:bintheta},
we let $D$ be the discriminant of $N$, and write
$\calO_D \subset k_{D} = \Q(\sqrt{D}) \cong U$
for the order of discriminant $D$ in $\k_{D}$.
For convenience of the reader, we recall the statement of Theorem \ref{thm:preimages},
which we will now prove.
\begin{theorem}\label{thm:preimages-appendix}
  For every $h \in T(\A_f)$, there is a harmonic Maass form
  $\calG_N(\tau,h) \in H_{1,\rho_N}^!$, only depending on the class of $h$ in $\Cl(\calO_D)$,
  with holomorphic part
  \[
    \calG_N^+(\tau,h) = \sum_{\mu \in N'/N} \sum_{m \gg -\infty} c_{N}^{+}(h,m,\mu) e(m \tau)  \phi_{\mu}
  \]
  satisfying the following properties:
  \begin{enumerate}
  \item We have $L_1(\calG_N(\tau,h)) = \theta_N(\tau,h)$.
  \item Let $\mu \in L'/L$ and $m \in \Q$
    with $m \equiv Q(\mu) \bmod{\Z}$ and $(m, \mu) \neq (0,0)$. There is an
    algebraic number $\alpha(h, m, \mu) \in H_D^\times$ such that
    \begin{equation}
      \label{eq:cprma2}
      c_{N}^{+}(h, m, \mu) = -\frac{1}{r}\log|\alpha_N(h, m, \mu)|,
    \end{equation}
    for some $r \in \Z_{>0}$ only depending on $N$.
    \item For all $h \in T(\A_f)$, we have
    \begin{equation}
    	\alpha_N(h, m, \mu) = \alpha_N(1,m,\mu)^{[h,\, \k_D]}.
    \end{equation}
    \item Additionally, there is an $\alpha_N(h, 0, 0) \in H_D^\times$, such that
  \[
      c_{N}^{+}(h, 0, 0) = \frac{2}{r} \log|\alpha_N(h, 0, 0)| + \kappa(0,0).
  \]
  \end{enumerate}
\end{theorem}
For the proof, we consider the lattice $L := P \oplus N$ of signature $(1,2)$,
where $P = \Z$ with the quadratic form $x^2$.
We put $V = L \otimes \Q$ and let $\calD$ be the asociated symmetric domain.
We let $H = \GSpin(V)$ and $K = \GSpin(\hat{L})$, so that
the theta lift $\Phi_L(z, h, f)$ of any $f \in M_{1/2, \bar\rho_L}^!$
defines a meromorphic modular form on $X_K$.
We view $Z(U)$ as a CM cylce on $X_K$ as in Section \ref{sec:2}.
For $m \in \Q$ and $\mu \in N'/N$ with $m \equiv Q(\mu) \bmod{\Z}$,
we let $f_{m, \mu} \in M_{1, \bar\rho_N}^!(\Q)$ be as in the previous section.

\begin{remark}
  We remark that all of the following arguments can easily be adopted
  to work with any lattice $L$ of signature $(1,2)$
  such that we have a primitive isometry $N \hookrightarrow L$.
  In \cite{Eh} we used the lattice for $\Gamma_0(|D|)$
  for odd squarefree $D$ to obtain more precise information
  about the algebraic numbers appearing in Theorem \ref{thm:preimages-appendix}
  (for instance integrality and the prime factorization),
  and for computational purposes it can also be useful to tweak the choice of $L$.
  For the purposes of proving the statements of Theorem \ref{thm:preimages-appendix},
  however, our simple choice suffices.
\end{remark}

\begin{proposition}\label{prop:seesaw}
  Let $m \in \Q$ and $\mu \in N'/N$
  such that $Q(\mu) \equiv m \bmod{\Z}$.
  There is a weakly holomorphic modular form
  $\calF_{m,\mu} \in M_{1/2, \bar{\rho}_L}^!(\Q)$
  such that $c_{\calF_{m,\mu}}(0,0) = 0$ and satisfying
  \[
    12 \Phi_N(h, f_{m,\mu}) = \Phi_L(z_U^{\pm}, h, \calF_{m,\mu}).
  \]
  Moreover, for every $n_0 \in \Q$ there is a constant $B \in \Z_{>0}$
  such that for all $n \leq n_0$ and all $m, \mu$, we have
  $B \cdot c_{\calF_{m, \mu}}(n, \nu) \in \Z$.
\end{proposition}
\begin{proof}
  We follow the argument given in Theorem 6.6 of \cite{Vi3} and Section 4.2 of \cite{Eh}
  (here the special case for $A=1$ in \cite{Eh} is sufficient).
  For any integer $k$, the space $M_{k-1/2,\bar\rho_P}^!$ is isomorphic to the space of
  $J_{k,1}^!$ of weakly holomorphic Jacobi forms of weight $k$ and index $1$
  via the theta expansion of Jacobi forms \cite{EZ}.

  We let
  $\tilde\phi_{-2,1} \in J_{-2,1}^! \cong M_{-5/2, \bar\rho_P}^!$
  and $\tilde\phi_{0,1} \in J_{0,1}^! \cong M_{-1/2, \bar\rho_P}^!$
  be the two generators
  of the ring of weak Jacobi forms of even weight
  over $M_{\ast}$, the ring of holomorphic modular forms for $\SL_2(\Z)$
  as in \cite[Theorem 9.3]{EZ}.
  These two forms correspond to vector valued weakly holomorphic modular forms
  $\psi_{-2,1} \in M_{-5/2, \bar\rho_P}^!$ and $\psi_{0,1} \in M_{-1/2, \bar\rho_P}^!$.
  For any weak Jacobi form $\phi(\tau, z) \in J_{k,n}^{\mathrm{weak}}$,
  the specialization $\phi(\tau,0)$ is a holomorphic modular form of weight $k$
  for $\SL_2(\Z)$.
  Hence,
  \[
    \langle \psi_{-2,1}, \theta_P \rangle = \tilde\phi_{-2,1}(\tau, 0) = 0
  \]
  and $\langle \psi_{0,1}, \theta_P \rangle = \tilde\phi_{0,1}(\tau, 0)$ is a constant.
  By inspection of the Fourier expansion of $\tilde\phi_{0,1}$ it is easily seen that
  \[
  	\langle \psi_{0,1}, \theta_P \rangle = \tilde\phi_{0,1}(\tau, 0) = 12.
  \]
  Using the identification $\bar\rho_{L} \cong \bar\rho_P \otimes \bar\rho_N$,
  we view $\psi_{0,1} \otimes f_{m,\mu}$ as an element of
  $M_{1/2, \bar\rho_L}^!$.
  Thus, we obtain the relation of theta lifts
  \[
     12 \Phi_N(h, f_{m,\mu}) = \Phi_L(z_U^{\pm}, h, \psi_{0,1} \otimes f_{m,\mu}).
  \]
  The constant term of index $(0,0)$ of $\psi_{0,1} \otimes f_{m,\mu}$ might be non-zero.
  In that case, let $a \in \Z_{>0}$ be minimal such that
  the Fourier coefficient of index $(a,0)$
  of $\psi_{-2,1} \otimes \theta_{N(-1)}$
  is non-zero and let $g \in M_{2}^!$
  be the unique weakly holomorphic modular form of weight $2$
  with principal part equal to $q^{-a}$.
  Note that the constant term of $g$ necessarily vanishes
  and that the
  $\phi_0$-component of $\psi_{-2,1} \otimes \theta_{N(-1)}$ does not have any non-zero
  Fourier coefficients of negative index.
  Hence, $g \psi_{-2,1} \otimes \theta_{N(-1)}$
  has a non-zero and integral constant term of index $(0,0)$.
  Using $\langle \psi_{-2,1}, \theta_P \rangle = 0$, we obtain
  \[
    \Phi_L(z_U^{\pm}, h, g \psi_{-2,1} \otimes \theta_{N(-1)}) = 0,
  \]
  and thus we can define $\calF_{m,\mu} = \psi_{0,1} \otimes f_{m,\mu} - x \cdot g \psi_{-2,1} \otimes \theta_{N(-1)}$
  with a suitable constant $x \in \Q^\times$.

  Finally, to obtain the bound, note that
  $\psi_{0,1}$ has integral Fourier coefficients
  and a principal part equal to $q^{-1/4} \phi_{1/2 + \Z}$.
  Let $A$ be the bound in Lemma \ref{lem:bound1}
  such that $A a_{m, \mu}(n, \nu) \in \Z$
  for all $n \leq n_0 + 1/4$.
  Then $A \psi_{0,1} \otimes f_{m, \mu}$ has
  integral Fourier coefficients, up to $q^{n_0}$.

  The Fourier coefficient $c_0$ of index $(0,0)$ of
  $g \psi_{-2,1} \otimes \theta_{N(-1)}$
  is an integer since $\theta_{N(-1)}$, $g$,
  and $\psi_{-2,1}$ have integral Fourier coefficients.
  Thus, $c_0 x \in \Z$.
  Consequently, the denominators of the Fourier coefficients of $\calF_{m, \mu}$,
  up to $q^{n_0}$, are bounded by $B = \mathrm{lcm}(A, c_0)$.
\end{proof}

According to \cite[Theorem 13.3]{Bo1},
there is a meromorphic modular form $\Psi_{L}(z, h, \calF_{m, \mu})$ of weight $0$
(and some multiplier system of finite order), such that
\begin{equation}\label{eq:BP}
  	\Phi_L(z, h, \calF_{m, \mu}) = -4 \log |\Psi_{L}(z, h, \calF_{m, \mu})|,
  \end{equation}
with $\mathrm{div}(\Psi_{L}(z, h, \calF_{m, \mu})) = Z(\calF_{m, \mu}) = \sum_{\nu \in N'/N} \sum_{n<0} c_{\calF_{m, \mu}}(n, \nu) Z(n, \nu)$.
The identity \eqref{eq:BP} holds on the complement of
  \begin{equation}\label{eq:union}
    \bigcup_{\substack{n < 0 \\ \nu \in L'/L \\  c_{\calF_{m, \mu}}(n,\nu) \neq 0}} Z(n,\nu).
  \end{equation}

\begin{corollary}\label{cor:ratfun}
There is a constant $A_0 \in \Z_{>0}$, only depending on $N$, such that
for all $m$ and $\mu$, the Borcherds product $\Psi_{L}(z, h, A_0 \calF_{m, \mu})$
defines a meromorphic function on $X_K$ which is defined over $\Q$.
\end{corollary}
\begin{proof}
  We use $n_0 = 1$ in Proposition \ref{prop:seesaw}
  to obtain a bound $B$ on the denominator of the Fourier coefficients
  of $\calF_{m, \mu}$, up to $q^1$.
  By \cite[Theorem A]{HowardMadapusi},
  there is a constant $A_0$ with $B \mid A_0$
  such that $\Psi_L(z, h, A_0 \calF_{m,\mu})$ is defined over $\Q$.
  An inspection of the proof of \cite[Theorem A]{HowardMadapusi} shows that
  $A_0$ can be chosen indepently of $m$ and $\mu$.
\end{proof}

\begin{proof}[Proof of Theorem \ref{thm:preimages-appendix}]
  We let $\calG_N$ be defined as in Lemma \ref{lem:GNexist}.
  Then (1) is clear. To prove (2), let $(m, \mu) \neq (0,0)$.
  We use that $\Phi_L(z_U^{\pm}, h, \calF_{m,\mu})$
  is the logarithm of a CM value of a Borcherds product
  and then invoke CM theory and Shimura reciprocity.

  It is not hard to see (cf. Proposition 4.18 of \cite{Eh})
  that $c_{f_{m, \mu}}(0,0) = 0$ implies that the Borcherds product
  $\Psi_{L}(z, h, \calF_{m,\mu})$ is always defined and non-zero
  at the CM point $(z_U^\pm, h)$,
  even if $(z_U^\pm, h)$ is contained in one of the divisors $Z(n, \nu)$
  appearing in \eqref{eq:union}.
  Using Corollary \ref{cor:sing}, it is straightforward to check that \eqref{eq:BP}
  then still holds up to the logarithm of a non-zero rational number.
  Hence,
  \[
    \Phi_L(z_U^{\pm}, h, \calF_{m,\mu}) = -4 \log |\Psi_{L}(z_U^\pm, h, \calF_{m,\mu})| - \log|t|,
  \]
  where $t \in \Q^\times$ is equal to $1$ if $(z_U^\pm,h)$ is not contained in \eqref{eq:union}.

  We let $A_0$ be the constant in Corollary \ref{cor:ratfun}.
  Then $\Psi_{L}(z_U^\pm, h, A_0 \calF_{m,\mu})$ is defined over $\Q$ and
  hence, we infer that the algebraic number
  \[
    \alpha_N(h, m, \mu) := t^{A_0} \cdot \Psi_L(z_U^\pm, h, A_0 \calF_{m,\mu})^{4}
  \]
  is contained in the ring class field $H_D$,
  since $(z_U^\pm, h)$ is defined over $H_D$.
  The relation in (2) now follows by Proposition \ref{prop:seesaw} and Lemma \ref{lem:GNexist}.
  Item (3) then follows from Shimura reciprocity \cite[Theorem 6.31]{shimauto}, i.e.,
  \[
    \alpha_N(h, m, \mu) = t^{A_0} \cdot \Psi_{L}(z_U^\pm, h, A_0 \calF_{m,\mu})^{4}
  = t^{A_0} \cdot (\Psi_{L}(z_U^\pm, 1, A_0 \calF_{m,\mu})^{4})^{[h, k_D]} = \alpha_N(1, m, \mu)^{[h, k_D]}.
  \]
  Finally, note that $\xi_1(\calG_N(\tau) - E_N'(\tau, 0, 1))$ is a cusp form.
  By considering the pairing with the holorphic Eisenstein series $\frac{1}{2} v\overline{E_N(\tau, 0 -1)}$, (4) follows from \cite[Proposition 3.5]{BF}.
\end{proof}

\printbibliography

\end{document}